\documentclass[12pt]{article}
\usepackage{amssymb,amsmath,amsthm,tikz,multirow,nccrules}
\usetikzlibrary{calc, arrows, arrows.meta, decorations.markings}

\title{Side-to-side Tiling of the Sphere by Congruent Curvilinear Triangles}
\author{Keyi Jin, East China Normal University \\
	Linming Lu, Chinese University of Hong Kong \\
	Erxiao Wang\thanks{Corresponding author (wang.eric@zjnu.edu.cn).  Research was supported by National Natural Science Foundation of China NSFC-RGC 12361161603 and Key Projects of Zhejiang Natural Science Foundation LZ22A010003.}, Zhejiang Normal University \\
	Lijuan Wu, Jinhua No.4 Middle School \\
	Min Yan\thanks{Research was supported by NSFC-RGC Joint Research Scheme N-HKUST607/23 and Hong Kong RGC General Research Fund 16310925.}, 
	Hong Kong University of Science and Technology}

\usepackage[hidelinks, hyperindex]{hyperref}

\newcommand{\mc}{\mathcal}

\newcommand*\circled[1]{\tikz[baseline=(char.base)]{
		\node[shape=circle,draw,inner sep=0.5pt] (char) {#1};}}
		
\newtheorem{theorem}{Theorem}
\newtheorem{lemma}[theorem]{Lemma}

\newtheorem{proposition}[theorem]{Proposition}
\newtheorem*{theorem*}{Theorem}

\theoremstyle{definition}
\newtheorem*{definition*}{Definition}
\newtheorem*{case*}{Case}
\newtheorem*{subcase*}{Subcase}

\theoremstyle{remark}

\numberwithin{equation}{section}

\begin{document}
\date{}
\maketitle

\begin{abstract}
The edge-to-edge tilings of the sphere by congruent polygons, where all edges are straight, have been completely classified. We classify the curvilinear version of the similar triangular tilings, where the edges may not be straight, and find that these are the modifications of the straight triangular tilings.
	
	
	{\it 2020 MR Subject Classification}:	52C20, 05B45.
\end{abstract}

\section{Introduction}

Isometry or congruence is defined for objects on surfaces of constant Gaussian curvature, and we may consider tilings by congruent polygons on such surfaces. Here a polygon has corners and sides, and the interior is homeomorphic to an open disk. However, we do not assume the sides are straight (i.e., geodesic). We call such a polygon {\em curvilinear}. 

In this paper, we classify {\em side-to-side} tilings of the sphere by congruent curvilinear triangles. Figure \ref{bird} is an example of such tiling. In the literature, side-to-side tilings are usually called {\em edge-to-edge} tilings. For tilings by congruent curvilinear polygons, however, we need to distinguish edges and sides. Since the concept of sides characterises the curvilinear property, side-to-side is more precise than edge-to-edge. Moreover, to simplify the discussion, we also assume that all vertices in a tiling have degree $\ge 3$.

\begin{figure}[h]
	\centering
	\includegraphics[width=0.3\textwidth]{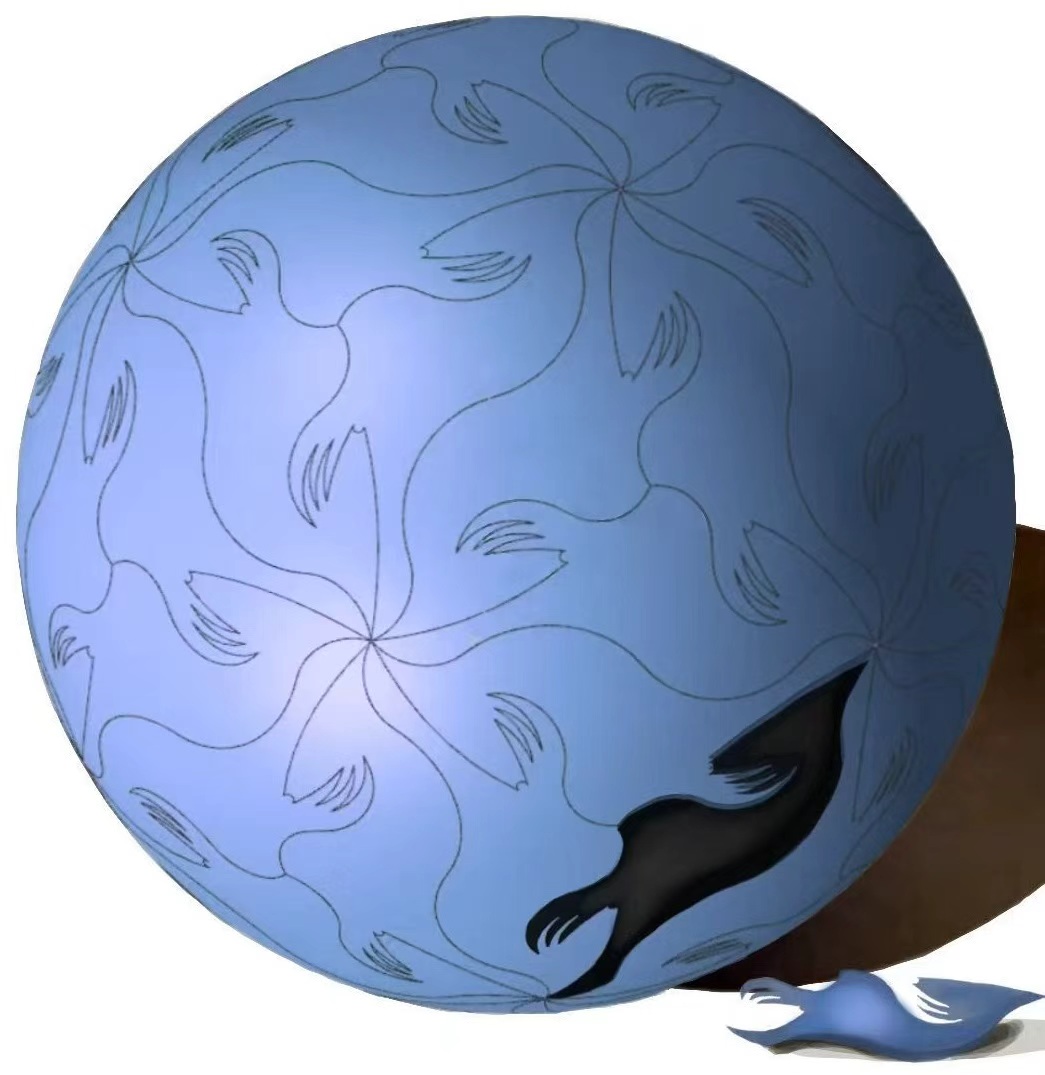}
\caption{A tiling of the sphere by a bird-like curvilinear triangle.}
\label{bird}
\end{figure}

The classification of side-to-side tilings of the sphere by congruent straight (i.e., non-curvilinear) triangles was started a century ago by Sommerville \cite{so}, further developed by Davies \cite{dav} in 1967, and completed by Ueno and Agaoka \cite{ua} in 2002. Cheung, Luk and Yan gave a modern treatment of the classification in \cite{cly}. In fact, side-to-side tilings of the sphere by congruent straight polygons (i.e., including quadrilateral and pentagonal tilings) have been completely classified \cite{wy1, wy2, awy, lw1, lw2, lw3, lwy, cly, cly5, slw,llhw}. We have also made initial progress on the classification of non-side-to-side tilings by congruent straight polygons \cite{clw}. We would also like to mention the recently completed classification of tilings of sphere by (straight) regular polygons \cite{johnson, zal, aehj}. 

While tremendous progress has been made in the research on tilings by straight polygons, curvilinear tilings is also an exciting and worthwhile research direction. Many paintings by the popular Dutch artist Escher are basically curvilinear tilings (see \cite{gs, a1}). Many tilings in nature and engineering are curvilinear. One such example is the tilings book by Heesch and Kienzle
\cite{hk} for engineers. An interesting mathematical question is whether there are curvilinear tilings that are essentially different from straight tilings. In this paper, we classify the simplest case of curvilinear triangular tilings of the sphere and find no new tilings.

\begin{theorem*} 
Side-to-side tilings of the sphere by congruent curvilinear triangles, such that all vertices have degree $\ge 3$, are modifications of side-to-side tilings of the sphere by congruent straight triangles.
\end{theorem*}

The details of the theorem are given in the sequence of propositions in Section \ref{tiling}. There are sixteen curvilinear triangles suitable for tiling. They are listed in Figure \ref{triangle}, from most curvilinear to almost straight. The more curvilinear the triangle is, the more rigid the tilings are. This is illustrated by the earlier propositions in Section \ref{tiling}, where the tilings are always the Platonic type. The less curvilinear the triangle is, the more flexible the tilings are. This is illustrated by the later propositions, where the tilings are closer to tilings by straight triangles, including earth map tilings and their modifications. 

In \cite{ua, cly}, the side-to-side tilings of the sphere by congruent straight triangles are classified as follows: Platonic solids $P_n$ with triangular faces ($n=4,8,20$), triangular subdivisions $T_{\triangle}P_n$ and barycentric subdivisions $B_{\triangle}P_n$ of Platonic solids ($n=4,6,8,12,20$), simple triangular subdivisions $S_{\triangle}P_6$ of the cube, three families of earth map tilings $E_{\triangle}n$ ($n=1,2,3$, and including special cases $E_{\triangle}^I1$ and $E_{\triangle}^J1$), and flip modifications $FB_{\triangle}P_8$, $FE_{\triangle}1$, $F'E_{\triangle}1$, $FE_{\triangle}2$, $FE_{\triangle}3$. See Section 2 of \cite{cly} for further detailed explanations of these constructions. It turns out there are no curvilinear versions of barycentric subdivisions $B_{\triangle}P_n$, the third earth map tiling $E_{\triangle}3$, and its flip modification, and the special flip modification $F'E_{\triangle}1$. 

For the remaining tilings, we summerise the triangular types (in Figure \ref{triangle}) in Table \ref{alltiling}. We note that some flip modifications $FE_{\triangle}$ become rotation modifications $RE_{\triangle}$ in the table. The reason is that the flip modifications for straight triangular tilings can also be interpreted as rotation modifications. Due to the rigidity of curvilinear edges, sometimes only the flip or the rotation interpretation can be applied to curvilinear tilings. 

\renewcommand{\arraystretch}{1.2}

\begin{table}[htp]
\centering
\begin{tabular}{|c|c|c|} 
\hline
tiling
& triangle \\
\hline  \hline 
\multirow{2}{*}{$P_4$}
& $g\bar{g}r$, $g\bar{g}a$, $rr^{-1}{\color{red} r'}$,  $rr^{-1}a$,  $r{\color{red} r'}{\color{blue} r''}$, $r{\color{red} r'}a$, \\
\cline{2-2}
& $ra{\color{red} a'}$, $h\bar{h}r$, $h\bar{h}a$, $rr{\color{red} r'}$, $rra$, $raa$  \\
\hline
$P_n$, $n=4,8,20$
& $rrr$, $rrr^{-1}$ \\
\hline
\multirow{2}{*}{$T_{\triangle}P_n$, $n=4,6,8,12,20$}
& $g\bar{g}^{-1}r$, $g\bar{g}^{-1}a$,   \\
\cline{2-2}
& $h\bar{h}r$, $h\bar{h}a$, $rr{\color{red} r'}$, $rra$, $raa$  \\
\hline
$S_{\triangle}P_6$ 
& $h\bar{h}r$, $h\bar{h}a$, $rr{\color{red} r'}$, $rra$, $raa$ \\
\hline
$E_{\triangle}1$ 
&  $ra{\color{red} a'}$ \\
\hline
$FE_{\triangle}1$
& $ra{\color{red} a'}$ \\
\hline
$E_{\triangle}^I1$, $E_{\triangle}^J1$, $E_{\triangle}2$, $RE_{\triangle}^J1$, $RE_{\triangle}2$
&  $h\bar{h}r$, $h\bar{h}a$, $rr{\color{red} r'}$, $rra$, $raa$ \\
\hline
$FE_{\triangle}^I1$
& $h\bar{h}r$, $h\bar{h}a$ \\
\hline
$RE_{\triangle}^I1$
& $rr{\color{red} r'}$, $rra$, $raa$ \\
\hline
\end{tabular}
\caption{Tilings of the sphere by congruent curvilinear triangles.}
\label{alltiling}
\end{table}

\section{Curvilinear Triangle}

Consider a side-to-side tiling of a surface, such that all tiles are congruent to a curvilinear polygon $P$, called {\em prototile}. The prototile has {\em sides} and {\em corners}. They are usually called {\em edges} and {\em vertices}, but actually carry extra meaning. A side $s$ of $P$ is not just a curve in the boundary of $P$, but also includes the side of the curve that is inside $P$. The particular side of the curve is illustrated by the shaded region on the left of Figure \ref{general1}. A corner is the meeting place of two compatible sides. 

\begin{figure}[htp]
\centering
\begin{tikzpicture}[>=latex,scale=1]


\fill[gray!50]
	(0,0) to[out=80, in=-20] (0,1.4) -- (0.2,1.4) to[out=-30, in=80] (0.2,0);

\draw
	(0,0) to[out=80, in=-20] (0,1.4);

\draw[dashed]
	(0,1.4) -- (1.2,1.4) arc (90:-90:0.5 and 0.7) -- (0,0);

\node at (0.35,0.7) {\small $s$};
\node at (0,0.7) {\small $\bar{s}$};
\node at (1,0.7) {\small $P$};


\begin{scope}[xshift=5cm]

\foreach \a in {0,1}
\draw[xshift=1.4*\a cm]
	(0,0) to[out=80, in=-20] (0,1.4);

\draw[dashed]
	(1.4,1.4) -- (-1.2,1.4) arc (90:270:0.5 and 0.7) -- (1.4,0);
	
\node at (0.35,0.7) {\small $s$};
\node at (0,0.7) {\small $\bar{s}$};
\node at (1.4,0.7) {\small $\bar{s}$};
\node at (0.8,0.3) {\small tile 1};
\node at (-0.7,0.3) {\small tile 2};

\end{scope}

\end{tikzpicture}
\caption{Curvilinear side.}
\label{general1}
\end{figure}
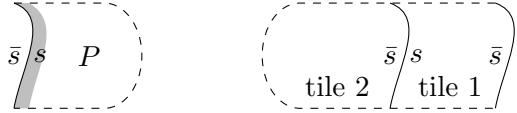

The curve part of a side $s$ an {\em edge} $e$. We denote the other side of $e$ by $\bar{s}$. In a {\em side-to-side} tiling, $s$ belongs to a tile, and $\bar{s}$ belong to an adjacent tile. Since both tiles are congruent to $P$, both $s$ and $\bar{s}$ are sides of $P$. See the right of Figure \ref{general1}. If $s$ and $\bar{s}$ are distinct, in the sense that there is no self isometry of the edge $e$ that sends one side to the other side, then $s$ and $\bar{s}$ must be different sides of $P$. If $s$ and $\bar{s}$ are the same, in the sense there is such isometry, then we write $s=\bar{s}$. In this case, the side $s$ can be glued to itself in the tiling.

There are three nontrivial isometries that can be applied to an edge $e$, such that the set of two end points is preserved: the horizontal flip $e^h$, the vertical flip $e^v$, and the rotation $e^r$, given by Figure \ref{transform}. Since the horizontal flip can be regarded as the reverse of direction, we also write $e^h$ as $e^{-1}$. 

\begin{figure}[htp]
\centering
\begin{tikzpicture}[>=latex,scale=1]

\foreach \a in {1,-1}
\foreach \b in {1,-1}
\draw[xscale=\a,yscale=\b]
	(-1.5,0) to[out=30, in=100] (1.5,0);

\draw[dotted]
	(-2.2,0) -- (2.2,0)
	(0,1.3) -- (0,-1.3);

\fill
	(1.5,0) circle (0.1)
	(-1.5,0) circle (0.1);

\node at (1,0.8) {$e$};
\node at (1,-0.8) {$e^v$};
\node at (-1,0.85) {$e^h$};
\node at (-1,-0.8) {$e^r$};

\node at (2.4,0) {$v$};
\node at (0,1.5) {$h$};

\end{tikzpicture}
\caption{A general curvilinear edge and its isometric transformations.}
\label{transform}
\end{figure}
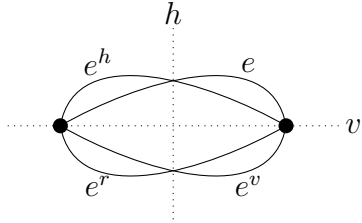

The isometry can also be applied to a side $s$ of the edge $e$, and $s\ne s^{-1}$ means $s\ne s^h$. There are four types of edges and correspondingly, four types of sides.
\begin{enumerate}
\item general: $e,e^v,e^h,e^r$ are all distinct. We have $s\ne\bar{s}$, and $s\ne s^{-1}$.
\item $h$-symmetric: $e=e^h\ne e^r$. We have $s\ne\bar{s}$, and $s=s^{-1}$.
\item $r$-symmetric: $e=e^r\ne e^h$. We have $s=\bar{s}$, and $s\ne s^{-1}$.
\item straight: $e=e^v$. This implies $e,e^v,e^h,e^r$ are the same. We have $s=\bar{s}$, and $s=s^{-1}$.
\end{enumerate}

Figure \ref{type} gives examples of the four types of edges, and the schematic ways of drawing them in the pictures. We denote the four types of sides by $g$ (general), $h$ ($h$-symmetric), $r$ ($r$-symmetric), $a$ (straight arc). We use $g^{-1},r^{-1}$ to indicate the reverse direction, and we do not use $h^{-1},a^{-1}$. We also use $g,\bar{g}$ and $h,\bar{h}$ to denote the other sides of $g$ and $h$, and we do not use $\bar{r},\bar{a}$.

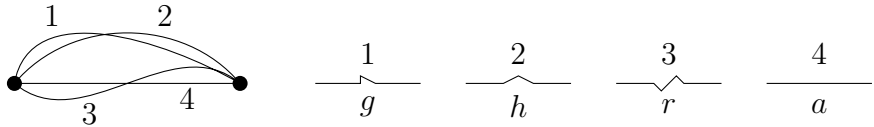
\begin{figure}[htp]
\centering
\begin{tikzpicture}[>=latex,scale=1]

\draw
	(-1.5,0) -- (1.5,0)
	(-1.5,0) to[out=80, in=150] (1.5,0)
	(-1.5,0) to[out=50, in=130] (1.5,0)
	(-1.5,0) to[out=-40, in=140] (1.5,0)
	;

\fill
	(1.5,0) circle (0.1)
	(-1.5,0) circle (0.1);

\node at (-1,0.9) {1};
\node at (0.5,0.9) {2};
\node at (-0.5,-0.4) {3};
\node at (0.8,-0.2) {4};


\begin{scope}[xshift=2.5cm]

\draw
	(0,0) -- (0.6,0) -- (0.6,0.1) -- (0.8,0) -- (1.4,0);

\draw[xshift=2cm]
	(0,0) -- (0.5,0) -- (0.7,0.1) -- (0.9,0) -- (1.4,0);

\draw[xshift=4cm]
	(0,0) -- (0.5,0) -- (0.6,-0.1) -- (0.8,0.1) -- (0.9,0) -- (1.4,0);
	
\draw
	(6,0) -- ++(1.4,0);

\node at (0.7,0.4) {1};
\node at (2.7,0.4) {2};
\node at (4.7,0.4) {3};
\node at (6.7,0.4) {4};

\node at (0.7,-0.3) {$g$};
\node at (2.7,-0.3) {$h$};
\node at (4.7,-0.3) {$r$};
\node at (6.7,-0.3) {$a$};

\end{scope}	

\end{tikzpicture}
\caption{Four types of edges and their schematic drawings.}
\label{type}
\end{figure}

In schematic drawings, it is hard to visually see the distinction between $r$ and $r^{-1}$. We add a circle to the picture, as in Figure \ref{r_opposite}, to emphasise the distinction. 

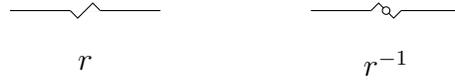
\begin{figure}[htp]
\centering
\begin{tikzpicture}[>=latex]


\begin{scope}
[decoration={markings,mark=at position 0.5 with {\draw[white,thick](0.2,0)--(-0.2,0);
\draw (0.2,0)--(0.1,0.1)--(-0.1,-0.1)--(-0.2,0);}}]

\draw[postaction={decorate}]
	(-1,0) -- (1,0);
											
\end{scope}
	
\node at (0,-0.7) {$r$};	


\begin{scope}[xshift=4cm]

\begin{scope}
[decoration={markings,mark=at position 0.5 with {\draw[white,thick](0.2,0)--(-0.2,0);
\draw (0.2,0)--(0.1,-0.1)--(-0.1,0.1)--(-0.2,0);
\filldraw[fill=white] (0,0) circle (0.05);}}]

\draw[postaction={decorate}]
	(-1,0) -- (1,0);
											
\end{scope}
	
\node at (0,-0.7) {$r^{-1}$};	

\end{scope}

\end{tikzpicture}
\caption{Add a circle to emphasise $r^{-1}$.}
\label{r_opposite}
\end{figure}

If a side $s$ is general or $h$-symmetric, then $s\ne \bar{s}$, and cannot be glued to itself. This implies $s$ and $\bar{s}$ appear in pairs as sides of the prototile $P$, and we conclude the following.

\begin{lemma}
In a side-to-side tiling by congruent curvilinear polygons, $g$ and $\bar{g}$ appear the same number of times in the prototile, and $h$ and $\bar{h}$ appear the same number of times in the prototile. 
\end{lemma}

As a consequence of the lemma, there are sixteen curvilinear triangles. They are listed in Figure \ref{triangle}, from the most curvilinear to almost straight.

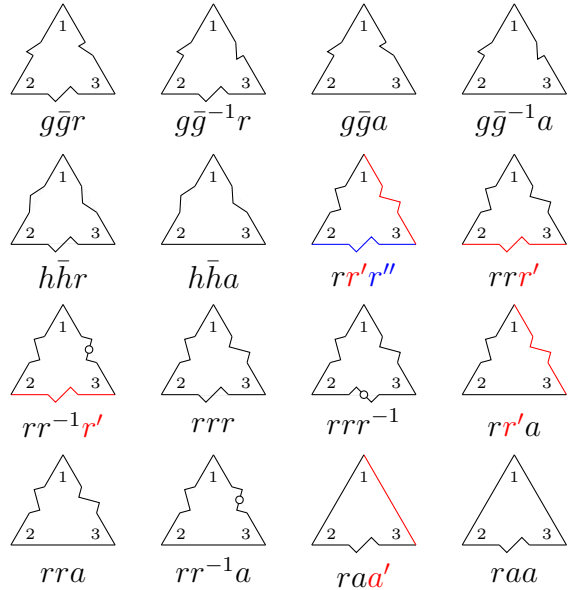
\begin{figure}[htp]
\centering
\begin{tikzpicture}[>=latex,scale=1]

\foreach \a in {0,1,2,3}
\foreach \b in {0,1,2,3}
{
\begin{scope}[xshift=2*\a cm, yshift=-2*\b cm]
	
\node at (90:0.5) {\tiny 1};
\node at (210:0.5) {\tiny 2};
\node at (-30:0.5) {\tiny 3};

\end{scope}
}
		

\begin{scope}
[decoration={markings,mark=at position 0.5 with {
\draw[white, thick] (0.1,0)--(-0.1,0);
\draw (0.1,0)--(-0.1,0.1)--(-0.1,0);
}}]

\foreach \a in {1,3}
\draw[postaction={decorate}, xshift=2*\a cm]
	(-30:0.8) -- (90:0.8);

\foreach \a in {0,1,2,3}
\draw[postaction={decorate}, xshift=2*\a cm]
	(210:0.8) -- (90:0.8);

\end{scope}	


\begin{scope}[
decoration={markings,mark=at position 0.5 with {
\draw[white, thick] (0.1,0)--(-0.1,0);
\draw (0.1,0)--(-0.1,-0.1)--(-0.1,0);}}]

\foreach \a in {0,2}
\draw[postaction={decorate}, xshift=2*\a cm] 
	(90:0.8) -- (-30:0.8);
 
\end{scope}


\begin{scope}
[decoration={markings,mark=at position 0.5 with {
\draw[white, thick] (0.2,0)--(-0.2,0);
\draw (0.2,0)--(0,0.1)--(-0.2,0);}}]

\foreach \a in {0,1}
{
\draw[postaction={decorate}, xshift=2*\a cm, yshift=-2 cm]
	(-30:0.8) -- (90:0.8);
\draw[postaction={decorate}, xshift=2*\a cm, yshift=-2 cm]
	(210:0.8) -- (90:0.8);
}

\end{scope}


\begin{scope}
[decoration={markings,mark=at position 0.5 with {
\draw[white, thick] (0.2,0)--(-0.2,0);
\draw (0.2,0)--(0.1,0.1)--(-0.1,-0.1)--(-0.2,0);}}]

\foreach \a/\b in {0/0,1/0,0/1,1/2,2/3,3/3}
\draw[postaction={decorate}, xshift=2*\a cm, yshift=-2*\b cm]
	(210:0.8) -- (-30:0.8);

\draw[postaction={decorate}, blue, xshift=4 cm, yshift=-2 cm]
	(210:0.8) -- (-30:0.8);
	
\foreach \a/\b in {3/1,0/2}
\draw[postaction={decorate}, red, xshift=2*\a cm, yshift=-2*\b cm]
	(210:0.8) -- (-30:0.8);
		
\foreach \a/\b in {1/2,3/1,2/2,0/3}	
\draw[postaction={decorate}, xshift=2*\a cm, yshift=-2*\b cm]
	(-30:0.8) -- (90:0.8);

\foreach \a/\b in {2/1}	
\draw[postaction={decorate}, red, xshift=2*\a cm, yshift=-2*\b cm]
	(-30:0.8) -- (90:0.8);
			
\foreach \a/\b in {2/1,3/1,0/2,1/2,2/2,3/2,0/3,1/3}	
\draw[postaction={decorate}, xshift=2*\a cm, yshift=-2*\b cm]
	(90:0.8) -- (210:0.8);
	
\draw[postaction={decorate}, red, xshift=6 cm, yshift=-4 cm]
	(-30:0.8) -- (90:0.8);
					
\end{scope}	


\begin{scope}
[decoration={markings,mark=at position 0.5 with {\draw[white,thick](0.2,0)--(-0.2,0);
\draw (0.2,0)--(0.1,-0.1)--(-0.1,0.1)--(-0.2,0);
\filldraw[fill=white] (0,0) circle (0.05);}}]

\foreach \a/\b in {0/2,1/3}	
\draw[postaction={decorate}, xshift=2*\a cm, yshift=-2*\b cm]
	(90:0.8) -- (-30:0.8);

\draw[postaction={decorate}, xshift=4 cm, yshift=-4 cm]
	(210:0.8) -- (-30:0.8);
	
\end{scope}		


\foreach \a/\b in {2/0,3/0,1/1,3/2,0/3,1/3}	
\draw[shift={(2*\a cm, -2*\b cm)}]
	(-30:0.8) -- (210:0.8);

\draw[shift={(4 cm, -6 cm)}, red]
	(90:0.8) -- (-30:0.8);

\draw[shift={(6 cm, -6 cm)}]
	(90:0.8) -- (-30:0.8);
	
\foreach \a/\b in {2/3,3/3}	
\draw[shift={(2*\a cm, -2*\b cm)}]
	(90:0.8) -- (210:0.8);
	
\node at (0,-0.8) {$g\bar{g}r$};
\node at (2,-0.75) {$g\bar{g}^{-1}r$};	
\node at (4,-0.8) {$g\bar{g}a$};	
\node at (6,-0.75) {$g\bar{g}^{-1}a$};		

\node at (0,-2.8) {$h\bar{h}r$};
\node at (2,-2.8) {$h\bar{h}a$};	
\node at (4,-2.8) {$r{\color{red} r'}{\color{blue} r''}$};
\node at (6,-2.8) {$rr{\color{red} r'}$};	

\node at (0,-4.8) {$rr^{-1}{\color{red} r'}$};
\node at (2,-4.8) {$rrr$};	
\node at (4,-4.75) {$rrr^{-1}$};
\node at (6,-4.8) {$r{\color{red} r'}a$};	

\node at (0,-6.8) {$rra$};
\node at (2,-6.73) {$rr^{-1}a$};	
\node at (4,-6.8) {$ra{\color{red} a'}$};
\node at (6,-6.8) {$raa$};

\end{tikzpicture}
\caption{Curvilinear triangles suitable for tiling.}
\label{triangle}
\end{figure}

We remark that, with the exception of $rrr$ and $rrr^{-1}$, the edge $\overline{23}$ (connecting corners $2$ and $3$) in Figure \ref{triangle} can be distinguished from the other edges. Then the tiles form {\em companion pairs} that share the edge $\overline{23}$. The union of the pair is a quadrilateral, and the triangular tiling can be regarded as a quadrilateral tiling. Conversely, the triangular tiling is obtained from the quadrilateral tiling by using diagonals to cut individual quadrilaterals into companion pairs of triangles.

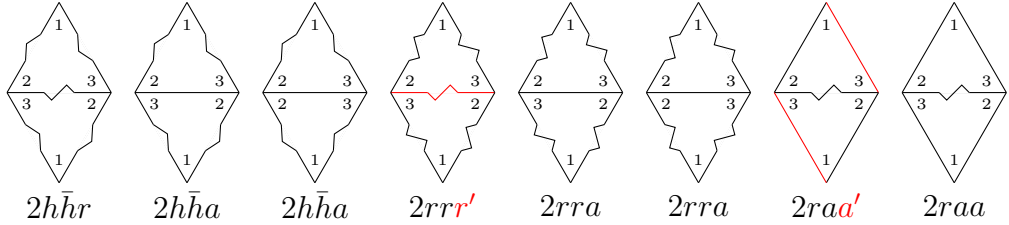
\begin{figure}[htp]
\centering
\begin{tikzpicture}[>=latex]


\begin{scope}
[decoration={markings,mark=at position 0.5 with {
\draw[white, thick] (0.2,0)--(-0.2,0);
\draw (0.2,0)--(0,0.1)--(-0.2,0);}}]

\foreach \a in {0,1,2}
\draw[postaction={decorate}, shift={(1.7*\a cm,0.4cm)}]
	(-30:0.8) -- (90:0.8);
	
\foreach \a in {0,1,2}	
\draw[postaction={decorate}, shift={(1.7*\a cm,0.4cm)}]
	(210:0.8) -- (90:0.8);
	
\foreach \a in {0,1}
\draw[postaction={decorate}, shift={(1.7*\a cm,-0.4cm)}, scale=-1]
	(-30:0.8) -- (90:0.8);
	
\foreach \a in {0,1}	
\draw[postaction={decorate}, shift={(1.7*\a cm,-0.4cm)}, scale=-1]
	(210:0.8) -- (90:0.8);

\foreach \a in {2}
\draw[postaction={decorate}, shift={(1.7*\a cm,-0.4cm)}, scale=-1]
	(90:0.8) -- (-30:0.8);
	
\foreach \a in {2}	
\draw[postaction={decorate}, shift={(1.7*\a cm,-0.4cm)}, scale=-1]
	(90:0.8) -- (210:0.8);
					
\end{scope}	


\begin{scope}
[decoration={markings,mark=at position 0.5 with {
\draw[white,thick](0.2,0)--(-0.2,0);
\draw (0.2,0)--(0.1,0.1)--(-0.1,-0.1)--(-0.2,0);}}]

\foreach \a in {0,6,7}
\draw[postaction={decorate}, shift={(1.7*\a cm,0.4cm)}]
	(210:0.8) -- (-30:0.8);

\foreach \a in {3}
\draw[postaction={decorate}, red, shift={(1.7*\a cm,0.4cm)}]
	(210:0.8) -- (-30:0.8);
	
\foreach \a in {3,4,5}
\draw[postaction={decorate}, shift={(1.7*\a cm,0.4cm)}]
	(-30:0.8) -- (90:0.8);
	
\foreach \a in {3,4,5}	
\draw[postaction={decorate}, shift={(1.7*\a cm,0.4cm)}]
	(210:0.8) -- (90:0.8);

\foreach \a in {3,4}
\draw[postaction={decorate}, shift={(1.7*\a cm,-0.4cm)}, scale=-1]
	(-30:0.8) -- (90:0.8);
	
\foreach \a in {3,4}	
\draw[postaction={decorate}, shift={(1.7*\a cm,-0.4cm)}, scale=-1]
	(210:0.8) -- (90:0.8);
	
\end{scope}


\begin{scope}
[decoration={markings,mark=at position 0.5 with {
\draw[white,thick](0.2,0)--(-0.2,0);
\draw (0.2,0)--(0.1,-0.1)--(-0.1,0.1)--(-0.2,0);}}]

\foreach \a in {5}
\draw[postaction={decorate}, shift={(1.7*\a cm,-0.4cm)}, scale=-1]
	(-30:0.8) -- (90:0.8);
	
\foreach \a in {5}	
\draw[postaction={decorate}, shift={(1.7*\a cm,-0.4cm)}, scale=-1]
	(210:0.8) -- (90:0.8);
	
\end{scope}


\foreach \a in {1,2,4,5}
\draw[shift={(1.7*\a cm,0.4cm)}]
	(-30:0.8) -- (210:0.8);

\foreach \a in {1,-1}
{
\begin{scope}[xshift=10.2cm]

\foreach \b in {0,1}
\draw[shift={(1.7*\b cm, 0.4*\a cm)}, scale=\a]
	(210:0.8) -- (90:0.8);

\draw[shift={(1.7 cm, 0.4*\a cm)}, scale=\a]
	(-30:0.8) -- (90:0.8);
	
\draw[red, yshift=0.4*\a cm, scale=\a]
	(-30:0.8) -- (90:0.8);

\end{scope}
}

\foreach \a in {1,-1}
\foreach \b in {0,1,3,4,6,7}
{
\begin{scope}[shift={(1.7*\b cm, 0.4*\a cm)}, scale=\a]

\node at (90:0.5) {\tiny 1};
\node at (210:0.5) {\tiny 2};
\node at (-30:0.5) {\tiny 3};

\end{scope}
}

\foreach \a in {1,-1}
\foreach \b in {2,5}
{
\begin{scope}[shift={(1.7*\b cm, 0.4*\a cm)}, yscale=\a]

\node at (90:0.5) {\tiny 1};
\node at (210:0.5) {\tiny 2};
\node at (-30:0.5) {\tiny 3};

\end{scope}
}

\node at (0,-1.5) {$2h\bar{h}r$};
\node at (1.7,-1.5) {$2h\bar{h}a$};
\node at (3.4,-1.5) {$2h\bar{h}a$};
\node at (5.1,-1.5) {$2rr{\color{red} r'}$};
\node at (6.8,-1.5) {$2rra$};
\node at (8.5,-1.5) {$2rra$};
\node at (10.2,-1.5) {$2ra{\color{red} a'}$};
\node at (11.9,-1.5) {$2raa$};

\end{tikzpicture}
\caption{Some companion pairs.}
\label{pair}
\end{figure}

We remark that, if the edge $\overline{23}$ is $r$, then the pair of tiles are glued together in unique way, and all the quadrilaterals are congruent. If the edge $\overline{23}$ is $a$, then the pair can be glued together in two ways, and the quadrilateral tiling may have two prototiles.

The corners $1,2,3$ in Figure \ref{triangle} have {\em angle values} that we denote by $[1],[2],[3]$. We denote a vertex in a tiling by combinations such as $12^23^2$, which means the vertex consists of one corner 1, two corners 2, and two corners 3. The sum of the angle values of all the corners at a vertex is $2\pi$. For example, the vertex $12^23^2$ implies $[1]+2[2]+2[3]=2\pi$. We call the equality the {\em angle sum} of the vertex.

The following is the triangular analogue of Lemma 4 of \cite{wy1}. If all sides of the triangle are straight, then the lemma is a consequence of the fact that the area of the triangle is $[1]+[2]+[3]-\pi$, and is also the total area $4\pi$ of the sphere divided by the number $f$ of tiles. 

\begin{lemma}\label{anglesum}
If the corners of all tiles in a tiling of the sphere by $f$ triangles have the same three angle values $[1],[2],[3]$, then 
\[
[1]+[2]+[3]
=(1 + \tfrac{4}{f})\pi.
\]
\end{lemma}

\begin{proof}
Since the angle sum of each vertex is $2\pi$, the total sum of all angles in the tiling is $2\pi v$, where $v$ is the number of vertices in the tiling. Moreover, the total sum of all angles is $f([1]+[2]+[3])$. Therefore we get $2\pi v = f([1]+[2]+[3])$. By $3f=2e$ and $v-e+f=2$, we get $f=2v-4$. Then we get $
[1]+[2]+[3] = 2\pi\tfrac{v}{f} = (1 + \tfrac{4}{f})\pi$.
\end{proof}

Strictly speaking, the edges need to be sufficiently nice in order for the angle values to make sense. Although we may assume all edges to be regular and continuously differentiable, the condition can be quite mild. For example, we may fix a small enough $\epsilon$, and then define the angle value of a corner to be $\frac{2\pi A}{A(\epsilon)}$. Here $A(\epsilon)$ is the area of the $\epsilon$-disk, and  $A$ is the area of intersection of the corner with the $\epsilon$-disk centered at the vertex.  Some kind of measurability is good enough for this definition. In a tiling, the sum of the angle values of all the corners at a vertex is still $2\pi$. Therefore Lemma \ref{anglesum} remains valid.

\section{Curvilinear Tiling of the Sphere}
\label{tiling}

\begin{proposition}
Tilings of the sphere by congruent $rrr$- or $rrr^{-1}$-triangles are regular tetrahedron, octahedron, and icosahedron.   
\end{proposition}

\begin{proof}
The distance between the three vertices of the $rrr$- or $rrr^{-1}$-triangles are equal. This implies tilings by congruent $rrr$- or $rrr^{-1}$-triangles are triangular Platonic solids, i.e., regular tetrahedron, octahedron, and icosahedron. Then we need to assign $r$ or $r^{-1}$ to the edges of the Platonic solids, such that all faces are $rrr$-triangles, or all faces are $rrr^{-1}$-triangles. 

If all faces are $rrr$-triangles, then by the consistent orientations among all tiles, all faces are $rrr$ (otherwise all faces are $r^{-1}r^{-1}r^{-1}$, which represent the same tiling). Then we get the three Platonic $rrr$-tilings on the left of Figure \ref{tiling_rrr2}.

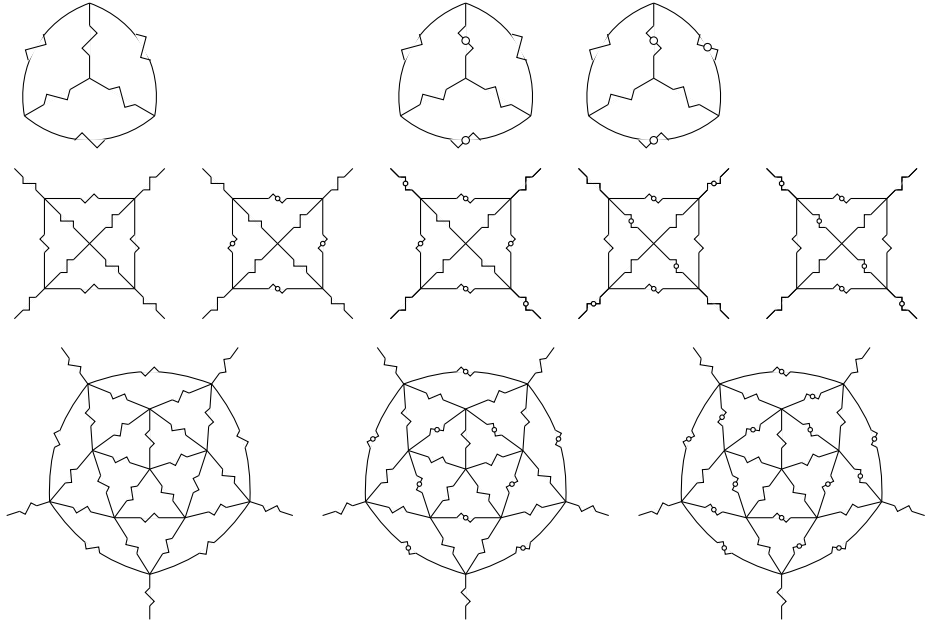
\begin{figure}[htp]
\centering
\begin{tikzpicture}[>=latex]



\begin{scope}
[decoration={markings,mark=at position 0.5 with 
{\draw[white,thick](0.2,0)--(-0.2,0);
\draw (0.2,0)--(0.1,-0.1)--(-0.1,0.1)--(-0.2,0);}}]

\foreach \b in {0,2,3}
{
\begin{scope}[xshift=2.5*\b cm]

\foreach \a in {-1,0}
\draw[postaction={decorate}, rotate=120*\a]
	(0,0) -- (-30:1);

\draw[postaction={decorate}]
	(90:1) to[out=200, in=100] (210:1);

\end{scope}	
}

\draw[postaction={decorate}]
	(0,0) -- (90:1);

\foreach \b in {0,2}
\draw[postaction={decorate}, xshift=2.5*\b cm]
	(-30:1) to[out=80, in=-20] (90:1);
	
\draw[postaction={decorate}]
	(210:1) to[out=-40, in=220] (-30:1);
									
\end{scope}	


\begin{scope}
[decoration={markings,mark=at position 0.5 with 
{\draw[white,thick](0.2,0)--(-0.2,0);
\draw (0.2,0)--(0.1,0.1)--(-0.1,-0.1)--(-0.2,0);
\filldraw[fill=white] (0,0) circle (0.05);}}]

\foreach \b in {2,3}
\draw[postaction={decorate}, xshift=2.5*\b cm]
	(0,0) -- (90:1);

\foreach \b in {2,3}	
\draw[postaction={decorate}, xshift=2.5*\b cm]
	(210:1) to[out=-40, in=220] (-30:1);

\draw[postaction={decorate}, xshift=7.5 cm]
	(-30:1) to[out=80, in=-20] (90:1);
								
\end{scope}


\begin{scope}[yshift=-2.2cm]

\foreach \a in {0,1,2,3}
{
\begin{scope}
[rotate=90*\a, decoration={markings,mark=at position 0.5 with {\draw[white,thick](0.12,0)--(-0.12,0);
\draw (0.12,0)--(0.06,0.06)--(-0.06,-0.06)--(-0.12,0);}}]

\draw[postaction={decorate}]
	(0,0) -- (0.6,0.6);
\draw[postaction={decorate}]
	(0.6,0.6) -- (-0.6,0.6);
\draw[postaction={decorate}]
	(0.6,0.6) -- (1,1);
								
\end{scope}	
}


\begin{scope}
[decoration={markings,mark=at position 0.5 with 
{\draw[white,thick](0.12,0)--(-0.12,0);
\draw (0.12,0)--(0.06,0.06)--(-0.06,-0.06)--(-0.12,0);}}]

\foreach \a in {0,1,2,3}
{

\foreach \b in {1,2}
\draw[postaction={decorate}, xshift=2.5*\b cm, rotate=90*\a]
	(0,0) -- (0.6,0.6);

\draw[postaction={decorate}, xshift=2.5 cm, rotate=90*\a]
	(0.6,0.6) -- (1,1);
}

\foreach \a in {1,-1}
\foreach \b in {3,4}
{

\foreach \c in {2,4}
\draw[postaction={decorate}, xshift=2.5*\c cm, scale=\a]
	(0.6,0.6) -- (1,1);

\draw[postaction={decorate}, xshift=7.5 cm, scale=\a]
	(-0.6,0.6) -- (-1,1);

\draw[postaction={decorate}, xshift=2.5*\b cm, scale=\a]
	(0.6,0.6) -- (0.6,-0.6);
	
\draw[postaction={decorate}, xshift=2.5*\b cm, scale=\a]
	(0,0) -- (0.6,0.6);
}
									
\end{scope}	


\begin{scope}
[decoration={markings,mark=at position 0.5 with 
{\draw[white,thick](0.12,0)--(-0.12,0);
\draw (0.12,0)--(0.06,-0.06)--(-0.06,0.06)--(-0.12,0);
\filldraw[fill=white] (0,0) circle (0.03);}}]

\foreach \a in {0,1,2,3}
\foreach \b in {1,2}
\draw[postaction={decorate}, xshift=2.5*\b cm, rotate=90*\a]
	(0.6,0.6) -- (-0.6,0.6);				

\foreach \a in {1,-1}
\foreach \b in {3,4}
{

\foreach \c in {2,4}
\draw[postaction={decorate}, xshift=2.5*\c cm, scale=\a]
	(0.6,-0.6) -- (1,-1);

\draw[postaction={decorate}, xshift=7.5 cm, scale=\a]
	(0.6,0.6) -- (1,1);
	
\draw[postaction={decorate}, xshift=2.5*\b cm, scale=\a]
	(0.6,0.6) -- (-0.6,0.6);
	
\draw[postaction={decorate}, xshift=2.5*\b cm, scale=\a]
	(0,0) -- (-0.6,0.6);
}
					
\end{scope}	

\end{scope}


\begin{scope}[shift={(0.8cm, -5.2cm)}]


\begin{scope}
[decoration={markings,mark=at position 0.5 with {
\draw[white,thick](0.12,0)--(-0.12,0);
\draw (0.12,0)--(0.06,0.06)--(-0.06,-0.06)--(-0.12,0);}}]

\foreach \a in {0,...,4}
{

\foreach \b in {0,1,2}
{
\begin{scope}[xshift=4.2*\b cm, rotate=72*\a]

\draw[postaction={decorate}]
	(0,0) -- (18:0.8);
\draw[postaction={decorate}]
	(18:0.8) -- (54:1.4);\draw[postaction={decorate}]
	(54:1.4) -- (54:2);												
\end{scope}	
}

\draw[postaction={decorate}, rotate=72*\a]
	(18:0.8) -- (90:0.8);

\draw[postaction={decorate}, rotate=72*\a]
	(54:1.4) to[out=160, in=20] (126:1.4);	

\foreach \b in {0,1}
\draw[postaction={decorate}, xshift=4.2*\b cm, rotate=72*\a]
	(54:1.4) -- (90:0.8);

}
\end{scope}


\begin{scope}
[decoration={markings,mark=at position 0.5 with 
{\draw[white,thick](0.12,0)--(-0.12,0);
\draw (0.12,0)--(0.06,-0.06)--(-0.06,0.06)--(-0.12,0);
\filldraw[fill=white] (0,0) circle (0.03);}}]

\foreach \a in {0,...,4}
{

\foreach \b in {1,2}
{
\begin{scope}[xshift=4.2*\b cm, rotate=72*\a]

\draw[postaction={decorate}]
	(18:0.8) -- (90:0.8);
\draw[postaction={decorate}]
	(54:1.4) to[out=160, in=20] (126:1.4);												
\end{scope}	
}

\draw[postaction={decorate}, xshift=8.4 cm, rotate=72*\a]
	(54:1.4) -- (90:0.8);
}			
					
\end{scope}	

\end{scope}

\end{tikzpicture}
\caption{Platonic solids of types $rrr$, $rrr^{-1}$.}
\label{tiling_rrr2}
\end{figure}

For $rrr^{-1}$-tiling, we need to assign orientations to the edges of the Platonic solid, such that each face is $rrr^{-1}$ or $rr^{-1}r^{-1}$. This means assigning $\circ$ to the edges, such that each face has either one or two $\circ$.

There are exactly two $rrr^{-1}$-tetrahedra, given by the second and third in the first row of Figure \ref{tiling_rrr2}. There are many more $rrr^{-1}$-octahedra and $rrr^{-1}$-icosahedra. Figure \ref{tiling_rrr2} gives some of them.
\end{proof} 

\begin{proposition}
Tilings of the sphere by congruent triangles of types $g\bar{g}r$, $g\bar{g}a$, $rr^{-1}{\color{red} r'}$,  $rr^{-1}a$,  $r{\color{red} r'}{\color{blue} r''}$, $r{\color{red} r'}a$ are the tetrahedra in Figure \ref{tiling_gga2}. Moreover, the two red ${\color{red} r'}$s in the $rr^{-1}{\color{red} r'}$-tetrahedron in the third of Figure \ref{tiling_gga2} can independently change directions.  
\end{proposition}

\begin{proof}
The left of Figure \ref{tiling_gga1} shows the $g\bar{g}r$-triangle and the $g\bar{g}a$-triangle. The glueing of $g$ and $\bar{g}$ means the side $\overline{12}$ of one tile matches the side $\overline{31}$ of the other tile. This implies one of the tiles \circled{1}, \circled{2}, \circled{3} determines the other two. Therefore a vertex $\bullet$ is $(213)^k=213213\cdots 213$ for $g\bar{g}r$, and a combination of 213 and 312 (instead of 213 only) for $g\bar{g}a$. By Lemma \ref{anglesum}, we have $[1]+[2]+[3]>\pi$. Since the angle sum of a vertex is $2\pi$, this implies $k=1$. Therefore $213$ is the only vertex. This further implies $f=4$, and the tiling is a tetrahedron, given by the first and second of Figure \ref{tiling_gga2}.

\begin{figure}[htp]
\centering
\begin{tikzpicture}[>=latex]

\foreach \a in {30,150}
\foreach \b in {0,1}
{
\begin{scope}[yshift=-1.7*\b cm, shift={(\a:0.8)}]

\begin{scope}
[decoration={markings,mark=at position 0.5 with {
\draw[white, thick] (0.1,0)--(-0.1,0);
\draw (0.1,0)--(-0.1,0.1)--(-0.1,0);
}}]

\draw[postaction={decorate}]
	(210:0.8) -- (90:0.8);	

\end{scope}

\begin{scope}
[decoration={markings,mark=at position 0.5 with {
\draw[white, thick] (0.1,0)--(-0.1,0);
\draw (0.1,0)--(-0.1,-0.1)--(-0.1,0);
}}]

\draw[postaction={decorate}]
	(90:0.8) -- (-30:0.8);

\end{scope}	

\end{scope}	
}


\begin{scope}
[decoration={markings,mark=at position 0.5 with {
\draw[white, thick] (0.2,0)--(-0.2,0);
\draw (0.2,0)--(0.1,-0.1)--(-0.1,0.1)--(-0.2,0);
\filldraw[fill=white] (0,0) circle (0.05);
}}]

\foreach \a in {30,150}
\foreach \b in {0,1}
\draw[postaction={decorate}, xshift=3.5 cm, yshift=-1.7*\b cm, shift={(\a:0.8)}]
	(-30:0.8) -- (90:0.8);	

\draw[postaction={decorate}]
	(-0.693, 1.2) -- ++(1.386,0);	
	
\end{scope}

\begin{scope}
[decoration={markings,mark=at position 0.5 with {
\draw[white, thick] (0.2,0)--(-0.2,0);
\draw (0.2,0)--(0.1,0.1)--(-0.1,-0.1)--(-0.2,0);
}}]

\foreach \a in {30,150}
\foreach \b in {0,1}
\draw[postaction={decorate}, xshift=3.5 cm, yshift=-1.7*\b cm, shift={(\a:0.8)}]
	(90:0.8) -- (210:0.8);

\foreach \a in {30,150}
\foreach \b in {0,1}
\draw[postaction={decorate}, red, xshift=7 cm, yshift=-1.7*\b cm, shift={(\a:0.8)}]
	(90:0.8) -- (-30:0.8);	

\foreach \a in {30,150}
\draw[postaction={decorate}, xshift=7 cm, shift={(\a:0.8)}]
	(210:0.8) -- (90:0.8);

\foreach \a in {30,150}
\draw[postaction={decorate}, xshift=7 cm, yshift=-1.7 cm, shift={(\a:0.8)}]
	(210:0.8) -- (90:0.8);

\foreach \a/\b in {-1.386/0, 0/0}
\draw[postaction={decorate}]
	(\a,\b) -- ++(1.386,0);	

\foreach \a/\b in {-1.386/0, 0/0, -0.693/1.2}
\draw[postaction={decorate}, red]
	(3.5+\a,\b) -- ++(1.386,0);
			
\foreach \a/\b in {-1.386/0, 0/0, -0.693/1.2}
\draw[postaction={decorate}, blue, xshift=7 cm]
	(\a,\b) -- ++(1.386,0);	
	
\end{scope}

\foreach \b in {0,1,2}
\draw[xshift=3.5*\b cm, yshift=-1.7 cm]
	(-1.386,0) -- (1.386,0)
	(-0.693,1.2) -- ++(1.386,0);


\foreach \b in {0,1,2}
\foreach \c in {0,1}
{
\begin{scope}[xshift=3.5*\b cm, yshift=-1.7*\c cm]
	
\foreach \a in {30,150}
{
\begin{scope}[shift={(\a:0.8)}]
	
\node at (90:0.5) {\tiny 1};
\node at (210:0.5) {\tiny 2};
\node at (-30:0.5) {\tiny 3};

\end{scope}
}

\node[inner sep=0.5, draw, shape=circle] at (30:0.8) {\tiny 1};
\node[inner sep=0.5, draw, shape=circle] at (90:0.8) {\tiny 2};
\node[inner sep=0.5, draw, shape=circle] at (150:0.8) {\tiny 3};

\fill
	(0,0) circle (0.05);

\end{scope}
}

\foreach \c in {0,1}
{
\begin{scope}[yshift=-1.7*\c cm]

\begin{scope}[shift={(90:0.8)}]

\node at (-90:0.5) {\tiny 1};
\node at (150:0.5) {\tiny 2};
\node at (30:0.5) {\tiny 3};

\end{scope}

\foreach \b in {1,2}
{
\begin{scope}[xshift=3.5*\b cm, shift={(90:0.8)}]

\node at (-90:0.5) {\tiny 1};
\node at (150:0.5) {\tiny 3};
\node at (30:0.5) {\tiny 2};

\end{scope}
}

\end{scope}
}

\end{tikzpicture}
\caption{Vertex in tilings of types $g\bar{g}r$, $g\bar{g}a$, $rr^{-1}{\color{red} r'}$,  $rr^{-1}a$,  $r{\color{red} r'}{\color{blue} r''}$, $r{\color{red} r'}a$.}
\label{tiling_gga1}
\end{figure}
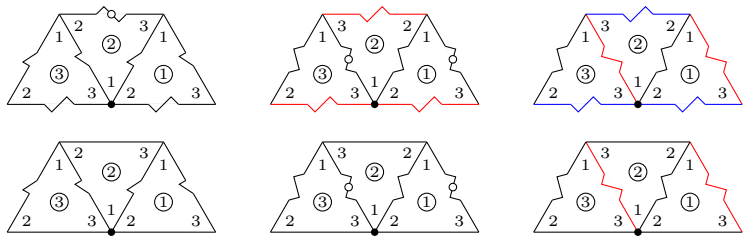

The same argument applies to $rr^{-1}{\color{red} r'}$ and $rr^{-1}a$. In the middle of Figure \ref{tiling_gga1}, one of the tiles \circled{1}, \circled{2}, \circled{3} again determines the other two, up to the independent change of the direction of red ${\color{red} r'}$ in $rr^{-1}{\color{red} r'}$. The reason for the change of direction is due to the horizontal flip of the $rr^{-1}{\color{red} r'}$-triangle in Figures \ref{triangle} and \ref{tiling_gga1}, which gives the same triangle except the reversion of the direction of the red ${\color{red} r'}$. Then we conclude a vertex $\bullet$ is $213$, and the tiling is a tetrahedron, given by the third and fourth of Figure \ref{tiling_gga2}. 

\begin{figure}[htp]
\centering
\begin{tikzpicture}[>=latex]

\foreach \a in {0,1}
{
\begin{scope}[xshift=2.2*\a cm]

\begin{scope}
[decoration={markings,mark=at position 0.5 with {
\draw[white, thick] (0.1,0)--(-0.1,0);
\draw (0.1,0)--(-0.1,0.1)--(-0.1,0);
}}]
 
\draw[postaction={decorate}]
	(210:1) -- (0,0);	
\draw[postaction={decorate}]
	(-30:1) to[out=80, in=-20] (90:1);
	
\end{scope}

\begin{scope}
[decoration={markings,mark=at position 0.5 with {
\draw[white, thick] (0.1,0)--(-0.1,0);
\draw (0.1,0)--(-0.1,-0.1)--(-0.1,0);
}}]

\draw[postaction={decorate}]
	(0,0) -- (-30:1);
\draw[postaction={decorate}]
	(90:1) to[out=200, in=100] (210:1);

\end{scope}

\end{scope}
}

\begin{scope}
[decoration={markings, mark=at position 0.5 with {
\draw[white, thick] (0.2,0)--(-0.2,0);
\draw (0.2,0)--(0.1,0.1)--(-0.1,-0.1)--(-0.2,0);
}}]

\foreach \a in {2,3}
\draw[postaction={decorate}, xshift=2.2*\a cm]
	(0,0) -- (-30:1);

\draw[postaction={decorate}, xshift=8.8cm]
	(0,0) -- (90:1);

\draw[postaction={decorate}, xshift=11 cm]
	(210:1) -- (0,0);
	
\foreach \a in {4,5}
\draw[postaction={decorate}, red, xshift=2.2*\a cm]
	(-30:1) -- (0,0);
	
\draw[postaction={decorate}, red, xshift=4.4cm]
	(0,0) -- (90:1);
	
\draw[postaction={decorate}, blue, xshift=8.8 cm]
	(210:1) -- (0,0);
			
\foreach \a in {2,3}
\draw[postaction={decorate}, xshift=2.2*\a cm]
	(90:1) to[out=200, in=100] (210:1);

\foreach \a in {0,4}
\draw[postaction={decorate}, xshift=2.2*\a cm]
	(210:1) to[out=-40, in=220] (-30:1);

\draw[postaction={decorate}, xshift=11 cm]
	(-30:1) to[out=80, in=-20] (90:1);
			
\draw[postaction={decorate}, red, xshift=4.4cm]
	(210:1) to[out=-40, in=220] (-30:1);
	
\foreach \a in {4,5}
\draw[postaction={decorate}, red, xshift=2.2*\a cm]
	(90:1) to[out=200, in=100] (210:1);

\draw[postaction={decorate}, blue, xshift=8.8 cm]
	(-30:1) to[out=80, in=-20] (90:1);
					
\end{scope}

\begin{scope}
[decoration={markings, mark=at position 0.5 with {
\draw[white, thick] (0.2,0)--(-0.2,0);
\draw (0.2,0)--(0.1,-0.1)--(-0.1,0.1)--(-0.2,0);
\filldraw[fill=white] (0,0) circle (0.05);
}}]

\foreach \a in {2,3}
\draw[postaction={decorate}, xshift=2.2*\a cm]
	(210:1) -- (0,0);

\foreach \a in {2,3}			
\draw[postaction={decorate}, xshift=2.2*\a cm]
	(-30:1) to[out=80, in=-20] (90:1);

\draw[postaction={decorate}]
	(0,0) -- (90:1);

\end{scope}

\foreach \a in {1,3,5}
\draw[xshift=2.2*\a cm]
	(0,0) -- (90:1)
	(210:1) to[out=-40, in=220] (-30:1);

\end{tikzpicture}
\caption{Tilings of types $g\bar{g}r$, $g\bar{g}a$, $rr^{-1}{\color{red} r'}$,  $rr^{-1}a$,  $r{\color{red} r'}{\color{blue} r''}$, $r{\color{red} r'}a$. The two red ${\color{red} r'}$-curves in the third picture can independently change directions.}
\label{tiling_gga2}
\end{figure}
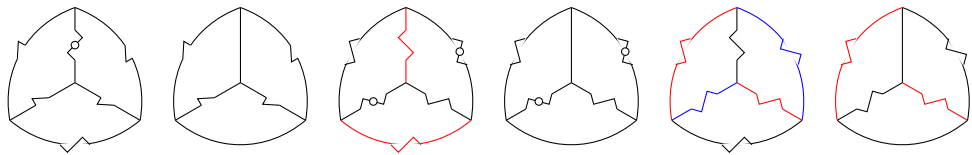

The same argument also applies to $r{\color{red} r'}{\color{blue} r''}$ and $r{\color{red} r'}a$. Again we get the unique combinations on the right of Figure \ref{tiling_gga2}. Then a vertex $\bullet$ is $213$, and the tiling is a tetrahedron, given by the fifth and sixth of Figure \ref{tiling_gga2}. 
\end{proof}

\begin{proposition}
Tilings of the sphere by congruent triangles of types $g\bar{g}^{-1}r$, $g\bar{g}^{-1}a$ are the triangular subdivisions of the Platonic solids. Moreover, the faces in $g\bar{g}^{-1}a$-triangular subdivisions can independently change orientations.
\end{proposition}

The $g\bar{g}^{-1}r$-triangular subdivisions of regular triangle, quadrilateral, and pentagon are given by the second, third and fourth of Figure \ref{tiling_gga3}. This can be applied to the regular faces of Platonic solids, and produces $g\bar{g}^{-1}r$-triangular subdivisions of Platonic solids. The fifth of Figure \ref{tiling_gga3} shows the $g\bar{g}^{-1}r$-triangular subdivision of the regular tetrahedron. 

The $g\bar{g}^{-1}a$-tilings are also triangular subdivisions of Platonic solids. The $a$-edge allows us to independently change orientations of the faces of Platonic solids. The fourth of Figure \ref{tiling_gga4} shows a $g\bar{g}^{-1}a$-triangular subdivisions of the cube, in which some faces have different orientations.

\begin{proof}
The first of Figure \ref{tiling_gga3} shows the $g\bar{g}^{-1}r$-triangle. The glueing of $g$ and $\bar{g}$ means the side $\overline{12}$ of one tile matches the side $\overline{13}$ of the other tile. Then \circled{1} determines \circled{2}, and \circled{2} further determines the next tile around the vertex $\bullet$. Therefore the vertex $\bullet$ is $1^k=11\cdots1$, and all the tiles around the vertex form a regular $k$-gon $N(1^k)$. Then the tiling is a side-to-side tiling of regular $k$-gons $N(1^k)$. This is a regular Platonic solid. In particular, we get $k=3,4,5$, and the $g\bar{g}^{-1}r$-tiling is triangular subdivision of Platonic solids.

\begin{figure}[htp]
\centering
\begin{tikzpicture}[>=latex]

\begin{scope}[yshift=0.7cm]

\foreach \a in {0,1}
{
\begin{scope}[rotate=60*\a]

\node at (-90:0.3) {\tiny 1};
\node at (-110:1.1) {\tiny 2};
\node at (-70:1.1) {\tiny 3};

\end{scope}
}

\fill
	(0,0) circle (0.05);

\node[inner sep=0.5, draw, shape=circle] at (-90:0.8) {\tiny 1};
\node[inner sep=0.5, draw, shape=circle] at (-30:0.8) {\tiny 2};

\end{scope}

\begin{scope}
[decoration={markings,mark=at position 0.5 with {
\draw[white, thick] (0.1,0)--(-0.1,0);
\draw (0.1,0)--(-0.1,0.1)--(-0.1,0);
}}]

\foreach \a in {0,1,2}
\draw[postaction={decorate}, yshift=0.7cm, rotate=60*\a]
	(-120:1.386) -- (0,0);	

\foreach \a in {0,1,2}
\draw[postaction={decorate}, xshift=2.5cm, yshift=-0.2cm, rotate=120*\a]
	(-30:1.2) -- (0,0);

\foreach \a in {0,1,2,3}
\draw[postaction={decorate}, xshift=5cm, rotate=90*\a]
	(45:1.2) -- (0,0);

\foreach \a in {0,1,2,3,4}
\draw[postaction={decorate}, xshift=7.5cm, rotate=72*\a]
	(90:1.2) -- (0,0);

\end{scope}

\begin{scope}
[decoration={markings,mark=at position 0.5 with {
\draw[white, thick] (0.2,0)--(-0.2,0);
\draw (0.2,0)--(0.1,0.1)--(-0.1,-0.1)--(-0.2,0);
}}]

\foreach \a in {0,1}
\draw[postaction={decorate}, yshift=0.7cm, rotate=60*\a]
	(-120:1.386) -- (-60:1.386);

\foreach \a in {0,1,2}
\draw[postaction={decorate}, xshift=2.5cm, yshift=-0.2cm, rotate=120*\a]
	(-30:1.2) -- (210:1.2);

\foreach \a in {0,1,2,3}
\draw[postaction={decorate}, xshift=5cm, rotate=90*\a]
	(45:1.2) -- (-45:1.2);

\foreach \a in {0,1,2,3,4}
\draw[postaction={decorate}, xshift=7.5cm, rotate=72*\a]
	(90:1.2) -- (18:1.2);
					
\end{scope}

\begin{scope}[xshift=10.5cm]

\foreach \a in {0,1,2}
{
\begin{scope}
[rotate=120*\a,
decoration={markings,mark=at position 0.5 with {
\draw[white, thick] (0.2,0)--(-0.2,0);
\draw (0.2,0)--(0.1,0.1)--(-0.1,-0.1)--(-0.2,0);
}}]

\draw[postaction={decorate}]
	(210:1.5) to[out=-40, in=220] (-30:1.5);
\draw[postaction={decorate}]
	(0,0) -- (-30:1.5);
		
\end{scope}
}

\foreach \a in {0,1,2}
{
\begin{scope}
[rotate=120*\a,
decoration={markings,mark=at position 0.5 with {
\draw[white, thick] (0.1,0)--(-0.1,0);
\draw (0.1,0)--(-0.1,0.1)--(-0.1,0);
}}]

\draw[postaction={decorate}]
	(0,0) -- (30:0.7);
\draw[postaction={decorate}]
	(90:1.5) -- (30:0.7);
\draw[postaction={decorate}]
	(-30:1.5) -- (30:0.7);
\draw[postaction={decorate}]
	(90:1.5) -- (90:2.2);
				
\end{scope}	
}

\end{scope}

\end{tikzpicture}
\caption{$g\bar{g}^{-1}r$-tiling.}
\label{tiling_gga3}
\end{figure}
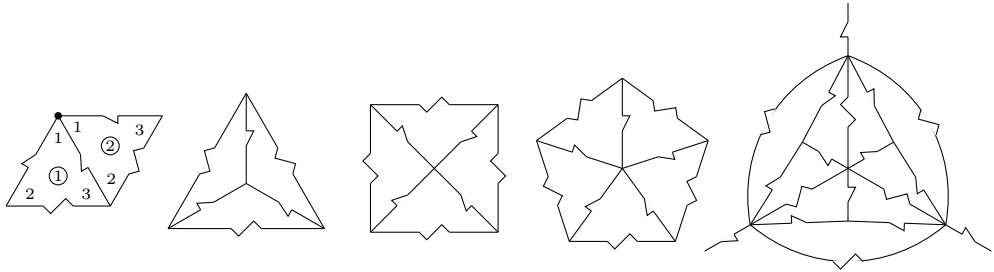

The argument for $g\bar{g}^{-1}a$ is similar. The tiling is a side-to-side tiling of regular $k$-gons $N(1^k)$, which are obtained by changing the $r$-edges in the the regular $k$-gons in Figure \ref{tiling_gga3} to $a$-edges. See the first, second and third of Figure \ref{tiling_gga4}. The tiling is then a Platonic solid, with each face given by $N(1^k)$. In other words, the tiling is a triangular subdivision of a Platonic solid. 

\begin{figure}[htp]
\centering
\begin{tikzpicture}[>=latex]

\begin{scope}
[decoration={markings,mark=at position 0.5 with {
\draw[white, thick] (0.1,0)--(-0.1,0);
\draw (0.1,0)--(-0.1,0.1)--(-0.1,0);
}}]

\foreach \a in {0,1,2}
\draw[postaction={decorate}, yshift=-0.2cm, rotate=120*\a]
	(-30:1.2) -- (0,0);

\foreach \a in {0,1,2,3}
\draw[postaction={decorate}, xshift=2.5cm, rotate=90*\a]
	(45:1.2) -- (0,0);

\foreach \a in {0,1,2,3,4}
\draw[postaction={decorate}, xshift=5cm, rotate=72*\a]
	(90:1.2) -- (0,0);

\end{scope}

\foreach \a in {0,1,2}
\draw[yshift=-0.2cm, rotate=120*\a]
	(-30:1.2) -- (210:1.2);

\foreach \a in {0,1,2,3}
\draw[xshift=2.5cm, rotate=90*\a]
	(45:1.2) -- (-45:1.2);

\foreach \a in {0,1,2,3,4}
\draw[xshift=5cm, rotate=72*\a]
	(90:1.2) -- (18:1.2);

\begin{scope}[xshift=9cm]

\foreach \a in {0,...,3}
\draw[rotate=90*\a]
	(-45:1) -- (45:1) -- (45:2.2) -- (-45:2.2);

\begin{scope}
[decoration={markings,mark=at position 0.5 with {
\draw[white, thick] (0.1,0)--(-0.1,0);
\draw (0.1,0)--(-0.1,0.1)--(-0.1,0);
}}]

\foreach \a in {0,...,3}
{
\begin{scope}[rotate=90*\a]

\draw[postaction={decorate}]
	(45:1) -- (0,0);
\draw[postaction={decorate}]
	(45:2.2) -- (45:3);
	
\end{scope}
}

\foreach \a in {0,1,2}
{
\begin{scope}[rotate=90*\a]

\draw[postaction={decorate}]
	(45:1) -- (1.1,0);
\draw[postaction={decorate}]
	(-45:1) -- (1.1,0);
\draw[postaction={decorate}]
	(45:2.2) -- (1.1,0);
\draw[postaction={decorate}]
	(-45:2.2) -- (1.1,0);
	
\end{scope}
}

\draw[postaction={decorate}]
	(0,-1.1) -- (-45:1);
\draw[postaction={decorate}]
	(0,-1.1) -- (-135:1);
\draw[postaction={decorate}]
	(0,-1.1) -- (-45:2.2);
\draw[postaction={decorate}]
	(0,-1.1) -- (-135:2.2);
						
\end{scope}	

\draw[<-]
	(0,-0.2) arc (-90:240:0.2);
	
\draw[->]
	(0,-1.3) arc (-90:240:0.2);

\end{scope}

\end{tikzpicture}
\caption{$g\bar{g}^{-1}a$-tiling.}
\label{tiling_gga4}
\end{figure}
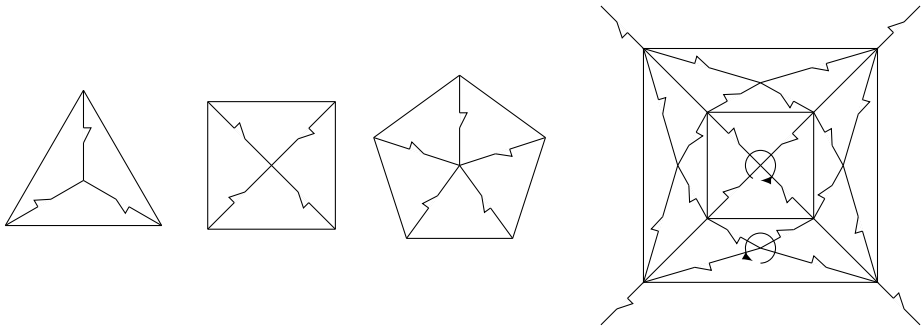

We note that, for $g\bar{g}^{-1}a$, each $N(1^k)$ has an orientation, and we may independently change the orientation of any $N(1^k)$ (i.e., flip the face) and still get a tiling. The fourth of Figure \ref{tiling_gga4} shows a $g\bar{g}^{-1}a$-triangular subdivision of the cube, with two faces having different orientations.
\end{proof}

\begin{proposition}
Tilings of the sphere by congruent $ra{\color{red} a'}$-triangles are the tetrahedron, the earth map tiling $E_{\triangle}1$ and its flip modification $FE_{\triangle}1$.   
\end{proposition}

\begin{proof}
The seventh of Figure \ref{pair} shows that a companion pair of two tiles form a quadrilateral $a{\color{red} a'}a{\color{red} a'}$. By Lemma 1 of \cite{cly}, there is no tiling with this quadrilateral as the prototile, such that all vertices have degree $\ge 3$. Therefore the quadrilateral tiling has degree $2$ vertices. This means that in a $ra{\color{red} a'}$-triangular tiling, either $123$ or $2^23^2$ is a vertex.

The angle sum of $123$ and Lemma \ref{anglesum} imply $f=4$. Then the tiling is the tetrahedron on the left of Figure \ref{tiling_raa1}. 

\begin{figure}[htp]
\centering
\begin{tikzpicture}[>=latex]

\begin{scope}[xshift=-2.5cm]

\draw
	(0,0) -- (210:1)
	(-30:1) to[out=80, in=-20] (90:1);

\draw[red]
	(0,0) -- (-30:1)
	(90:1) to[out=200, in=100] (210:1);

\begin{scope}
[decoration={markings, mark=at position 0.5 with {
\draw[white, thick] (0.2,0)--(-0.2,0);
\draw (0.2,0)--(0.1,0.1)--(-0.1,-0.1)--(-0.2,0);
}}]

\draw[postaction={decorate}]
	(0,0) -- (90:1);
	
\draw[postaction={decorate}]
	(210:1) to[out=-40, in=220] (-30:1);
					
\end{scope}

\end{scope}


\fill[gray!30]
	(0,1) rectangle (1.6,-1);

\foreach \a in {0,2,4}
\draw[xshift=0.8*\a cm]
	(0,0) -- (0.2,0) -- (0.3,-0.1) -- (0.5,0.1) -- (0.6,0) -- (0.8,0);

\foreach \a in {1,3}
\draw[xshift=0.8*\a cm]
	(0,0) -- (0.2,0) -- (0.3,0.1) -- (0.5,-0.1) -- (0.6,0) -- (0.8,0);

\foreach \a in {0,...,4}
{
\node at (0.4+0.8*\a,0.9) {\tiny 1};
\node at (0.4+0.8*\a,-0.9) {\tiny 1};
}

\foreach \a in {0,2,4}
{
\draw
	(0.8*\a,0) -- ++(0,1)
	(0.8+0.8*\a,0) -- ++(0,-1);
	
\draw[red]
	(0.8*\a,0) -- ++(0,-1)
	(0.8+0.8*\a,0) -- ++(0,1);
		
\node at (0.1+0.8*\a,0.15) {\tiny 2};
\node at (0.7+0.8*\a,0.15) {\tiny 3};
\node at (0.1+0.8*\a,-0.15) {\tiny 3};
\node at (0.7+0.8*\a,-0.15) {\tiny 2};
}

\foreach \a in {1,3}
{
\node at (0.1+0.8*\a,0.15) {\tiny 3};
\node at (0.7+0.8*\a,0.15) {\tiny 2};
\node at (0.1+0.8*\a,-0.15) {\tiny 2};
\node at (0.7+0.8*\a,-0.15) {\tiny 3};
}


\begin{scope}[xshift=6.5cm]

\foreach \a in {0,2,4}
\draw[xshift=-1cm + 0.4*\a cm]
	(0,0) -- (0.1,0) -- (0.15,-0.05) -- (0.25,0.05) -- (0.3,0) -- (0.4,0);
	
\foreach \a in {1,3}
\draw[xshift=-1cm + 0.4*\a cm]
	(0,0) -- (0.1,0) -- (0.15,0.05) -- (0.25,-0.05) -- (0.3,0) -- (0.4,0);
	
\foreach \a in {1,-1}	
{
\begin{scope}[scale=\a]

\draw[gray]
	(0,1) to[out=210, in=90] (-0.6,0)
	(0,1) to[out=-60, in=90] (0.2,0);

\draw[red!50]
	(0,1) to[out=240, in=90] (-0.2,0)
	(0,1) to[out=-30, in=90] (0.6,0);

\end{scope}
}
	
\draw[gray!50]
	(-45:1.3) -- (135:1.3);

\foreach \b in {1,-1}
{
\begin{scope}[scale=\b]
	
\draw
	(0:1) arc (0:90:1);
	
\draw[red]
	(0:1) arc (0:-90:1);

\node at (0,0.8) {\scriptsize $1^p$};
\node at (0.85,-0.2) {\scriptsize 2};
\node at (0.85,0.2) {\scriptsize 3};
	
\end{scope}
}

\end{scope}

\end{tikzpicture}
\caption{$ra{\color{red} a'}$-tiling.}
\label{tiling_raaA}
\end{figure}

Suppose $2^23^2$ is a vertex. The angle sum of the vertex implies $[2]+[3]=\pi$. This means that the quadrilateral $a{\color{red} a'}a{\color{red} a'}$ is actually a $2$-gon $G$ with $1$ as the top corner. The following argument is copied from the proof of Lemma 12 of \cite{cly}, and classifies all tilings (including non-side-to-side) by $G$. 

The corner combination at a vertex $\bullet$ is $1^k$ or $1^k\pi$, where $\pi$ appears when $\bullet$ lies in the interior of an edge of another tile. If $\bullet$ is $1^k$, then the $k$ copies of $G$ at $\bullet$ form the earth map tiling of the sphere on the left of Figure \ref{loontiling}.

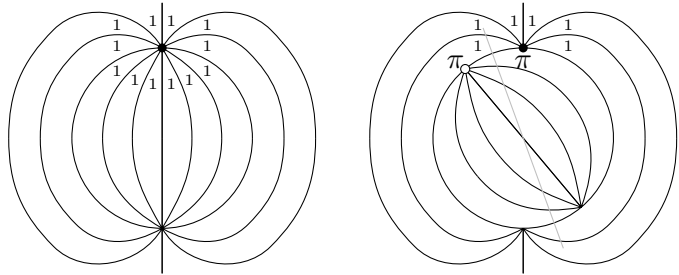
\begin{figure}[htp]
\centering
\begin{tikzpicture}[>=latex, scale=1.2]

\foreach \a in {1,-1}
\foreach \b in {1,-1}
\foreach \c in {0,1}
\draw[xshift=4*\c cm, rotate=40*\c, xscale=\a, yscale=\b]
	(0,0) -- (0,1) arc (90:0:1)
	(0.33,0) to[out=90,in=-60] (0,1)
	(0.66,0) to[out=90,in=-30] (0,1);

\foreach \a in {1,-1}
\foreach \b in {1,-1}
\foreach \c in {0,1}
\draw[xshift=4*\c cm, xscale=\a, yscale=\b]
	(0,1) -- (0,1.5)
	(1.35,0) to[out=90,in=-45] (45:1.35) to[out=135,in=30] (0,1)
	(1.7,0) to[out=90,in=-45] (45:1.7) to[out=135,in=60] (0,1);
	
\fill
	(0,1) circle (0.05)
	(4,1) circle (0.05);

\filldraw[xshift=4cm, fill=white]
	(130:1) circle (0.05);

\draw[xshift=4cm, gray!50]
	(110:1.3) -- (110:-1.3);
		
\foreach \a in {1,-1}
{
\begin{scope}[xscale=\a]
	
\node at (0.1,0.6) {\tiny 1};
\node at (0.3,0.65) {\tiny 1};
\node at (0.5,0.75) {\tiny 1};

\node at (0.1,1.3) {\tiny 1};
\node at (0.5,1.25) {\tiny 1};
\node at (0.5,1.02) {\tiny 1};

\end{scope}
}

\foreach \a in {1,-1}
{
\begin{scope}[xshift=4cm, xscale=\a]

\node at (0.1,1.3) {\tiny 1};
\node at (0.5,1.25) {\tiny 1};
\node at (0.5,1.02) {\tiny 1};

\end{scope}
}

\node at (4,0.85) {\small $\pi$};
\node at (3.25,0.85) {\small $\pi$};

\end{tikzpicture}
\caption{Tilings by $2$-gons.}
\label{loontiling}
\end{figure}

Next, we may assume all vertices are $1^k\pi$. Then the $k$ copies of $G$ at the $1^k$ part of the $\bullet$-vertex $1^k\pi$ form a half earth map tiling of a hemisphere. This is the part of the tiling outside the circle on the right of Figure \ref{loontiling}. Then any tile in the complementary hemisphere (which is inside the circle) has a $\circ$-vertex $1^k\pi$, and $k$ copies of $G$ at the $1^k$ part of $\circ$ form a half earth map tiling of the complementary hemisphere. The right of Figure \ref{loontiling} is obtained from the left by flipping the tiling of the complementary hemisphere with respect to the gray line. 

We add the $r$-edge to the tilings in Figure \ref{loontiling}, and require the triangular tiling to be side-to-side. From the earth map tiling on the left of Figure \ref{loontiling}, we get the earth map tiling $E_{\triangle}1$ in the middle of Figure \ref{tiling_raaA} (see Figures 5 of \cite{cly}). The tiling is obtained by repeating the timezone consisting of four $ra{\color{red} a'}$-triangles in the gray region. All the upward edges converge to the a vertex (north pole), and all the downward edges converge to the another vertex (south pole).

Suppose the number of timezones in $E_{\triangle}1$ is an odd number $p=2q+1$. This means $f=4p=8q+4$ and $[1]=\frac{2\pi}{2p}=\frac{\pi}{2q+1}$. Then $q$ and half timezones form a hemisphere. The middle of Figure \ref{tiling_raaA} is a hemisphere with $q=2$, and the right of Figure \ref{tiling_raaA} shows the hemisphere as a disk, together with the corner combinations along the boundary of the hemisphere.

If we glue two hemispheres in the usual way, such that the north and south poles match, and $a$-edges and ${\color{red} a'}$-edges match, then we get the earth map tiling $E_{\triangle}1$. If we flip one hemisphere with respect to the gray line, and then glue the two hemispheres together, then we get the flip modification tiling $FE_{\triangle}1$ (see Figure 20 of \cite{cly}).
\end{proof}

\begin{proposition}
Tilings of the sphere by congruent triangles of types $h\bar{h}r$, $h\bar{h}a$, $rr{\color{red} r'}$, $rra$ are the following:
\begin{itemize}
\item Tetrahedron $P_4$.
\item Triangular subdivisions  $T_{\triangle}P_n$ of all five Platonic solids.
\item Simple triangular subdivisions $S_{\triangle}P_6$ of the cube.
\item Earth map tilings $E_{\triangle}^I1$, $E_{\triangle}^J1$, $E_{\triangle}2$, and the flip modification $FE_{\triangle}^I1$ (for $h\bar{h}r$, $h\bar{h}a$), and rotation modifications $RE_{\triangle}^I1$ (for $rr{\color{red} r'}$, $rra$), $RE_{\triangle}^J1$, $RE_{\triangle}2$. 
\end{itemize} 
\end{proposition}

\begin{proof}
We use the concept of fan that was first introduced in \cite{cly5}. The edge $\overline{23}$ is distinguished from the other two edges in $h\bar{h}r$, $h\bar{h}a$, $rr{\color{red} r'}$, $rra$. At any vertex, the edge divides the corners at the vertex into several groups, which we call {\em fans}. The fans are illustrated in Figure \ref{fan}, with the edges $\overline{23}$ indicated by thick lines. A fan consists of a sequence of corners $a1\cdots 1b$, with $a,b$ being 2 or 3, and all corners between $a$ and $b$ are 1.

\begin{figure}[htp]
\centering
\begin{tikzpicture}[>=latex]

\fill (0,0) circle (0.1);

\foreach \a in {0,150,240}
\draw[very thick]
	(0,0) -- (\a:1.5);

\foreach \a in {0,...,11}
\draw
	(0,0) -- (30*\a:1.5);

\foreach \a in {1,2,3}
\draw[very thick]
	(30*\a:1.5) -- (30+30*\a:1.5);

\draw
	(0:1.5) -- (30:1.5)
	(120:1.5) -- (150:1.5);

\node at (15:0.4) {\tiny 2};
\node at (5:1.35) {\tiny 3};
\node at (25:1.35) {\tiny 1};

\foreach \a in {0,1,2}
\node at (45+30*\a:0.4) {\tiny 1};

\node at (135:0.4) {\tiny 3};
\node at (145:1.35) {\tiny 2};
\node at (125:1.35) {\tiny 1};

\draw[<->]
	(5:1.7) arc (5:145:1.7);
\draw[<->]
	(155:1.7) arc (155:235:1.7);
\draw[<->]
	(245:1.7) arc (245:355:1.7);

\node[rotate=-15] at (75:1.9) {fan};
\node[rotate=105] at (195:1.9) {fan};
\node[rotate=-150] at (-60:1.9) {fan};

\begin{scope}[shift={(4cm,0.2cm)}]

\foreach \a in {1,2,3}
\draw[rotate=45*\a]
	(0,0) -- (0.5,0) -- (0.7,-0.1) -- (0.9,0) -- (1.4,0);
	
\draw[thick]
	(-1.4,0) -- (1.4,0)
	(45:1.4) -- (90:1.4) -- (135:1.4);

\draw
	(0:1.4) -- (15:1.3) -- (22.5:1.4) -- (30:1.3) -- (45:1.4)
	(180:1.4) -- (165:1.3) -- (157.5:1.2) -- (150:1.3) -- (135:1.4);
	
\fill
	(0,0) circle (0.1);
	
\node at (22.5:0.3) {\tiny 2};
\node at (6:1.2) {\tiny 3};
\node at (39:1.2) {\tiny 1};

\node at (157.5:0.3) {\tiny 3};
\node at (174:1.2) {\tiny 2};
\node at (141:1.2) {\tiny 1};

\foreach \a in {1,2}
{
\begin{scope}[rotate=45*\a]

\node at (22.5:0.3) {\tiny 1};
\node at (6:1.2) {\tiny 3};
\node at (39:1.2) {\tiny 2};

\end{scope}
}

\end{scope}

\begin{scope}[shift={(4cm,-1.6cm)}]

\foreach \a in {1,2,3}
\draw[rotate=45*\a]
	(0,0) -- (0.5,0) -- (0.6,-0.1) -- (0.8,0.1) -- (0.9,0) -- (1.4,0);
	
\foreach \a in {0,4}
\draw[rotate=45*\a, thick, red]
	(0,0) -- (0.5,0) -- (0.6,-0.1) -- (0.8,0.1) -- (0.9,0) -- (1.4,0);
	
\foreach \a in {0,3}
\draw[rotate=45*\a]
	(0:1.4) -- (12:1.3) -- (18:1.4) -- (27:1.2) -- (33:1.3) -- (45:1.4);

\foreach \a in {1,2}
\draw[rotate=45*\a, thick, red]
	(0:1.4) -- (12:1.3) -- (18:1.4) -- (27:1.2) -- (33:1.3) -- (45:1.4);
		
\fill
	(0,0) circle (0.1);
	
\node at (22.5:0.3) {\tiny 2};
\node at (6:1.2) {\tiny 3};
\node at (39:1.2) {\tiny 1};

\node at (157.5:0.3) {\tiny 3};
\node at (174:1.2) {\tiny 2};
\node at (141:1.2) {\tiny 1};

\foreach \a in {1,2}
{
\begin{scope}[rotate=45*\a]

\node at (22.5:0.3) {\tiny 1};
\node at (6:1.2) {\tiny 3};
\node at (39:1.2) {\tiny 2};

\end{scope}
}

\end{scope}

\end{tikzpicture}
\caption{Fans at a vertex.}
\label{fan}
\end{figure}
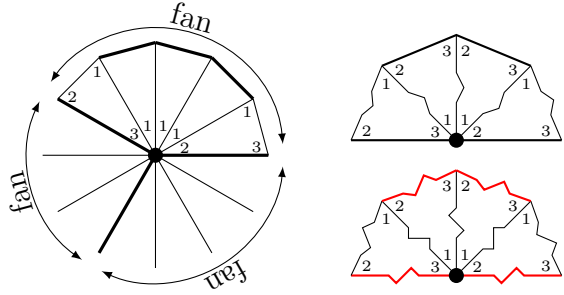

The upper right of Figure \ref{fan} shows one fan at a vertex in an $h\bar{h}a$-tiling. We find a fan must be $21\cdots 1 3$, and cannot be $21\cdots 1 2$ or $31\cdots 1 3$. The lower right picture shows that the same happens to $rr{\color{red} r'}$ tiling. Actually, the same also happens to $h\bar{h}r$ and $rra$. A consequence of the observation is that the numbers of the corners $2$ and $3$ at any vertex are the same. In other words, all vertices $1^k2^l3^l$.

By (3.2) and (3.3) in Section 3.1 of \cite{cly}, the number $v_i$ of vertices of degree $i$ satisfy
\[
3v_3+2v_4+v_5=12+v_7+2v_8+3v_9+\cdots.
\]
Therefore a triangular tiling of the sphere has a vertex of degree 3, 4, or 5. The vertices $1^k2^l3^l$ of such degrees are $123,1^3,1^4,1^5,1^223,12^23^2,1^323,2^23^2$. We assume one such vertex, and combine with Lemma \ref{anglesum} to get the angle values of the corners. 
\begin{align*}
123 &\colon 
[1]+[2]+[3]=2\pi, \; 
f=4.\\
1^3 &\colon 
[1]=\tfrac{2}{3}\pi,\;
[2]+[3]=(\tfrac{1}{3}+\tfrac{4}{f})\pi. \\
1^4 &\colon
[1]=\tfrac{1}{2}\pi,\;
[2]+[3]=(\tfrac{1}{2}+\tfrac{4}{f})\pi. \\
1^5 &\colon
[1]=\tfrac{2}{5}\pi,\;
[2]+[3]=(\tfrac{3}{5}+\tfrac{4}{f})\pi. \\
1^223 &\colon
[1]=(1-\tfrac{4}{f})\pi,\;
[2]+[3]=\tfrac{8}{f}\pi. \\
12^23^2 &\colon
[1]=\tfrac{8}{f}\pi,\;
[2]+[3]=(1-\tfrac{4}{f})\pi. \\
1^323 &\colon
[1]=(\tfrac{1}{2}-\tfrac{2}{f})\pi,\;
[2]+[3]=(\tfrac{1}{2}+\tfrac{6}{f})\pi. \\
2^23^2 &\colon [1]=\tfrac{4}{f}\pi,\;
[2]+[3]=\pi. 
\end{align*}

We remark that the number $f$ of tiles is even. The reason is that, for side-to-side triangular tilings, $3f$ is twice of the number of edges. 

Suppose we always have $l=k$ in $1^k2^l3^l$. Then by Lemma \ref{anglesum}, we know $k=l=1$ (otherwise the angle sum $>2\pi$). This means $\alpha\beta\gamma$ is the only vertex. This is the first case of Table \ref{avc}.

Suppose we do not always have $k=l$. The total number of each corner in the whole tiling is the number $f$ of tiles. Therefore the total numbers of $1,2,3$ are the same. This implies that we have a vertex $1^k2^l3^l$ with $k>l$, and also a vertex with $k<l$. 

If $1^3$ is a vertex, then the vertex $1^k2^l3^l$ means $k\tfrac{2}{3}+l(\tfrac{1}{3}+\tfrac{4}{f})=2$. This is the same as $f=\frac{12l}{6-2k-l}$. We know there are vertices satisfying $k<l$. The vertex also satisfies $k+2l\ge 3$. The two inequalities together is equivalent to $k<l\ge 2$. Then we find all non-negative integers $k,l$ satisfying $k<l\ge 2$, such that $f=\frac{12l}{6-2k-l}$ is an even integer $\ge 4$. The result is $f=6,12,24,36,60$. Then we get specific values of $[1]$ and $[2]+[3]$, which we may use to determine all the possible vertices $1^k2^l3^l$. These are the cases in Table \ref{avc} that include $1^3$ as one of the vertices.

We carry out the same argument for $1^4,1^5,1^323$. We get $f=8,16,24$ for $1^4$, and $f=10,20,60$ for $1^5$, and $f=12,20,36$ for $1^323$. Then we get all the cases in Table \ref{avc} that has fixed $f$.

\renewcommand{\arraystretch}{1.2}

\begin{table}[htp]
\centering
\begin{tabular}{|c|c|c|c|} 
\hline
$f$
& [1]
& [2]+[3]
& vertex \\
\hline  \hline 
4
& 
&
& $123$  \\
\hline
$6$
& $\frac{1}{3}\pi$ 
& $\pi$ 
& $1^3,2^23^2$ (3) \\
\hline
8
& $\frac{1}{2}\pi$ 
& $\pi$ 
& $1^4,1^223,2^23^2$ (3) \\
\hline
$10$
& $\frac{2}{5}\pi$ 
& $\pi$ 
& $1^5,2^23^2$ (3) \\
\hline
\multirow{2}{*}{$12$}
& $\frac{2}{3}\pi$ 
& $\frac{2}{3}\pi$ 
& $1^3,1^223,12^23^2,2^33^3$ \\
\cline{2-4}
& $\frac{1}{3}\pi$ 
& $\pi$ 
& $1^323,1^6,2^23^2$ (3)  \\
\hline
$16$
& $\frac{1}{2}\pi$ 
& $\frac{3}{4}\pi$ 
& $1^4,12^23^2$ (2) \\
\hline
$20$
& $\frac{2}{5}\pi$ 
& $\frac{4}{5}\pi$ 
& $1^5,1^323,12^23^2$ (2) \\
\hline
\multirow{2}{*}{$24$}
& $\frac{2}{3}\pi$ 
& $\frac{1}{2}\pi$ 
& $1^3,2^43^4$ \\
\cline{2-4} 
& $\frac{1}{2}\pi$ 
& $\frac{2}{3}\pi$
& $1^4,2^33^3$ \\
\hline
\multirow{2}{*}{$36$}
& $\frac{2}{3}\pi$
& $\frac{4}{9}\pi$ 
& $1^3,12^33^3$ \\
\cline{2-4} 
& $\frac{4}{9}\pi$ 
& $\frac{2}{3}\pi$
& $1^323,2^33^3$ \\
\hline
\multirow{2}{*}{$60$}
& $\frac{2}{3}\pi$
& $\frac{2}{5}\pi$ 
& $1^3,2^53^5$ \\
\cline{2-4} 
& $\frac{2}{5}\pi$ 
& $\frac{2}{3}\pi$
& $1^5,2^33^3$ \\
\hline
(1)  
& $\frac{4}{f}\pi$ 
& $\pi$ 
& $2^23^2,1^{\frac{f}{4}}23,1^{\frac{f}{2}}$ \\
\hline
(2)
& $(1-\frac{4}{f})\pi$ 
& $\frac{8}{f}\pi$ 
& $1^223,12^{\frac{f+4}{8}}3^{\frac{f+4}{8}},2^{\frac{f}{4}}3^{\frac{f}{4}}$ \\
\hline 
(3)
& $\frac{8}{f}\pi$ 
& $(1-\frac{4}{f})\pi$ 
& $12^23^2,1^{\frac{f+4}{8}}23,1^{\frac{f}{4}}$ \\
\hline
\end{tabular}
\caption{All the vertices for tilings of types $h\bar{h}r$, $h\bar{h}a$, $rr{\color{red} r'}$, $rra$.}
\label{avc}
\end{table}

If $1^223$ is a vertex, then by the angle values, a vertex is $1^223,12^l3^l,2^l3^l$. Then we find $l=\frac{f+4}{8}$ in $12^l3^l$, and $l=\frac{f}{4}$ in $2^l3^l$. If $12^23^2$ or $2^23^2$ is a vertex, then we get all the vertices in similar way. These cases allow variable $f$, and become the cases (1), (2), (3) in Table \ref{avc}.

Some cases with the fixed $f$ are labeled by (2) or (3). These are special cases of the corresponding variable $f$ cases. Therefore we will not separately discuss these cases. 

\medskip

\noindent {\bf Case $f=4$.} Tetrahedron $P_4$.

\medskip

The only vertex is $123$. The tiling is tetrahedron, given by Figure \ref{tiling_hha2}.

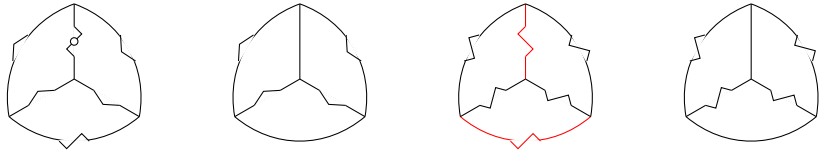
\begin{figure}[htp]
\centering
\begin{tikzpicture}[>=latex,scale=1]


\begin{scope}[yshift=2.5cm]


\begin{scope}
[decoration={markings,mark=at position 0.5 with {
\draw[white, thick] (0.2,0)--(-0.2,0);
\draw (0.2,0)--(0,0.1)--(-0.2,0);}}]

\foreach \a in {0,1}
{
\begin{scope}[xshift=3*\a cm]
			
\draw[postaction={decorate}]
	(-30:1) to[out=80, in=-20] (90:1);
\draw[postaction={decorate}]
	(210:1) to[out=100, in=200] (90:1);
\draw[postaction={decorate}]
	(210:1) -- (0,0);	
\draw[postaction={decorate}]
	(-30:1) -- (0,0);				

\end{scope}	
}
		
\end{scope}	


\begin{scope}
[decoration={markings,mark=at position 0.5 with {
\draw[white, thick] (0.2,0)--(-0.2,0);
\draw (0.2,0)--(0.1,0.1)--(-0.1,-0.1)--(-0.2,0);
}}]

\draw[postaction={decorate}]
	(210:1) to[out=-40, in=220] (-30:1);

\draw[postaction={decorate}, red, xshift=6cm]
	(210:1) to[out=-40, in=220] (-30:1);
\draw[postaction={decorate}, red, xshift=6cm]
	(0,0) -- (90:1);
	
\foreach \a in {2,3}
{
\begin{scope}[xshift=3*\a cm]
			
\draw[postaction={decorate}]
	(-30:1) to[out=80, in=-20] (90:1);
\draw[postaction={decorate}]
	(210:1) to[out=100, in=200] (90:1);
\draw[postaction={decorate}]
	(210:1) -- (0,0);	
\draw[postaction={decorate}]
	(-30:1) -- (0,0);				

\end{scope}	
}
	
\end{scope}


\begin{scope}
[decoration={markings,mark=at position 0.5 with {
\draw[white, thick] (0.2,0)--(-0.2,0);
\draw (0.2,0)--(0.1,-0.1)--(-0.1,0.1)--(-0.2,0);
\filldraw[fill=white] (0,0) circle (0.05);
}}]

\draw[postaction={decorate}]
	(0,0) -- (90:1);
	
\end{scope}

\foreach \a in {1,3}
\draw[xshift=3*\a cm]
	(0,0) -- (90:1)
	(210:1) to[out=-40, in=220] (-30:1);
	
\end{scope}

\end{tikzpicture}
\caption{Tetrahedron tiling of types $h\bar{h}r$, $h\bar{h}a$, $rr{\color{red} r'}$, $rra$.}
\label{tiling_hha2}
\end{figure}

\medskip

\noindent {\bf Case $f=24,60$.} Triangular subdivision $T_{\triangle}P_n$ of Platonic solid $P_n$.

\medskip

Figure \ref{tiling_hha4} shows that the tiles at the vertex $1^k$ ($k=3,4,5$) form unique neighborhood tiling $N(1^k)$. Then the tiling by $N(1^k)$ is a Platonic solid $P_n$, $n=4,6,8,12,20$, and the whole tiling is the triangular subdivision $T_{\triangle}P_n$ of Platonic solid.

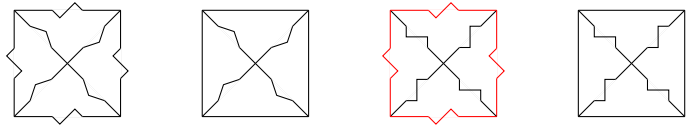
\begin{figure}[htp]
\centering
\begin{tikzpicture}[>=latex]


\begin{scope}
[decoration={markings,mark=at position 0.5 with {
\draw[white, thick] (0.2,0)--(-0.2,0);
\draw (0.2,0)--(0,0.1)--(-0.2,0);}}]

\foreach \a in {0,1,2,3}
\foreach \b in {0,1}
\draw[postaction={decorate}, xshift=2.5*\b cm, rotate=90*\a]
	(45:1) -- (0,0);
			
\end{scope}	


\begin{scope}
[decoration={markings,mark=at position 0.5 with {
\draw[white,thick](0.2,0)--(-0.2,0);
\draw (0.2,0)--(0.1,0.1)--(-0.1,-0.1)--(-0.2,0);}}]

\foreach \a in {0,1,2,3}
{

\draw[postaction={decorate}, rotate=90*\a]
	(45:1) -- (-45:1);

\foreach \b in {2,3}
\draw[postaction={decorate}, xshift=2.5*\b cm, rotate=90*\a]
	(45:1) -- (0,0);

\draw[postaction={decorate}, red, xshift=5cm, rotate=90*\a]
	(45:1) -- (-45:1);

}
		
\end{scope}

\foreach \a in {0,1,2,3}
\foreach \b in {1,3}
\draw[postaction={decorate}, xshift=2.5*\b cm, rotate=90*\a]
	(45:1) -- (-45:1);

\end{tikzpicture}
\caption{$N(1^4)$ of types $h\bar{h}r$, $h\bar{h}a$, $rr{\color{red} r'}$, $rra$.}
\label{tiling_hha4}
\end{figure}

The triangular subdivision tilings $T_{\triangle}P_n$ are unique for $h\bar{h}r$ and $rr{\color{red} r'}$. The left of Figure \ref{tiling_hha4A} shows the $h\bar{h}r$-triangular subdivision of the cube $T_{\triangle}P_6$. However, for $h\bar{h}a$ and $rra$, we may independently change the orientations of $N(1^k)$. The right of Figure \ref{tiling_hha4A} shows one $h\bar{h}a$-triangular subdivision of the cube $T_{\triangle}P_6$, in which two faces have different orientations.

\begin{figure}[htp]
\centering
\begin{tikzpicture}[>=latex]


\foreach \a in {0,1,2,3}
{
\begin{scope}
[decoration={markings,mark=at position 0.5 with {
\draw[white, thick] (0.2,0)--(-0.2,0);
\draw (0.2,0)--(0,0.1)--(-0.2,0);}},
rotate=90*\a]

\draw[postaction={decorate}]
	(45:1) -- (0,0);
\draw[postaction={decorate}]
	(-45:1) -- (0:1.1);
\draw[postaction={decorate}]
	(45:1) -- (0:1.1);
\draw[postaction={decorate}]
	(-45:2.2) -- (0:1.1);
\draw[postaction={decorate}]
	(45:2.2) -- (0:1.1);

\draw[postaction={decorate}]
	(45:2.2) -- (45:3.2);
					
\end{scope}	
}


\begin{scope}
[decoration={markings,mark=at position 0.5 with {
\draw[white,thick](0.2,0)--(-0.2,0);
\draw (0.2,0)--(0.1,0.1)--(-0.1,-0.1)--(-0.2,0);}}]

\foreach \a in {0,1,2,3}
{
\draw[postaction={decorate}, rotate=90*\a]
	(45:1) -- (-45:1);
\draw[postaction={decorate}, rotate=90*\a]
	(45:1) -- (45:2.2);
\draw[postaction={decorate}, rotate=90*\a]
	(45:2.2) -- (-45:2.2);
}
		
\end{scope}

\begin{scope}[xshift=6cm]

\foreach \a in {0,1,2,3}
\draw[rotate=90*\a]
	(-45:1) -- (45:1) -- (45:2.2) -- (-45:2.2);


\begin{scope}
[decoration={markings,mark=at position 0.5 with {
\draw[white, thick] (0.2,0)--(-0.2,0);
\draw (0.2,0)--(0,0.1)--(-0.2,0);}}]

\foreach \a in {0,1,2,3}
{
\draw[postaction={decorate}, rotate=90*\a]
	(45:1) -- (0,0);
\draw[postaction={decorate}, rotate=90*\a]
	(45:2.2) -- (45:3.2);
}

\foreach \a in {0,1,2}
{	
\begin{scope}[rotate=90*\a]

\draw[postaction={decorate}]
	(-45:1) -- (0:1.1);
\draw[postaction={decorate}]
	(45:1) -- (0:1.1);
\draw[postaction={decorate}]
	(-45:2.2) -- (0:1.1);
\draw[postaction={decorate}]
	(45:2.2) -- (0:1.1);

\end{scope}
}

\draw[postaction={decorate}]
	(0,-1.1) -- (-45:1);
\draw[postaction={decorate}]
	(0,-1.1) -- (-135:1);
\draw[postaction={decorate}]
	(0,-1.1) -- (-45:2.2);
\draw[postaction={decorate}]
	(0,-1.1) -- (-135:2.2);

\end{scope}

\draw[<-]
	(0,-0.2) arc (-90:240:0.2);
	
\draw[->]
	(0,-1.3) arc (-90:240:0.2);

\end{scope}

\end{tikzpicture}
\caption{Triangular subdivisions $T_{\triangle}P_6$ for $h\bar{h}r$ and $h\bar{h}a$.}
\label{tiling_hha4A}
\end{figure}
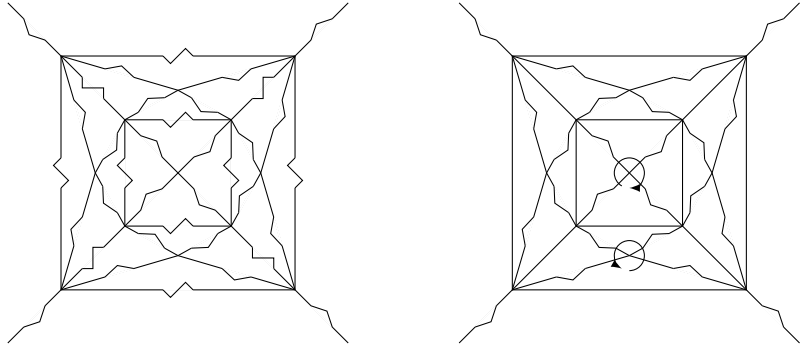

\medskip

\noindent {\bf Case $f=36$.} No tiling.

\medskip

We consider the tiling by the quadrilaterals formed by the companion pairs. Although the quadrilaterals may not be congruent, such as the second and third of Figure \ref{pair}, or the fifth and sixth of Figure \ref{pair}, we ignore the difference between $2$ and $3$, and ignore the difference between $h,\bar{h}$, and between $r,r^{-1}$ ($r$ becomes $r^{-1}$ if we flip $rra$). Then the quadrilateral tiling has corners $\alpha,\beta,\alpha,\beta$, with one of $\alpha,\beta$ being $1$, and the other being any of $22,23,33$. Then the vertices of the quadrilateral tiling are $\alpha^3\beta,\beta^3$. Then the argument based on Figure 34 of \cite{cly} still works, and leads to a contradiction.

\medskip

\noindent {\bf Case $f=12$.} Simple triangular subdivision $S_{\triangle}P_6$ of cube $P_6$.

\medskip

In tilings of types $h\bar{h}r$, $h\bar{h}a$, $rr{\color{red} r'}$, $rra$, the tiles form companion pairs sharing $r$-edge or $a$-edge. For $f=12$, by the corner combinations at vertices in Table \ref{avc}, the companion pairs form a quadrilateral tiling with all vertices having degree $3$. This is the cube. Therefore the triangle tilings are simple triangular subdivisions of the cube $S_{\triangle}P_6$ in Section 2.3 of \cite{cly}, which use diagonals to divide all quadrilaterals into triangle pairs. By Figure 10 of \cite{cly}, there are seven non-equivalent ways of dividing the cube faces into halves. 

For $h\bar{h}r$, the first of Figure \ref{pair} shows that the faces of cube are $h\bar{h}h\bar{h}$. For $h\bar{h}a$, the second and third of Figure \ref{pair} show that the faces of cube are $h\bar{h}h\bar{h}$ or $hh\bar{h}\bar{h}$. On the left of Figure \ref{tiling_hha3A}, we indicate the two quadrilaterals by connecting the two $\bar{h}$-edges with dotted lines. This is a much easier way to visualise the cube, and we can easily find that there are exactly four cubes with $h\bar{h}h\bar{h}$ or $hh\bar{h}\bar{h}$ as faces, on the right of Figure \ref{tiling_hha3A}.

\begin{figure}[htp]
\centering
\begin{tikzpicture}[>=latex]

\begin{scope}[shift={(-4cm,0.2cm)}]


\begin{scope}
[decoration={markings,mark=at position 0.5 with {
\draw[white, thick] (0.2,0)--(-0.2,0);
\draw (0.2,0)--(0,0.1)--(-0.2,0);}}]

\foreach \a in {0,1}
{
\draw[postaction={decorate}, rotate=180*\a]
	(0.6,-0.6) -- (0.6,0.6);
\draw[postaction={decorate}, rotate=180*\a]
	(-0.6,0.6) -- (0.6,0.6);
\draw[postaction={decorate}, xshift=1.8cm, rotate=90*\a]
	(0.6,-0.6) -- (-0.6,-0.6);
\draw[postaction={decorate}, xshift=1.8cm, rotate=90*\a]
	(0.6,0.6) -- (-0.6,0.6);
}
			
\end{scope}	

\draw[densely dotted]
	(-0.5,0) -- (0.5,0);

\draw[densely dotted, xshift=1.8cm]
	(-0.5,0) to[out=0, in=-90] (0,0.5);

\node at (0,-1) {$h\bar{h}h\bar{h}$};
\node at (1.8,-1) {$hh\bar{h}\bar{h}$};

\end{scope}


\foreach \a in {0,1,2,3}
\foreach \b in {0,1,2,3}
\draw[xshift=2.5*\b cm, rotate=90*\a]
	(-0.4,0.4) -- (0.4,0.4) -- (0.8,0.8) -- (-0.8,0.8);

\begin{scope}[densely dotted]

\foreach \a in {0,1,2}
\draw
	(-0.4+2.5*\a,0) -- ++(0.8,0);

\foreach \a in {-1,1}
\draw[scale=\a]
	(0.6,0.6) -- (0.6,-0.6)
	(0,0.4) -- (0,0.8);
	
\foreach \a in {0,1}
\draw[xshift=5*\a cm, rotate=90*\a]
	(-0.8,0) to[out=180, in=180]
	(-0.8,-1) -- (0.8,-1) to[out=0, in=0]
	(0.8,0);
	
\draw[xshift=2.5 cm]
	(-0.6,0.6) -- (-0.6,-0.6)
	(0,-0.4) -- (0,-0.8);

\foreach \a in {1,2,3}	
\draw[xshift=2.5*\a cm]	
	(0,0.4) to[out=90, in=180] (0.6,0.6)
	(0.6,-0.6) to[out=90, in=180] (0.8,0);
	
\foreach \a in {1,3}	
\draw[xshift=2.5*\a cm]	
	(0,0.8) to[out=90, in=45] 
	(-0.9,0.9) to[out=225, in=180] 
	(-0.8,0);

\draw[xshift=5 cm]	
	(0,-0.4) to[out=-90, in=0] (-0.6,-0.6)
	(-0.6,0.6) to[out=-90, in=0] (-0.8,0);
		
\draw[xshift=7.5 cm]	
	(0,-0.8) to[out=90, in=0] (-0.6,-0.6)
	(-0.6,0.6) to[out=-90, in=180] (-0.4,0)
	(0,-0.4) to[out=90, in=180] (0.4,0);

\end{scope}
	
\end{tikzpicture}
\caption{Cubes with $h\bar{h}h\bar{h}$ and $hh\bar{h}\bar{h}$ faces.}
\label{tiling_hha3A}
\end{figure}
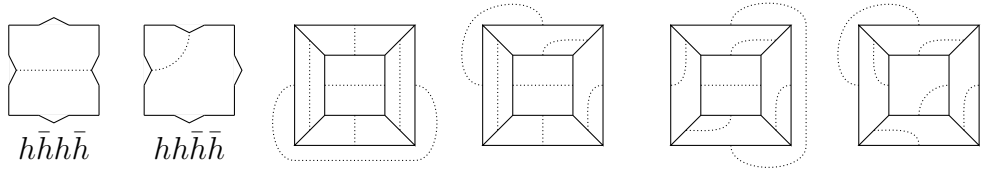

The cube from an $h\bar{h}r$-tiling can only have $h\bar{h}h\bar{h}$ as faces. This means the cube can only be the first of the four. For each face of the cube, we can take either of the two diagonals of $h\bar{h}h\bar{h}$ as $r$ to divide the quadrilateral face into two $h\bar{h}r$-triangles. By Figure 10 of \cite{cly}, there are seven non-equivalent $h\bar{h}r$-tilings, four of which are given by the first row of Figure \ref{tiling_hha3B}. We remark that the first and second tilings differ only in the choice of the diagonal for the central face. If one choice is $r$, then the other choice is $r^{-1}$.

\begin{figure}[htp]
\centering
\begin{tikzpicture}[>=latex,scale=1]



\begin{scope}
[decoration={markings,mark=at position 0.5 with {
\draw[white, thick] (0.12,0)--(-0.12,0);
\draw (0.12,0)--(0,0.1)--(-0.12,0);}}]


\foreach \a in {0,...,3}
\foreach \b in {1,-1}
{
\begin{scope}[xshift=3*\a cm, scale=\b]
			
\draw[postaction={decorate}]
	(1,1) -- (-1,1);				
\draw[postaction={decorate}]
	(1,1) -- (1,-1);
\draw[postaction={decorate}]
	(-0.4,0.4) -- (0.4,0.4);
\draw[postaction={decorate}]
	(0.4,-0.4) -- (0.4,0.4);
\draw[postaction={decorate}]
	(1,1) -- (0.4,0.4);
\draw[postaction={decorate}]
	(0.4,-0.4) -- (1,-1);
				
\end{scope}	
}


\begin{scope}[yshift=-3cm]

\foreach \a in {0,1}
\foreach \b in {1,-1}
{
\begin{scope}[xshift=3*\a cm, scale=\b]
			
\draw[postaction={decorate}]
	(1,1) -- (-1,1);				
\draw[postaction={decorate}]
	(1,1) -- (1,-1);
\draw[postaction={decorate}]
	(-0.4,0.4) -- (0.4,0.4);
\draw[postaction={decorate}]
	(0.4,-0.4) -- (0.4,0.4);
\draw[postaction={decorate}]
	(1,1) -- (0.4,0.4);
\draw[postaction={decorate}]
	(0.4,-0.4) -- (1,-1);
				
\end{scope}	
}


\begin{scope}[xshift=6cm]

\draw[postaction={decorate}]
	(-0.4,0.4) -- (0.4,0.4);				
\draw[postaction={decorate}]
	(0.4,-0.4) -- (-0.4,-0.4);
\draw[postaction={decorate}]
	(0.4,-0.4) -- (0.4,0.4);				
\draw[postaction={decorate}]
	(-0.4,0.4) -- (-0.4,-0.4);

\draw[postaction={decorate}]
	(0.4,0.4) -- (1,1);
\draw[postaction={decorate}]
	(0.4,-0.4) -- (1,-1);
\draw[postaction={decorate}]
	(-1,-1) -- (-0.4,-0.4);
\draw[postaction={decorate}]
	(-0.4,0.4) -- (-1,1);
					
\draw[postaction={decorate}]
	(-1,1) -- (1,1);				
\draw[postaction={decorate}]
	(-1,-1) -- (1,-1);
\draw[postaction={decorate}]
	(-1,-1) -- (-1,1);				
\draw[postaction={decorate}]
	(1,-1) -- (1,1);
						
\end{scope}


\begin{scope}[xshift=9cm]

\draw[postaction={decorate}]
	(-0.4,0.4) -- (0.4,0.4);				
\draw[postaction={decorate}]
	(-0.4,-0.4) -- (0.4,-0.4);
\draw[postaction={decorate}]
	(0.4,-0.4) -- (0.4,0.4);				
\draw[postaction={decorate}]
	(-0.4,-0.4) -- (-0.4,0.4);
				
\draw[postaction={decorate}]
	(0.4,0.4) -- (1,1);
\draw[postaction={decorate}]
	(-0.4,-0.4) -- (-1,-1);
\draw[postaction={decorate}]
	(-0.4,0.4) -- (-1,1);
\draw[postaction={decorate}]
	(0.4,-0.4) -- (1,-1);
	
\draw[postaction={decorate}]
	(-1,1) -- (1,1);				
\draw[postaction={decorate}]
	(-1,-1) -- (1,-1);
\draw[postaction={decorate}]
	(1,-1) -- (1,1);				
\draw[postaction={decorate}]
	(-1,-1) -- (-1,1);
							
\end{scope}

\end{scope}

\end{scope}


\begin{scope}
[decoration={markings,mark=at position 0.5 with {
\draw[white, thick] (0.12,0)--(-0.12,0);
\draw (0.12,0)--(0.06,0.06)--(-0.06,-0.06)--(-0.12,0);
}}]


\foreach \a in {0,1,3}
{
\begin{scope}[xshift=3*\a cm]

\draw[postaction={decorate}]
	(-0.4,0.4) -- (-1,-1);
\draw[postaction={decorate}]
	(0.4,-0.4) -- (-1,-1);	

\end{scope}
}

\foreach \a in {0,1}
{
\begin{scope}[xshift=3*\a cm]

\draw[postaction={decorate}]
	(-0.4,0.4) -- (1,1);
\draw[postaction={decorate}]
	(0.4,-0.4) -- (1,1);
	
\end{scope}
}

\foreach \a in {0,3}
\draw[postaction={decorate}, xshift=3*\a cm]
	(0.4,-0.4) -- (-0.4,0.4);

\foreach \a in {1,-1}
\draw[postaction={decorate}, xshift=6 cm, scale=\a]
	(-0.4,0.4) -- (1,1);

\foreach \a in {0,1,2}	
\draw[postaction={decorate}, xshift=3*\a cm]
	(1,1) -- (1.3,1.3);

\end{scope}


\begin{scope}
[decoration={markings,mark=at position 0.5 with {
\draw[white, thick] (0.12,0)--(-0.12,0);
\draw (0.12,0)--(0.06,-0.06)--(-0.06,0.06)--(-0.12,0);
\filldraw[fill=white] (0,0) circle (0.03);
}}]

\foreach \a in {1,2}
\draw[postaction={decorate}, xshift=3*\a cm]
	(-0.4,-0.4) -- (0.4,0.4);

\foreach \a in {1,-1}
\draw[postaction={decorate}, xshift=6 cm, scale=\a]
	(0.4,0.4) -- (1,-1);

\begin{scope}[xshift=9cm]

\draw[postaction={decorate}]
	(0.4,0.4) -- (-1,1);
\draw[postaction={decorate}]
	(0.4,0.4) -- (1,-1);
\draw[postaction={decorate}]
	(-1.3,1.3) -- (-1,1);
	
\end{scope}

\end{scope}


\foreach \a in {0,1,2}
\draw[xshift=3*\a cm]
	(-1,-1) -- (-1.3,-1.3);

\draw[xshift=9 cm]
	(1.3,-1.3) -- (1,-1);

\begin{scope}[yshift=-3cm]

\foreach \a in {0,1,3}
\draw[xshift=3*\a cm]
	(-0.4,0.4) -- (-1,-1) -- (0.4,-0.4);

\foreach \a in {0,1}
\draw[xshift=3*\a cm]
	(-0.4,0.4) -- (1,1) -- (0.4,-0.4)
	(1,1) -- (1.3,1.3)
	(-1,-1) -- (-1.3,-1.3);

\foreach \a in {0,3}
\draw[xshift=3*\a cm]
	(-0.4,0.4) -- (0.4,-0.4);

\draw[xshift=3 cm]
	(0.4,0.4) -- (-0.4,-0.4);

\draw[xshift=6 cm]
	(1.3,-1.3) -- (1,-1) -- (0.4,0.4) -- (-1,1) -- (-0.4,-0.4) -- (0.4,0.4)
	(-1,-1) -- (0.4,-0.4)
	(-1,1) -- (-1.3,1.3);

\draw[xshift=9 cm]
	(-1.3,1.3) -- (-1,1) -- (0.4,0.4) -- (1,-1) -- (1.3,-1.3);
				
\end{scope}
				
\end{tikzpicture}
\caption{Simple triangular subdivisions $S_{\triangle}P_6$ of types $h\bar{h}r$ and $h\bar{h}a$.}
\label{tiling_hha3B}
\end{figure}
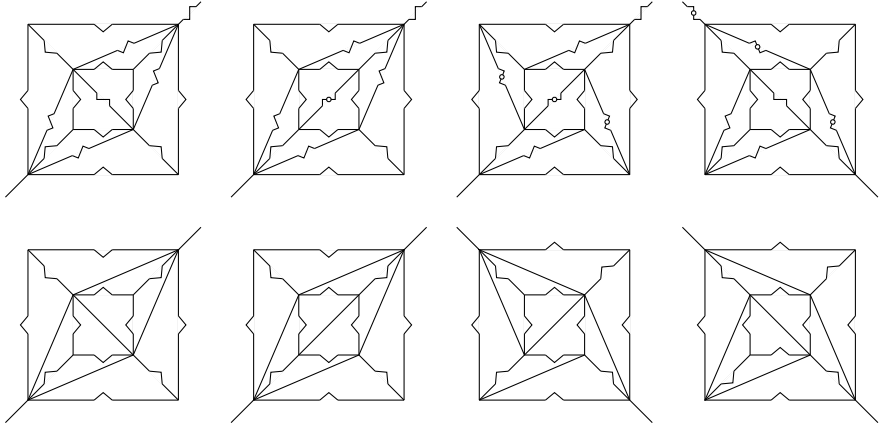

The cube from an $h\bar{h}a$-tiling can have both $h\bar{h}h\bar{h}$ and $hh\bar{h}\bar{h}$ as faces. This means the cube can be any of the four. For the face $h\bar{h}h\bar{h}$, we can take either of the two diagonals as $a$ to divide the quadrilateral face into two $h\bar{h}a$-triangles. For the face $hh\bar{h}\bar{h}$, we can only take one diagonal (that intersects the dotted line) as $a$ to divide the quadrilateral face into two $h\bar{h}a$-triangles.     

The second row of Figure \ref{tiling_hha3B} are four $h\bar{h}a$-tilings. The first and second are obtained from the first and second $h\bar{h}r$-tilings in the first row, by simply changing $r$ to $a$. The cube in the third $h\bar{h}a$-tiling is the second cube in Figure \ref{tiling_hha3A}. The cube has three $hh\bar{h}\bar{h}$ faces, and the choice of $a$ edge diagonals for the three faces is unique. The cube in the fourth $h\bar{h}a$-tiling is the fourth cube in Figure \ref{tiling_hha3A}. All faces of this cube are $hh\bar{h}\bar{h}$. The choice of $a$ edge diagonals is unique for all faces.

The simple triangular subdivision $S_{\triangle}P_6$ for $rr{\color{red} r'}$ and $rra$ completely parallels the subdivision for $h\bar{h}r$ and $h\bar{h}a$. The fourth of Figure \ref{pair} shows that the cube faces for $rr{\color{red} r'}$ are $rrrr$, and the cube faces for $rra$ are $rrrr$ and $rrr^{-1}r^{-1}$. On the left of Figure \ref{tiling_hha3C}, we draw the two quadrilaterals by using solid and dotted lines to represent $r$ and $r^{-1}$. This is a much easier way to visualise the cube, and there are exactly four cubes with $rrrr$ or $rrr^{-1}r^{-1}$ as faces, on the right of Figure \ref{tiling_hha3C}.

\begin{figure}[htp]
\centering
\begin{tikzpicture}[>=latex]

\begin{scope}[shift={(-4cm,0.2cm)}]


\begin{scope}
[decoration={markings,mark=at position 0.5 with {
\draw[white, thick] (0.2,0)--(-0.2,0);
\draw (0.2,0)--(0.1,0.1)--(-0.1,-0.1)--(-0.2,0);
}}]

\foreach \a in {0,...,3}
\draw[postaction={decorate}, rotate=90*\a]
	(0.6,-0.6) -- (0.6,0.6);

\foreach \a in {0,1}
\draw[postaction={decorate}, xshift=1.8cm, rotate=90*\a]
	(0.6,-0.6) -- (0.6,0.6);

\end{scope}

\foreach \a in {0,1}
\draw[densely dotted, xshift=1.8cm, rotate=90*\a]
	(-0.6,0.6) -- (-0.6,0.2) -- (-0.5,0.1) -- (-0.7,-0.1) -- (-0.6,-0.2) -- (-0.6,-0.6);

\filldraw[fill=white, xshift=1.8cm] 
	(-0.6,0) circle (0.05)
	(0,-0.6) circle (0.05);

\node at (0,-1) {$rrrr$};
\node at (1.8,-0.95) {$rrr^{-1}r^{-1}$};
	
\end{scope}


\foreach \a in {0,1,2,3}
\draw[rotate=90*\a]
	(-0.4,0.4) -- (0.4,0.4) -- (0.8,0.8) -- (-0.8,0.8);

\draw[xshift=2.5cm]
	(0.4,0.4) -- (0.8,0.8) 
	(0.8,0.8) rectangle (-0.8,-0.8)
	(0.8,-0.8) -- (0.4,-0.4) -- (0.4,0.4) -- (-0.4,0.4) -- (-0.8,0.8);

\draw[xshift=5cm]
	(0.8,-0.8) -- (-0.8,-0.8) -- (-0.8,0.8)
	(-0.8,-0.8) -- (-0.4,-0.4)
	(-0.4,0.4) -- (0.4,0.4) -- (0.4,-0.4) 
	(0.4,0.4) -- (0.8,0.8);

\draw[xshift=7.5cm]
	(-0.8,-0.8) -- (0.8,-0.8) -- (0.4,-0.4) -- (0.4,0.4) -- (-0.4,0.4) -- (-0.8,0.8) -- (-0.8,-0.8);
	
\begin{scope}[densely dotted]

\draw[xshift=2.5cm]
	(-0.4,0.4) -- (-0.4,-0.4) -- (0.4,-0.4)
	(-0.4,-0.4) -- (-0.8,-0.8);

\draw[xshift=5cm]
	(-0.4,0.4) -- (-0.4,-0.4) -- (0.4,-0.4) -- (0.8,-0.8) -- (0.8,0.8) -- (-0.8,0.8) -- cycle;

\draw[xshift=7.5cm]
	(-0.4,0.4) -- (-0.4,-0.4) -- (0.4,-0.4)
	(-0.4,-0.4) -- (-0.8,-0.8)
	(-0.8,0.8) -- (0.8,0.8) -- (0.8,-0.8)
	(0.4,0.4) -- (0.8,0.8);
			
\end{scope}
	
\end{tikzpicture}
\caption{Cubes with $rrrr$ and $rrr^{-1}r^{-1}$ faces.}
\label{tiling_hha3C}
\end{figure}
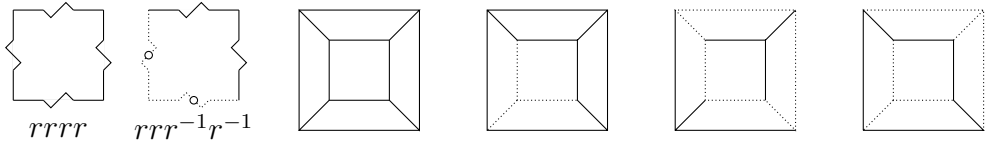

The $rr{\color{red} r'}$-tilings are similar to $h\bar{h}r$-tilings, and are the simple triangular subdivisions of the first of the four quadrilaterals in Figure \ref{tiling_hha3C}. There are seven non-equivalent subdivision tilings, four of which are given by the first row of Figure \ref{tiling_hha3D}.

\begin{figure}[htp]
\centering
\begin{tikzpicture}[>=latex,scale=1]



\begin{scope}
[decoration={markings,mark=at position 0.5 with {
\draw[white, thick] (0.12,0)--(-0.12,0);
\draw (0.12,0)--(0.06,0.06)--(-0.06,-0.06)--(-0.12,0);
}}]


\foreach \a in {0,1,2,3}
\foreach \b in {0,1,2,3}
{
\begin{scope}[xshift=3*\b cm, rotate=90*\a]

\draw[postaction={decorate}]
	(-0.4,0.4) -- (0.4,0.4);
\draw[postaction={decorate}]
	(0.4,0.4) -- (1,1);
\draw[postaction={decorate}]
	(-1,1) -- (1,1);

\end{scope}
}

\begin{scope}[red]

\foreach \a in {0,1}
\foreach \b in {1,-1}
{
\begin{scope}[xshift=3*\a cm, scale=\b]

\draw[postaction={decorate}]
	(-1,-1) -- (-0.4,0.4);
\draw[postaction={decorate}]
	(-1,-1) -- (0.4,-0.4);
	
\end{scope}
}

\foreach \a in {0,3}
\draw[postaction={decorate}, xshift=3*\a cm]
	(0.4,-0.4) -- (-0.4,0.4);

\foreach \a in {1,2}
\draw[postaction={decorate}, xshift=3*\a cm]
	(0.4,0.4) -- (-0.4,-0.4);

\foreach \a in {0,1,2}
\draw[postaction={decorate}, xshift=3*\a cm]
	(1,1) -- (1.3,1.3);
		
\begin{scope}[xshift=6 cm]

\foreach \a in {0,...,3}
\draw[postaction={decorate}, rotate=90*\a]
	(1,1) -- (-0.4,0.4);
	
\end{scope}

\begin{scope}[xshift=9 cm]
	
\draw[postaction={decorate}]
	(-1,-1) -- (-0.4,0.4);
\draw[postaction={decorate}]
	(-1,-1) -- (0.4,-0.4);
\draw[postaction={decorate}]
	(-1,1) -- (0.4,0.4);
\draw[postaction={decorate}]
	(1,-1) -- (0.4,0.4);
\draw[postaction={decorate}]
	(-1,1) -- (-1.3,1.3);
	
\end{scope}

\end{scope}


\begin{scope}[yshift=-3 cm]


\foreach \a in {0,1,2,3}
\foreach \b in {0,1}
{
\begin{scope}[xshift=3*\b cm, rotate=90*\a]

\draw[postaction={decorate}]
	(-0.4,0.4) -- (0.4,0.4);
\draw[postaction={decorate}]
	(0.4,0.4) -- (1,1);
\draw[postaction={decorate}]
	(-1,1) -- (1,1);

\end{scope}
}
	

\begin{scope}[xshift=6 cm]

\foreach \a in {0,1}
\draw[postaction={decorate}, rotate=90*\a]
	(0.4,0.4) -- (0.4,-0.4);

\foreach \a in {0,1,2}
\draw[postaction={decorate}, rotate=90*\a]
	(0.4,-0.4) -- (1,-1);
	
\foreach \a in {0,...,3}
\draw[postaction={decorate}, rotate=90*\a]
	(-1,1) -- (1,1);

\end{scope}


\begin{scope}[xshift=9 cm]

\foreach \a in {0,1}
{
\draw[postaction={decorate}, rotate=90*\a]
	(0.4,0.4) -- (0.4,-0.4);

\draw[postaction={decorate}, rotate=180*\a]
	(0.4,-0.4) -- (1,-1);

\draw[postaction={decorate}, rotate=90*\a]
	(-1,1) -- (-1,-1);
}

\end{scope}

\end{scope}

\end{scope}


\begin{scope}
[decoration={markings,mark=at position 0.5 with {
\draw[white, thick] (0.12,0)--(-0.12,0);
\draw (0.12,0)--(0.06,-0.06)--(-0.06,0.06)--(-0.12,0);
\filldraw[fill=white] (0,0) circle (0.03);
}}]


\begin{scope}[yshift=-3 cm]

\begin{scope}[xshift=6 cm]

\foreach \a in {0,1}
\draw[postaction={decorate}, rotate=90*\a]
	(-0.4,0.4) -- (-0.4,-0.4);

\draw[postaction={decorate}]
	(-0.4,-0.4) -- (-1,-1);
	
\end{scope}

\begin{scope}[xshift=9 cm]

\foreach \a in {0,1}
\draw[postaction={decorate}, rotate=90*\a]
	(-0.4,0.4) -- (-0.4,-0.4);

\draw[postaction={decorate}]
	(-0.4,-0.4) -- (-1,-1);

\foreach \a in {0,1}
\draw[postaction={decorate}, rotate=90*\a]
	(1,1) -- (1,-1);

\draw[postaction={decorate}]
	(0.4,0.4) -- (1,1);
		
\end{scope}

\end{scope}

\end{scope}


\foreach \a in {0,1,2}
\draw[xshift=3*\a cm, red]
	(-1,-1) -- (-1.3,-1.3);

\draw[xshift=9 cm, red]
	(1.3,-1.3) -- (1,-1);

\begin{scope}[yshift=-3cm]

\foreach \a in {0,1,3}
\draw[xshift=3*\a cm]
	(-0.4,0.4) -- (-1,-1) -- (0.4,-0.4);

\foreach \a in {0,1}
\draw[xshift=3*\a cm]
	(-0.4,0.4) -- (1,1) -- (0.4,-0.4)
	(1,1) -- (1.3,1.3)
	(-1,-1) -- (-1.3,-1.3);

\foreach \a in {0,3}
\draw[xshift=3*\a cm]
	(-0.4,0.4) -- (0.4,-0.4);

\draw[xshift=3 cm]
	(0.4,0.4) -- (-0.4,-0.4);

\draw[xshift=6 cm]
	(1.3,-1.3) -- (1,-1) -- (0.4,0.4)
	(-0.4,0.4) -- (1,1)
	(-0.4,0.4) -- (-1,-1) -- (0.4,-0.4) -- (-0.4,0.4)
	(-1,1) -- (-1.3,1.3);

\draw[xshift=9 cm]
	(-1.3,1.3) -- (-1,1) -- (0.4,0.4) -- (1,-1) -- (1.3,-1.3);
				
\end{scope}
				
\end{tikzpicture}
\caption{Simple triangular subdivisions $S_{\triangle}P_6$ of types $rr{\color{red} r'}$ and $rra$.}
\label{tiling_hha3D}
\end{figure}
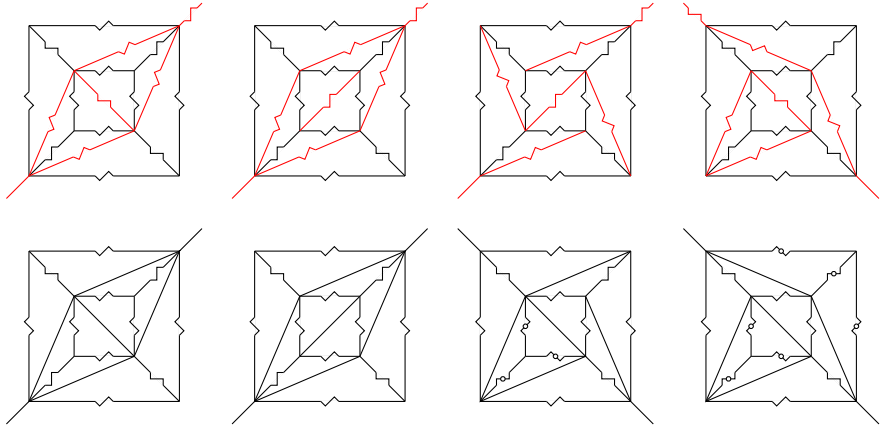

The $rra$-tilings are similar to $h\bar{h}a$-tilings, and are the simple triangular subdivisions of any of the four quadrilaterals in Figure \ref{tiling_hha3C}. The face $rrrr$ can be divided by either of the two diagonals, and the face $rrr^{-1}r^{-1}$ can be divided only by one of the two diagonals. The second row of Figure \ref{tiling_hha3D} are four $rra$-tilings. In the first and second tilings, all cube faces are $rrrr$, and there are exactly seven such simple triangular subdivisions. In the fourth tiling, all faces are $rrr^{-1}r^{-1}$, and there is only one simple triangular subdivision. 

\medskip

\noindent {\bf Case (1).} Earth map tiling $E_{\triangle}^I1$ and modifications.

\medskip

We know all the vertices $\text{AVC}
=\{2^23^2,1^{\frac{f}{4}}23,1^{\frac{f}{2}}\}$. We discuss $h\bar{h}a$-tilings. The proof for the other types are similar. 

Consecutive $1$s at a vertex are arranged as $\langle 1\langle 1\langle\cdots\langle 1\langle$. They determine the upper row of tiles on the left of Figure \ref{tiling_hha7A} (the picture shows four tiles). Moreover, each tile (say \circled{1}) in the upper row has a companion tile (say \circled{2}) sharing the $a$-edge. There are two possible ways of arranging \circled{2}. The $h$-edges of the companion tiles imply that all the companion tiles are arranged in the same way. Therefore the companion tiles form the lower row in the left or middle of Figure \ref{tiling_hha7A}. 

If $1^{\frac{f}{2}}$ is a vertex, then we get two possible versions of the earth map tiling $E_{\triangle}^I1$ (see Figures 5 of \cite{cly}). We note that \circled{1} and \circled{2} form one timezone, and the earth map tiling consists of $\frac{f}{2}$ timezones, and has $1^{\frac{f}{2}}$ as the pole vertices.

\begin{figure}[htp]
\centering
\begin{tikzpicture}[>=latex]

	
\foreach \a in {1,11}
\fill[gray!30, xshift=0.8*\a cm]
	(0,1) -- (0,0.7) -- (-0.1,0.5) -- (0,0.3) -- (0,-0.3) -- (-0.1,-0.5) -- (0,-0.7) -- (0,-1) -- (0.8,-1) -- (0.8,-0.7) -- (0.7,-0.5) -- (0.8,-0.3) -- (0.8,0.3) -- (0.7,0.5) -- (0.8,0.7) -- (0.8,1);

\fill[gray!30, xshift=4.8cm]
	(0,1) -- (0,0.7) -- (-0.1,0.5) -- (0,0.3) -- (0,-0.3) -- (0.1,-0.5) -- (0,-0.7) -- (0,-1) -- (0.8,-1) -- (0.8,-0.7) -- (0.9,-0.5) -- (0.8,-0.3) -- (0.8,0.3) -- (0.7,0.5) -- (0.8,0.7) -- (0.8,1);


\foreach \a in {0,...,14}
\draw[xshift=0.8*\a cm]
	(0,1) -- (0,0.7) -- (-0.1,0.5) -- (0,0.3) -- (0,0);

\foreach \a in {0,1,2,3,4,10,11,12,13,14}
\draw[xshift=0.8*\a cm]	
	(0,0) -- (0,-0.3) -- (-0.1,-0.5) -- (0,-0.7) -- (0,-1);

\foreach \a in {5,...,9}
\draw[xshift=0.8*\a cm]	
	(0,0) -- (0,-0.3) -- (0.1,-0.5) -- (0,-0.7) -- (0,-1);
	

\foreach \a in {10,...,13}
\draw[xshift=0.8*\a cm]
	(0,0) -- (0.2,0) -- (0.3,-0.1) -- (0.5,0.1) -- (0.6,0) -- (0.8,0);		


\draw
	(0,0) -- ++(3.2,0)
	(4,0) -- ++(3.2,0);

\foreach \x in {0,...,3}
{
\begin{scope}[xshift=0.8*\x cm, font=\tiny]

\node at (0.4,1) {1};
\node at (0.15,0.15) {2};
\node at (0.65,0.15) {3};

\node at (0.4,-1) {1};
\node at (0.15,-0.15) {2};
\node at (0.65,-0.15) {3};

\end{scope}
}

\node at (0,1.2) {\tiny $l_2$};
\node at (3.2,1.2) {\tiny $l_1$};

\node[inner sep=0.5, draw, shape=circle] at (1.2,0.5) {\tiny 1};
\node[inner sep=0.5, draw, shape=circle] at (1.2,-0.5) {\tiny 2};

\end{tikzpicture}
\caption{Earth map tiling $E_{\triangle}^I1$ for $h\bar{h}a$ and $h\bar{h}r$.}
\label{tiling_hha7A}
\end{figure}

The same argument can be applied to $h\bar{h}r$-tilings. The $r$-edge in $h\bar{h}r$ implies the companions (see the first of Figure \ref{pair}) can only be arranged as the right of Figure \ref{tiling_hha7A}. Therefore there is only one version of the earth map tiling $E_{\triangle}^I1$ for $h\bar{h}r$.  

The earth map tilings for $rra$ and $rr{\color{red} r'}$ are parallel to $h\bar{h}a$ and $h\bar{h}r$. If $1^{\frac{f}{2}}$ is a vertex, then we get two versions of the earth map tiling $E_{\triangle}^I1$ for $rra$, and only one version for $rr{\color{red} r'}$. See Figure \ref{tiling_hha7A}.

\begin{figure}[htp]
\centering
\begin{tikzpicture}[>=latex]

	
\foreach \a in {1,11}
\fill[gray!30, xshift=0.8*\a cm]
	(0,1) -- (0,0.7) -- (-0.1,0.6) -- (0.1,0.4) -- (0,0.3) -- (0,-0.3) -- (-0.1,-0.4) -- (0.1,-0.6) -- (0,-0.7) -- (0,-1) -- (0.8,-1) -- (0.8,-0.7) -- (0.9,-0.6) -- (0.7,-0.4) -- (0.8,-0.3) -- (0.8,0.3) -- (0.9,0.4) -- (0.7,0.6) -- (0.8,0.7) -- (0.8,1);

\fill[gray!30, xshift=4.8cm]
	(0,1) -- (0,0.7) -- (-0.1,0.6) -- (0.1,0.4) -- (0,0.3) -- (0,-0.3) -- (0.1,-0.4) -- (-0.1,-0.6) -- (0,-0.7) -- (0,-1) -- (0.8,-1) -- (0.8,-0.7) -- (0.7,-0.6) -- (0.9,-0.4) -- (0.8,-0.3) -- (0.8,0.3) -- (0.9,0.4) -- (0.7,0.6) -- (0.8,0.7) -- (0.8,1);


\foreach \a in {0,...,14}
\draw[xshift=0.8*\a cm]
	(0,1) -- (0,0.7) -- (-0.1,0.6) -- (0.1,0.4) -- (0,0.3) -- (0,0);

\foreach \a in {0,1,2,3,4,10,11,12,13,14}
\draw[xshift=0.8*\a cm]	
	(0,-1) -- (0,-0.7) -- (0.1,-0.6) -- (-0.1,-0.4) -- (0,-0.3) -- (0,0);

\foreach \a in {5,6,7,8,9}
{
\begin{scope}[xshift=0.8*\a cm]	

\draw
	(0,-1) -- (0,-0.7) -- (-0.1,-0.6) -- (0.1,-0.4) -- (0,-0.3) -- (0,0);
\filldraw[fill=white]
	(0,-0.5) circle (0.05);
	
\end{scope}
}

\foreach \a in {10,...,13}
\draw[xshift=0.8*\a cm, red]
	(0,0) -- (0.2,0) -- (0.3,-0.1) -- (0.5,0.1) -- (0.6,0) -- (0.8,0);		


\draw
	(0,0) -- ++(3.2,0)
	(4,0) -- ++(3.2,0);

\end{tikzpicture}
\caption{Earth map tiling $E_{\triangle}^I1$ for $rra$ and $rr{\color{red} r'}$.}
\label{tiling_hha7B}
\end{figure}

Suppose $1^{\frac{f}{2}}$ is not a vertex. Then $\text{AVC}=\{2^23^2,1^q23\}$, $q=\frac{f}{4}$, are all the vertices. Since the total numbers of the corners 1 and 2 in the tiling are the same, both must appear as vertices. We continue the argument for $h\bar{h}a$.

The vertex $1^q23$ is one fan $|3\langle 1\langle \cdots\langle 1\langle 2|$. The $1^q$ part of the fan induces a hemisphere tiling ${\mc H}$ consisting of $q$ timezones, like the left or middle of Figure \ref{tiling_hha7A}, with $1^q$ at both ends. We need to fill the rest of the tiling. 

On the left of Figure \ref{tiling_hha14A}, we draw the hemisphere ${\mc H}$ as the left of $l_1$ and the right of $l_2$. We need to fill the region on the right of $l_1$ and the left of $l_2$. The $23$ part of $1^q23$ already determines the tiles \circled{1} and \circled{2} in this region. Then the vertex $1_2\cdots=123\cdots=1^q23$. The $1^q$ part of this vertex induces another hemisphere tiling ${\mc H}$ that exactly fills the region (the picture shows the case $q=4$). Therefore the tiling is the union of two copies of ${\mc H}$.

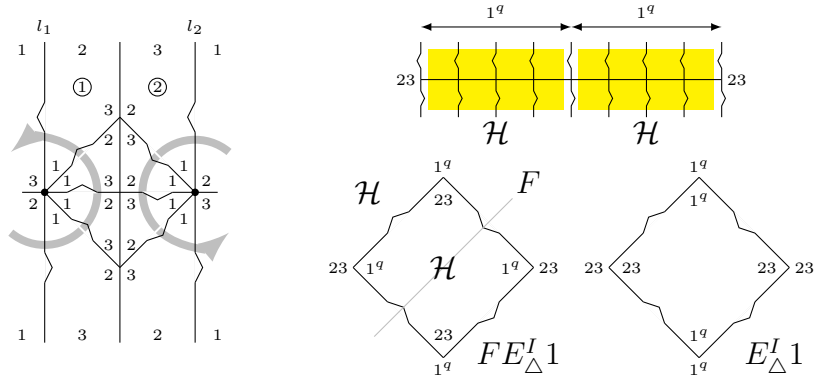
\begin{figure}[htp]
\centering
\begin{tikzpicture}[>=latex]


\begin{scope}[xshift=-6cm]

\foreach \c in {1,-1}
{
\begin{scope}[scale=\c]

\node at (0.85,0.35) {\tiny 1};
\node at (0.15,1.1) {\tiny 2};
\node at (0.5,1.9) {\tiny 3};

\node at (0.7,0.15) {\tiny 1};
\node at (0.15,0.15) {\tiny 2};
\node at (0.15,0.7) {\tiny 3};

\node at (0.85,-0.35) {\tiny 1};
\node at (0.15,-1.1) {\tiny 3};
\node at (0.5,-1.9) {\tiny 2};

\node at (0.7,-0.15) {\tiny 1};
\node at (0.15,-0.7) {\tiny 2};
\node at (0.15,-0.15) {\tiny 3};

\node at (1.3,1.9) {\tiny 1};
\node at (1.3,-1.9) {\tiny 1};
\node at (1.15,0.15) {\tiny 2};
\node at (1.15,-0.15) {\tiny 3};

\fill 
	(1,0) circle (0.05);

\end{scope}	
}

\node at (-1,2.2) {\tiny $l_1$};
\node at (1,2.2) {\tiny $l_2$};

\node[inner sep=0.5, draw, shape=circle] at (-0.5,1.4) {\tiny 1};
\node[inner sep=0.5, draw, shape=circle] at (0.5,1.4) {\tiny 2};

\draw
	(1,0) -- (1.3,0)
	(-1,0) -- (-1.3,0);

\draw
	(0,-2) -- (0,2);	

\foreach \b in {-1,1}
\draw[->, line width=3, gray!50, xshift=\b cm, scale=\b]
	(50:0.7) arc (50:310:0.7);


\foreach \c in {1,-1}
{
\begin{scope}
[decoration={markings,mark=at position 0.5 with {
\draw[white, thick] (0.2,0)--(-0.2,0);
\draw (0.2,0)--(0,0.1)--(-0.2,0);}}, scale=\c]

\foreach \a in {1,-1}
\draw[postaction={decorate}]
	(1,0) -- ++(0,2*\a);

\foreach \a in {1,0,-1}
\draw[postaction={decorate}]
	(1,0) -- (0,\a);
			
\end{scope}	
}

\end{scope}

\begin{scope}[yshift=1.5 cm]

\foreach \b in {1,-1}
\fill[yellow, scale=\b]
	(0.1,-0.4) rectangle (1.9,0.4);

\foreach \a in {-4,...,4}
\foreach \b in {1,-1}
\draw[xshift=0.5*\a cm, scale=\b]
	(0,0) -- (0,0.15) -- (-0.05,0.25) -- (0,0.35) -- (0,0.5);
	
\draw
	(-2,0) -- (2,0);	

\foreach \a in {1,-1}
{
\begin{scope}[xscale=\a]
	
\draw[<->]
	(0,0.7) -- (2,0.7);
	
\node at (1,0.9) {\tiny $1^q$};

\node at (1,-0.7) {${\mc H}$};

\end{scope}
}

\node at (2.2,0) {\tiny $23$};
\node at (-2.2,0) {\tiny $23$};

\end{scope}

\begin{scope}[yshift=-1cm]

\begin{scope}
[decoration={markings,mark=at position 0.5 with {
\draw[white, thick] (0.2,0)--(-0.2,0);
\draw (0.2,0)--(0,0.1)--(-0.2,0);}}]

\foreach \x in {-1,1}
\foreach \b in {-1,1}
{
\begin{scope}[xshift=1.7*\x cm, scale=\b]

\draw[postaction={decorate}]
	(1.2,0) -- (0,1.2);
\draw[postaction={decorate}]
	(-1.2,0) -- (0,1.2);

\end{scope}
}
				
\end{scope}

\foreach \b in {-1,1}
\foreach \x in {-1,1}
{
\begin{scope}[xshift=1.7*\x cm, scale=\b]

\node at (0,1.35) {\tiny $1^q$};
\node at (1.4,0) {\tiny $23$};

\end{scope}
}

\node at (2.7,-1.2) {$E_{\triangle}^I1$};

\begin{scope}[xshift=-1.7cm]

\node at (0.9,0) {\tiny $1^q$};
\node at (0,0.9) {\tiny $23$};
\node at (-0.9,0) {\tiny $1^q$};
\node at (0,-0.9) {\tiny $23$};

\draw[gray!50]
	(45:1.3) -- (45:-1.3);

\node at (45:1.6) {$F$};

\node at (0,0) {${\mc H}$};
\node at (-1,1) {${\mc H}$};

\node at (1,-1.2) {$FE_{\triangle}^I1$};

\end{scope}

\begin{scope}[xshift=1.7cm]

\node at (0,0.9) {\tiny $1^q$};
\node at (0.9,0) {\tiny $23$};
\node at (0,-0.9) {\tiny $1^q$};
\node at (-0.9,0) {\tiny $23$};

\end{scope}

\end{scope}

\end{tikzpicture}
\caption{Flip modification $FE_{\triangle}^I1$ for $h\bar{h}a$.}
\label{tiling_hha14A}
\end{figure}

We remark that, on the left of Figure \ref{tiling_hha14A}, we actually choose the hemisphere tiling ${\mc H}$ on the left of $l_1$ and the right of $l_2$ to be the middle of Figure \ref{tiling_hha7A}. In fact, the $h$-edge $\overline{12}_1$ of \circled{2} determines (as indicated by the thick gray arrows) all the $h$-edges of the tiles corresponding to the $1^q$ part of $1_2\cdots=1^q23$. This implies that the hemisphere tiling ${\mc H}$ can only be the middle of Figure \ref{tiling_hha7A}, and cannot be the left one.

On the right of Figure \ref{tiling_hha14A}, we interpret the tiling just obtained. On the upper right of Figure \ref{tiling_hha14A}, we show that the earth map tiling $E_{\triangle}^I1$ is also a union of two ${\mc H}$. On the lower right of Figure \ref{tiling_hha14A}, we compare the common boundary of the two ${\mc H}$ and the corner combinations along the common boundary, the first for the tiling on the left of Figure \ref{tiling_hha14A}, and the second for the earth map tiling $E_{\triangle}^I1$. We find the two are related by the flip of interior ${\mc H}$ with respect to the gray line. Therefore the tiling on the left of Figure \ref{tiling_hha14A} is the flip modification $FE_{\triangle}^I1$ (see Figures 20 of \cite{cly}).

The argument for $\text{AVC}
=\{2^23^2,1^q23\}$ and $h\bar{h}a$ also applies to $h\bar{h}r$, $rra$ and $rr{\color{red} r'}$. We get tilings in Figure \ref{tiling_hha14B} similar to the left of Figure \ref{tiling_hha14A}. Then we may interpret the structure of the tilings similar to the right of Figure \ref{tiling_hha14A}. We find that the first tiling in Figure \ref{tiling_hha14B}, for $h\bar{h}r$, is still the flip modification $FE_{\triangle}^I1$ described in Figure \ref{tiling_hha14A}. However, the second the third tilings in Figure \ref{tiling_hha14B}, for $rra$ and $rr{\color{red} r'}$, is the rotation modification $RE_{\triangle}^I1$ instead of the flip modification. The rotation modification is described on the right of Figure \ref{tiling_hha14B}, where the structure of the earth map tiling $E_{\triangle}^I1$ is given. We need to exchange $1^q$ and $23$ in the interior ${\mc H}$. The direction of the $r$-edge implies that a flip cannot be applied because it would change $r$ on the boundary to $r^{-1}$. Therefore the exchange of $1^q$ and $23$ can only be described as the rotation modification $RE_{\triangle}^I1$. 

\begin{figure}[htp]
\centering
\begin{tikzpicture}[>=latex]



\foreach \c in {1,-1}
{
\begin{scope}
[decoration={markings,mark=at position 0.5 with {
\draw[white, thick] (0.2,0)--(-0.2,0);
\draw (0.2,0)--(0,0.1)--(-0.2,0);}},  scale=\c]

\foreach \a in {1,-1}
\draw[postaction={decorate}]
	(1,0) -- ++(0,2*\a);

\foreach \a in {1,0,-1}
\draw[postaction={decorate}]
	(1,0) -- (0,\a);
				
\end{scope}	
}


\foreach \y in {-2,-1,0,1}
{
\draw[yshift=\y cm]
	(0,0) -- (0,0.3) -- (-0.1,0.4) -- (0.1,0.6) -- (0,0.7) -- (0,1);

\filldraw[fill=white] (0,0.5+\y) circle (0.05);
}


\foreach \a in {0,1,2}
\draw[xshift=3.2*\a cm]
	(1,0) -- (1.3,0)
	(-1,0) -- (-1.3,0);	

\draw
	(3.2,-2) -- ++(0,4);



\foreach \c in {1,-1}
\foreach \a in {1,2}
{
\begin{scope}
[decoration={markings,mark=at position 0.5 with {
\draw[white,thick](0.2,0)--(-0.2,0);
\draw (0.2,0)--(0.1,0.1)--(-0.1,-0.1)--(-0.2,0);}},
xshift=3.2*\a cm, scale=\c]

\foreach \a in {1,-1}
\draw[postaction={decorate}]
	(1,0) -- ++(0,2*\a);

\foreach \a in {1,0,-1}
\draw[postaction={decorate}]
	(1,0) -- (0,\a);

\end{scope}	
}

\foreach \y in {-2,-1,0,1}
\draw[xshift=6.4 cm, yshift=\y cm, red]
	(0,0) -- (0,0.3) -- (0.1,0.4) -- (-0.1,0.6) -- (0,0.7) -- (0,1);

	
\begin{scope}[xshift=9.8cm]

\draw[gray!50, ->]
	(0:0.5) arc (0:90:0.5);

\node at (-0.15,0.5) {$R$};
\node at (0,0) {${\mc H}$};
\node at (-1,1) {${\mc H}$};

\node at (0,0.9) {\tiny $1^q$};
\node at (0.9,0) {\tiny $23$};
\node at (0,-0.9) {\tiny $1^q$};
\node at (-0.9,0) {\tiny $23$};

\node at (0,1.35) {\tiny $1^q$};
\node at (1.4,0) {\tiny $23$};
\node at (0,-1.35) {\tiny $1^q$};
\node at (-1.4,0) {\tiny $23$};

\begin{scope}[decoration={markings,mark=at position 0.5 with {
\draw[white,thick](0.2,0)--(-0.2,0);
\draw (0.2,0)--(0.1,0.1)--(-0.1,-0.1)--(-0.2,0);}}]

\foreach \b in {0,1,2,3}
\draw[postaction={decorate}, rotate=90*\b]
	(1.2,0) -- (0,1.2);
		
\end{scope}

\node at (1,-1.2) {$RE_{\triangle}^I1$};

\end{scope}
	
\end{tikzpicture}
\caption{Modification of $E_{\triangle}^I1$ for $h\bar{h}r$, $rra$, $rr{\color{red} r'}$.}
\label{tiling_hha14B}
\end{figure}
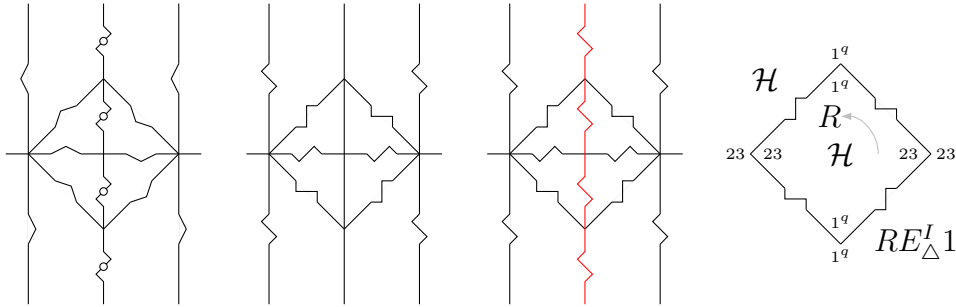

\medskip

\noindent {\bf Case (2).} Earth map tiling $E_{\triangle}^J1$ and modifications.

\medskip

We know all the vertices $\text{AVC}
=\{1^223,12^{\frac{f+4}{8}}3^{\frac{f+4}{8}},2^{\frac{f}{4}}3^{\frac{f}{4}}\}$. We discuss $h\bar{h}a$-tilings. The proof for the other types are similar. 

A fan $|2\rangle 3|$ determines tiles \circled{1} and \circled{2} on the left of Figure \ref{tiling_hha5A}. Then the vertex $1_11_2\cdots=1^223$ is one fan $|2\rangle 1\rangle 1\rangle 3|$. Therefore one fan $|2\rangle 3|$ determines a timezone consisting of \circled{1}, \circled{2}, \circled{3}, \circled{4}. 

The vertex $2^{\frac{f}{4}}3^{\frac{f}{4}}$ consists of $\frac{f}{4}$ fans $|2\rangle 3|$. These fans induce $\frac{f}{4}$ timezones. The timezones fit together only in the way in Figure \ref{tiling_hha5A}. Then we get the earth map tiling $E_{\triangle}^J1$ consisting of $\frac{f}{4}$ timezones (see Figure 5 of \cite{cly}). 

\begin{figure}[htp]
\centering
\begin{tikzpicture}[>=latex]

\fill[gray!30, xshift=1.2cm]
	(0,-1) -- (0,-0.7) -- (0.1,-0.5) -- (0,-0.3) -- (0,1) -- (1.2,1) -- (1.2,-0.3) -- (1.3,-0.5) -- (1.2,-0.7) -- (1.2,-1);

\fill[gray!30, xshift=7.2cm]
	(0,-1) -- (0,-0.7) -- (0.1,-0.5) -- (0,-0.3) -- (0,0.3) -- (0.1,0.4) -- (-0.1,0.6) -- (0,0.7) -- (0,1) -- (1.2,1) -- (1.2,0.7) -- (1.1,0.6) -- (1.3,0.4) -- (1.2,0.3) -- (1.2,-0.3) -- (1.3,-0.5) -- (1.2,-0.7) -- (1.2,-1);
	

\foreach \x in {1,3,5,7}
{
\begin{scope}[xshift=0.6*\x cm]

\node at (0.3,1) {\tiny 3};
\node at (0.1,0.15) {\tiny 1};
\node at (0.5,0.15) {\tiny 2};

\node at (-0.3,1) {\tiny 2};
\node at (-0.1,0.15) {\tiny 1};
\node at (-0.5,0.15) {\tiny 3};

\node at (0.3,-1) {\tiny 2};
\node at (0.1,-0.15) {\tiny 3};
\node at (0.5,-0.15) {\tiny 1};

\node at (-0.3,-1) {\tiny 3};
\node at (-0.1,-0.15) {\tiny 2};
\node at (-0.5,-0.15) {\tiny 1};

\end{scope}
}

\node at (4.8,1.2) {\tiny $l_1$};
\node at (0,1.2) {\tiny $l_2$};

\node[inner sep=0.5, draw, shape=circle] at (1.5,0.5) {\tiny 1};
\node[inner sep=0.5, draw, shape=circle] at (2.1,0.5) {\tiny 2};
\node[inner sep=0.5, draw, shape=circle] at (1.5,-0.5) {\tiny 3};
\node[inner sep=0.5, draw, shape=circle] at (2.1,-0.5) {\tiny 4};

\foreach \b in {0,1}
{
\begin{scope}[xshift=6*\b cm]

\foreach \a in {1,3,5,7}
\draw[xshift=0.6*\a cm]
	(0,1) -- (0,0.7) -- (0.1,0.5) -- (0,0.3) -- (0,0);

\foreach \a in {0,2,4,6,8}
\draw[xshift=0.6*\a cm]
	(0,-1) -- (0,-0.7) -- (0.1,-0.5) -- (0,-0.3) -- (0,0);

\foreach \a in {0,2,4,6}
\draw[xshift=0.6*\a cm]
	(0,0) -- (0.15,0) -- (0.3,0.1) -- (0.45,0) -- (0.75,0) -- (0.9,-0.1) -- (1.05,0) -- (1.2,0);

\end{scope}	
}
	
\foreach \a in {0,2,4,6,8}
\draw
	(0.6*\a,0) -- ++(0,1);

\foreach \a in {1,3,5,7}
\draw
	(0.6*\a,0) -- ++(0,-1);

\begin{scope}[xshift=6 cm]

\foreach \a in {0,2,4,6,8}
{
\draw[xshift=0.6*\a cm]
	(0,0) -- (0,0.3) -- (0.1,0.4) -- (-0.1,0.6) -- (0,0.7) -- (0,1);
}

\foreach \a in {1,3,5,7}
{
\draw[xshift=0.6*\a cm]
	(0,0) -- (0,-0.3) -- (0.1,-0.4) -- (-0.1,-0.6) -- (0,-0.7) -- (0,-1);
\filldraw[fill=white] (0.6*\a,-0.5) circle (0.05);
}

\end{scope}
			
\end{tikzpicture}
\caption{Earth map tiling $E_{\triangle}^J1$ for $h\bar{h}a$ and $h\bar{h}r$.}
\label{tiling_hha5A}
\end{figure}

The same argument can be applied to $h\bar{h}r$, $rra$, $rr{\color{red} r'}$. We get the earth map tilings $E_{\triangle}^J1$ for $h\bar{h}r$ on the right of Figure \ref{tiling_hha5A}. We also get $E_{\triangle}^J1$ for $rra$ and $rr{\color{red} r'}$ in Figure \ref{tiling_hha5B}.

\begin{figure}[htp]
\centering
\begin{tikzpicture}[>=latex]

\fill[gray!30, xshift=1.2cm]
	(0,-1) -- (0,-0.7) -- (0.1,-0.6) -- (-0.1,-0.4) -- (0,-0.3) -- (0,1) -- (1.2,1) -- (1.2,-0.3) -- (1.1,-0.4) -- (1.3,-0.6) -- (1.2,-0.7) -- (1.2,-1);

\foreach \a in {0,1}
\foreach \b in {0,-1}
\fill[gray!30, xshift=7.2cm+0.6*\a cm, yshift=\b cm]
	(0,1) -- (0,0.7) -- (-0.1,0.6) -- (0.1,0.4) -- (0,0.3) -- (0,0) -- (0.6,0) -- (0.6,0.3) -- (0.7,0.4) -- (0.5,0.6) -- (0.6,0.7) -- (0.6,1);
	

\foreach \b in {0,1}
{
\begin{scope}[xshift=6*\b cm]

\foreach \a in {1,3,5,7}
\draw[xshift=0.6*\a cm]
	(0,1) -- (0,0.7) -- (-0.1,0.6) -- (0.1,0.4) -- (0,0.3) -- (0,0);

\foreach \a in {0,2,4,6,8}
\draw[xshift=0.6*\a cm]
	(0,-1) -- (0,-0.7) -- (0.1,-0.6) -- (-0.1,-0.4) -- (0,-0.3) -- (0,0);

\foreach \a in {0,...,7}
\draw[xshift=0.6*\a cm]
	(0,0) -- (0.1,0) -- (0.2,-0.1) -- (0.4,0.1) -- (0.5,0) -- (0.6,0);

\end{scope}	
}

\foreach \a in {0,2,4,6,8}
\draw
	(0.6*\a,0) -- ++(0,1);

\foreach \a in {1,3,5,7}
\draw
	(0.6*\a,0) -- ++(0,-1);

\begin{scope}[xshift=6 cm]

\foreach \a in {0,2,4,6,8}
{
\draw[xshift=0.6*\a cm, red]
	(0,0) -- (0,0.3) -- (0.1,0.4) -- (-0.1,0.6) -- (0,0.7) -- (0,1);
}

\foreach \a in {1,3,5,7}
\draw[xshift=0.6*\a cm, red]
	(0,0) -- (0,-0.3) -- (-0.1,-0.4) -- (0.1,-0.6) -- (0,-0.7) -- (0,-1);

\end{scope}

\end{tikzpicture}
\caption{Earth map tiling $E_{\triangle}^J1$ for $rra$ and $rr{\color{red} r'}$.}
\label{tiling_hha5B}
\end{figure}

Suppose $1^{\frac{f}{2}}$ is not a vertex. Then $\text{AVC}=\{1^223,12^{q+1}3^{q+1}\}$, $q=\frac{f-4}{8}$, are all the vertices. Since the total numbers of the corners 1 and 2 in the tiling are the same, both must appear as vertices. We continue the argument for $h\bar{h}a$.

The vertex $12^{q+1}3^{q+1}$ consists of one fan $|2\rangle 1\rangle 3|$ and $q$ fans $|2\rangle 3|$. The $q$ fans $|2\rangle 3|$ determine a partial tiling ${\mc K}$ consisting of $q$ timezones on the left of Figure \ref{tiling_hha5A}. The boundaries of ${\mc K}$ are $l_1,l_2$. We need to fill the rest of the tiling. 

On the left of Figure \ref{tiling_hha8}, we draw the partial tiling ${\mc K}$ as the left of $l_1$ and the right of $l_2$. We need to fill the region on the right of $l_1$ and the left of $l_2$. The fan $|2\rangle 1\rangle 3|$ at $12^{q+1}3^{q+1}$ determines tiles \circled{1}, \circled{2}, \circled{3} in two possible ways. We will discuss the other way later.

\begin{figure}[htp]
\centering
\begin{tikzpicture}[>=latex]


\begin{scope}[font=\tiny]

\node at (-1.3,1.9) {3};
\node at (-1.15,0.15) {2};
\node at (-1.3,-1.9) {2};
\node at (-1.15,-0.15) {1};

\node at (2.3,1.9) {2};
\node at (2.15,0.15) {3};
\node at (2.3,-1.9) {3};
\node at (2.15,-0.15) {1};

\node at (0.5,1.9) {1};
\node at (1.5,1.9) {3};
\node at (-0.5,1.9) {2};

\node at (-0.5,-1.9) {1};
\node at (0.5,-1.9) {3};
\node at (1.5,-1.9) {2};

\node at (-0.15,1.3) {1};
\node at (-0.85,0.4) {3};
\node at (0.15,1.3) {3};
\node at (0.85,0.4) {2};
\node at (1.15,0.15) {1};
\node at (1.85,0.15) {2};

\node at (-0.15,-1.3) {2};
\node at (-0.85,-0.4) {3};
\node at (0.15,-1.3) {1};
\node at (0.85,-0.4) {2};
\node at (1.15,-0.15) {3};
\node at (1.85,-0.15) {1};

\node at (-0.15,-0.85) {3};
\node at (0.1,-0.9) {1};
\node at (-0.15,-0.5) {1};
\node at (0.15,-0.5) {2};

\node at (0.07,0.95) {2};
\node at (-0.15,0.85) {1};
\node at (-0.15,0.5) {3};
\node at (0.15,0.5) {1};

\node at (-0.6,0.3) {2};
\node at (0.6,0.35) {3};
\node at (-0.6,0) {3};
\node at (0.6,0) {2};
\node at (-0.6,-0.3) {2};
\node at (0.6,-0.3) {3};

\end{scope}

\begin{scope}[xshift=6cm, font=\tiny]

\node at (-1.3,1.9) {3};
\node at (-1.15,0.15) {2};
\node at (-1.3,-1.9) {2};
\node at (-1.15,-0.15) {1};

\node at (2.3,1.9) {2};
\node at (2.15,0.15) {3};
\node at (2.3,-1.9) {3};
\node at (2.15,-0.15) {1};

\node at (0.5,1.9) {1};
\node at (1.5,1.9) {3};
\node at (-0.5,1.9) {2};

\node at (-0.5,-1.9) {3};
\node at (0.5,-1.9) {2};
\node at (1.5,-1.9) {1};

\node at (-0.85,0.15) {3};
\node at (-0.85,-0.15) {1};
\node at (-0.15,0.15) {1};
\node at (-0.15,-0.15) {2};

\node at (0.15,0.4) {3};
\node at (0.15,-0.4) {3};
\node at (0.4,0.3) {2};
\node at (0.4,0) {3};
\node at (0.4,-0.3) {2};

\node at (1.85,0.4) {2};
\node at (1.85,-0.4) {2};
\node at (1.6,0.3) {3};
\node at (1.6,0) {2};
\node at (1.6,-0.35) {3};

\node at (0.85,1.3) {2};
\node at (1.15,1.3) {1};
\node at (0.85,-1.3) {1};
\node at (1.15,-1.3) {3};

\node[inner sep=0.5, draw, shape=circle] at (1.5,-1.4) {\tiny 9};

\end{scope}

\foreach \x in {0,1}
{
\begin{scope}[xshift=6*\x cm]

\node at (-1,2.2) {\tiny $l_1$};
\node at (2,2.2) {\tiny $l_2$};

\node[inner sep=0.5, draw, shape=circle] at (-0.5,1.4) {\tiny 1};
\node[inner sep=0.5, draw, shape=circle] at (0.5,1.4) {\tiny 2};
\node[inner sep=0.5, draw, shape=circle] at (1.5,1.4) {\tiny 3};

\end{scope}
}

\node[inner sep=0.5, draw, shape=circle] at (-0.5,-1.4) {\tiny 7};
\node[inner sep=0.5, draw, shape=circle] at (0.5,-1.4) {\tiny 10};
\node[inner sep=0.5, draw, shape=circle] at (1.5,-1.4) {\tiny 9};

\node[inner sep=0.5, draw, shape=circle] at (0.17,0.73) {\tiny 5};
\node[inner sep=0.5, draw, shape=circle] at (-0.35,0.5) {\tiny 4};
\node[inner sep=0.5, draw, shape=circle] at (-0.35,-0.5) {\tiny 6};
\node[inner sep=0.5, draw, shape=circle] at (0.17,-0.73) {\tiny 8};

\foreach \x in {0,7}
\fill
	(-1+\x,0) circle (0.05)
	(1+\x,0) circle (0.05);


\begin{scope}
[decoration={markings,mark=at position 0.5 with {
\draw[white, thick] (0.2,0)--(-0.2,0);
\draw (0.2,0)--(0,-0.1)--(-0.2,0);}}]


\foreach \y in {-1,2}
\draw[postaction={decorate}]
	(\y,-2) -- ++(0,2);

\draw[postaction={decorate}]
	(1,0) -- (1,2);
	
\draw[postaction={decorate}]
	(1,0) -- (2,0);
	
\foreach \y in {-0.4,1.2,-2}
\draw[postaction={decorate}]
	(0,\y) -- ++(0,0.8);

\foreach \y in {-0.4,1.2}
{
\draw[postaction={decorate}]
	(0,\y) -- ++(0,-0.8);
\draw[postaction={decorate}]
	(0,\y) -- (-1,0);
\draw[postaction={decorate}]
	(1,0) -- (0,-\y);
}


\begin{scope}[xshift=6cm]

\foreach \y in {-1,2}
\draw[postaction={decorate}]
	(\y,-2) -- ++(0,2);

\draw[postaction={decorate}]
	(0,0) -- (0,2);
	
\draw[postaction={decorate}]
	(0,0) -- (-1,0);

\foreach \y in {-0.4,1.2,-2}
\draw[postaction={decorate}]
	(1,\y) -- ++(0,0.8);

\foreach \y in {-0.4,1.2}
{
\draw[postaction={decorate}]
	(1,\y) -- ++(0,-0.8);
\draw[postaction={decorate}]
	(0,0) -- (1,-\y);
\draw[postaction={decorate}]
	(1,\y) -- (2,0);

}
		
\end{scope}
					
\end{scope}	


\draw
	(-1,2) -- (-1,0) -- (-1.5,0)
	(2,2) -- (2,0) -- (2.5,0)
	(1,0) -- (1,-2);

\foreach \y in {0.4,-1.2}
\draw
	(0,\y) -- (-1,0)
	(0,-\y) -- (1,0);

\begin{scope}[xshift=6cm]

\draw
	(-1,2) -- (-1,0) -- (-1.5,0)
	(2,2) -- (2,0) -- (2.5,0)
	(0,0) -- (0,-2);

\foreach \y in {0.4,-1.2}
\draw
	(1,-\y) -- (0,0)
	(1,\y) -- (2,0);
		
\end{scope}

\end{tikzpicture}
\caption{Modification of $E_{\triangle}^J1$ for $h\bar{h}a$.}
\label{tiling_hha8}
\end{figure}

The vertex $1_13_2\cdots=1^223$ or $12^{q+1}3^{q+1}$. In the left picture, we assume $1_13_2\cdots=1^223$. The vertex is one fan $|2\rangle 1\rangle 1\rangle 3|$, and determines \circled{4} and \circled{5}. Then $3_12_4\cdots=12^23\cdots=12^{q+1}3^{q+1}$. The vertex consists of one fan $|2\rangle 1\rangle 3|$ and $q$ fans $|2\rangle 3|$, which determine \circled{7} and a copy of ${\mc K}$ consisting of \circled{1}, \circled{2}, \circled{4}, \circled{5}, \circled{6}, \circled{8} and the tiles between them (the picture shows the case $q=2$). Then the fans at the vertex $1_32_23_53_8\cdots=12^{q+1}3^{q+1}$ determine \circled{9} and \circled{\footnotesize 10}. This fills the region on the right of $l_1$ and the left of $l_2$.

In the right of Figure \ref{tiling_hha8}, we assume $1_13_2\cdots=12^{q+1}3^{q+1}$. Then the one fan $|2\rangle 1\rangle 3|$ and $q$ fans $|2\rangle 3|$ at this vertex, together with the $h$-edge of the tile \circled{3}, determine all the tiles on the right of $l_1$ and the left of $l_2$.

Now consider the other way of arranging the fan $|2\rangle 1\rangle 3|$ at $12^{q+1}3^{q+1}$, with the corner 3 in the fan in \circled{1} and corner 2 in \circled{3}. Then we carry out the similar argument for the two cases of $1_12_2\cdots=1^223$ or $12^{q+1}3^{q+1}$. If $1_12_2\cdots=1^223$, then we get the same \circled{4}, \circled{5}, \circled{6}, \circled{8} and the tiles between them, such that corners 2 and 3 are switched in all these tiles. This implies the same \circled{9}, \circled{\footnotesize 10}, again with corners 2 and 3 switched. Then we find \circled{7} has two $\bar{h}$-edges, and \circled{9} has two $h$-edges. Both are contradictions. The assumption $1_12_2\cdots=12^{q+1}3^{q+1}$ leads to similar contradiction. Therefore the tilings in Figure \ref{tiling_hha8} are the only ones for $\text{AVC}=\{1^223,12^{q+1}3^{q+1}\}$ and $h\bar{h}a$.

Next we interpret the tilings in Figure \ref{tiling_hha8}. At the top of Figure \ref{tiling_hha9}, we find the earth map tiling $E_{\triangle}^J1$ can be decomposed into two copies of hemisphere tiling ${\mc H}$. Although $E_{\triangle}^J1$ is also a union of two ${\mc K}$, we choose not to use this viewpoint.

\begin{figure}[htp]
\centering
\begin{tikzpicture}[>=latex]
	

\begin{scope}[yshift=2.6cm]

\fill[yellow]
	(0.1,0.4) -- (2.9,0.4) -- (2.9,0.05) -- (2.4,0.05) -- (2.4,-0.4) -- (0.6,-0.4) -- (0.6,0.05) -- (0.1,0.05) 
	(-0.1,0.4) -- (-1.9,0.4) -- (-1.9,-0.05) -- (-2.4,-0.05) -- (-2.4,-0.4) -- (0.4,-0.4) -- (0.4,-0.05) -- (-0.1,-0.05);
	
\foreach \a in {-2,...,3}
\foreach \b in {0,1}
\draw[xshift=\a cm-0.5*\b cm, yshift=-0.5*\b cm]
	(0,0) -- (0,0.15) -- (0.05,0.25) -- (0,0.35) -- (0,0.5);
	
\foreach \a in {-2,...,2}
\draw[xshift=\a cm]	
	(0,-0.5) -- (0,0) -- (0.15,0) -- (0.25,-0.05) -- (0.35,0) -- (0.5,0) -- (0.5,0.5);

\foreach \a in {-2,...,3}
\draw[xshift=\a cm]	
	(0,0) -- (-0.15,0) -- (-0.25,0.05) -- (-0.35,0) -- (-0.5,0);

\draw[<->]
	(0,0.7) -- (3,0.7);
\draw[<->]
	(0,0.7) -- (-2,0.7);

\node at (1.5,0.9) {\tiny $(23)^{q+1}$};
\node at (-1,0.9) {\tiny $(23)^q$};

\node at (-1,-0.7) {${\mc H}$};
\node at (1.5,-0.7) {${\mc H}$};
		
\end{scope}


\begin{scope}[xshift=-4.8cm]

\foreach \a in {0,1,2}
\foreach \x in {0,...,3}
{
\begin{scope}
[xshift=3.2*\x cm, rotate=120*\a,
decoration={markings,mark=at position 0.5 with {
\draw[white, thick] (0.2,0)--(-0.2,0);
\draw (0.2,0)--(0,-0.1)--(-0.2,0);}}]

\draw[postaction={decorate}]
	(30:1.2) -- (-30:1.2);

\draw[postaction={decorate}]
	(30:1.2) -- (90:1.2);
					
\end{scope}	
}


\foreach \x in {0,2,3} 
{
\begin{scope}[xshift=3.2*\x cm]

\node at (90:1.4) {\tiny $(23)^{q+1}$};
\node at (30:1.3) {\tiny 1};
\node at (150:1.3) {\tiny 1};
\node at (210:1.4) {\tiny 123};
\node at (-30:1.4) {\tiny 123};
\node at (-90:1.4) {\tiny $(23)^q$};

\end{scope}	
}

\node at (90:0.9) {\tiny $(23)^q$};
\node at (30:0.9) {\tiny 123};
\node at (150:0.9) {\tiny 123};
\node at (210:1) {\tiny 1};
\node at (-30:1) {\tiny 1};
\node at (-85:0.85) {\tiny $(23)^{q+1}$};

\begin{scope}[xshift=6.4cm, rotate=120]

\node[rotate=30] at (90:0.8) {\tiny $(23)^q$};
\node at (30:0.9) {\tiny 123};
\node at (150:0.95) {\tiny 123};
\node at (210:1) {\tiny 1};
\node at (-30:1) {\tiny 1};
\node[rotate=30] at (-90:0.7) {\tiny $(23)^{q+1}$};

\end{scope}

\begin{scope}[xshift=9.6cm, rotate=-120]

\node[rotate=-30] at (90:0.8) {\tiny $(23)^q$};
\node at (30:0.95) {\tiny 123};
\node at (150:0.9) {\tiny 123};
\node at (210:1) {\tiny 1};
\node at (-30:1) {\tiny 1};
\node[rotate=-30] at (-90:0.7) {\tiny $(23)^{q+1}$};

\end{scope}

\node at (0,0) {${\mc H}$};
\node at (130:1.5) {${\mc H}$};

\begin{scope}[xshift=3.2cm]

\foreach \a in {0,1,2}
{
\node at (90+120*\a:1) {\tiny $\theta$};
\node at (90+120*\a:1.35) {\tiny $\bar{\theta}$};
\node at (-90+120*\a:1.35) {\tiny $\theta$};
\node at (-90+120*\a:1) {\tiny $\bar{\theta}$};
}

\draw[->, gray!50]
	(0:0.4) arc (0:120:0.4);
\draw[->, gray!50]
	(0:0.6) arc (0:240:0.6);	

\node at (-0.35,0.2) {$R$};
\node at (-0.15,-0.5) {$R$};

\end{scope}

\end{scope}

\end{tikzpicture}
\caption{Rotation modification $RE_{\triangle}^J1$ for $h\bar{h}a$ and $h\bar{h}r$.}
\label{tiling_hha9}
\end{figure}
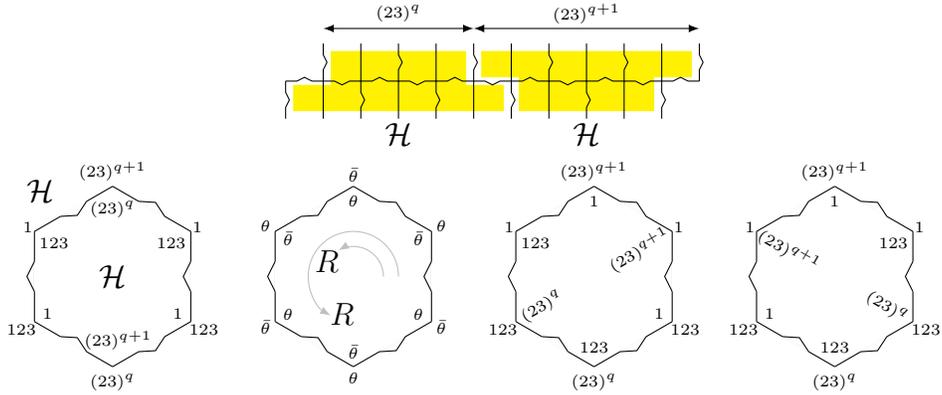

The bottom row consists of four pictures. The first is the same earth map tiling $E_{\triangle}^J1$, with two ${\mc H}$ being the interior and the exterior of a circle. We indicate the edges and corner combinations along the circle between the two hemispheres. In the second picture, we indicate the angle values of the corner combinations, with $\theta=(1-\frac{4}{f})\pi$ and $\bar{\theta}=2\pi-\theta=(1+\frac{4}{f})\pi$. Then it is clear that we may rotate the inner ${\mc H}$ by $\frac{2}{3}\pi$ or $\frac{4}{3}\pi$, and still get a tiling. These are the rotation modifications $RE_{\triangle}^J1$ of the earth map tiling $E_{\triangle}^J1$. The corner combinations along the circle in the rotation modifications $RE_{\triangle}^J1$ are indicated in the third and fourth pictures.

On the left of Figure \ref{tiling_hha8}, the left of $l_1$, the right of $l_2$, \circled{1} and \circled{3} form one hemisphere tiling ${\mc H}$. The rest of the tilings, consisting of \circled{2}, \circled{4}, \circled{6}, \circled{7}, \circled{8}, \circled{9}, \circled{\footnotesize 10} and the tiles between them, form another hemisphere tiling ${\mc H}$. If we regard the first ${\mc H}$ as outside the circle, and the second ${\mc H}$ as inside the circle, then the tiling is the rotation modification by $\frac{2}{3}\pi$. Similarly, the right of Figure \ref{tiling_hha8} is the rotation modification by $\frac{4}{3}\pi$. By exchanging the interior ${\mc H}$ and the exterior ${\mc H}$, we fine the two rotation modifications are actually the same tiling.

We just completed the argument that the $h\bar{h}a$-tilings for (1) are the earth map tiling $E_{\triangle}^J1$ and the rotation modification $RE_{\triangle}^J1$. The same argument can be applied to $h\bar{h}r$, $rra$, $rr{\color{red} r'}$, and we get similar conclusions. All have the usual earth map tiling $E_{\triangle}^J1$ given by Figures \ref{tiling_hha5A} and \ref{tiling_hha5B}. The two rotation modifications for $h\bar{h}r$ are given by Figure \ref{tiling_hha9}, and are still the same. 

The two rotation modifications for $rra$ and $rr{\color{red} r'}$ are given by Figure \ref{tiling_hha10}. Since the exchange of the interior ${\mc H}$ and the exterior ${\mc H}$ would change $r$ in the boundary circle to $r^{-1}$, we cannot apply the exchange, and the two rotation modifications are actually different tilings for $rra$ and $rr{\color{red} r'}$.

\begin{figure}[htp]
\centering
\begin{tikzpicture}[>=latex]

\foreach \a in {0,1,2}
\foreach \x in {0,...,3}
{
\begin{scope}[xshift=3.2*\x cm, rotate=120*\a, decoration={markings,mark=at position 0.5 with {
\draw[white,thick](0.2,0)--(-0.2,0);
\draw (0.2,0)--(0.1,0.1)--(-0.1,-0.1)--(-0.2,0);}}]

\draw[postaction={decorate}]
	(30:1.2) -- (-30:1.2);

\draw[postaction={decorate}]
	(30:1.2) -- (90:1.2);
					
\end{scope}	
}


\foreach \x in {0,2,3} 
{
\begin{scope}[xshift=3.2*\x cm]

\node at (90:1.4) {\tiny $(23)^{q+1}$};
\node at (30:1.3) {\tiny 1};
\node at (150:1.3) {\tiny 1};
\node at (210:1.4) {\tiny 123};
\node at (-30:1.4) {\tiny 123};
\node at (-90:1.4) {\tiny $(23)^q$};

\end{scope}	
}

\node at (90:0.9) {\tiny $(23)^q$};
\node at (30:0.9) {\tiny 123};
\node at (150:0.9) {\tiny 123};
\node at (210:1) {\tiny 1};
\node at (-30:1) {\tiny 1};
\node at (-85:0.85) {\tiny $(23)^{q+1}$};

\begin{scope}[xshift=6.4cm, rotate=120]

\node[rotate=30] at (90:0.8) {\tiny $(23)^q$};
\node at (30:0.9) {\tiny 123};
\node at (150:0.95) {\tiny 123};
\node at (210:1) {\tiny 1};
\node at (-30:1) {\tiny 1};
\node[rotate=30] at (-90:0.7) {\tiny $(23)^{q+1}$};

\end{scope}

\begin{scope}[xshift=9.6cm, rotate=-120]

\node[rotate=-30] at (90:0.8) {\tiny $(23)^q$};
\node at (30:0.95) {\tiny 123};
\node at (150:0.9) {\tiny 123};
\node at (210:1) {\tiny 1};
\node at (-30:1) {\tiny 1};
\node[rotate=-30] at (-90:0.7) {\tiny $(23)^{q+1}$};

\end{scope}

\node at (0,0) {${\mc H}$};
\node at (130:1.5) {${\mc H}$};

\begin{scope}[xshift=3.2cm]

\foreach \a in {0,1,2}
{
\node at (90+120*\a:1) {\tiny $\theta$};
\node at (90+120*\a:1.35) {\tiny $\bar{\theta}$};
\node at (-90+120*\a:1.35) {\tiny $\theta$};
\node at (-90+120*\a:1) {\tiny $\bar{\theta}$};
}

\draw[->, gray!50]
	(0:0.4) arc (0:120:0.4);
\draw[->, gray!50]
	(0:0.6) arc (0:240:0.6);	

\node at (-0.35,0.2) {$R$};
\node at (-0.15,-0.5) {$R$};

\end{scope}

\end{tikzpicture}
\caption{Rotation modifications $RE_{\triangle}^J1$ for $rra$ and $rr{\color{red} r'}$.}
\label{tiling_hha10}
\end{figure}
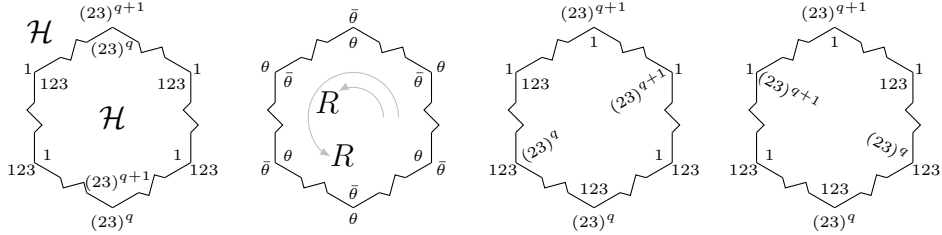

\medskip

\noindent {\bf Case (3).} Earth map tiling $E_{\triangle}2$ and modifications.

\medskip

We know all the vertices $\text{AVC}
=\{12^23^2,1^{\frac{f+4}{8}}23,1^{\frac{f}{4}}\}$. We discuss $h\bar{h}a$-tilings. The proof for the other types are similar.

Two adjacent corners $1$ at a vertex form $\langle 1\langle 1\langle$. This determines \circled{1} and \circled{2} in Figure \ref{tiling_hha6}. The vertex $2_23_1\cdots=12^23^2$ or $1^{\frac{f+4}{8}}23$. Since $1^{\frac{f+4}{8}}23$ is one fan $|3\langle 1\langle \cdots\langle 1\langle 2|$, which is incompatible with the $h$-edge between the corners $2_2$ and $3_1$, we know $2_23_1\cdots=12^23^2$. This is a combination of the fan $|3_1\langle 2_2|$ and a fan $|3\langle 1\langle 2|$. The fan $|3\langle 1\langle 2|$ determines \circled{3}, \circled{4}, \circled{5} in two ways, given by the respective pictures. Then there is a tile \circled{6} sharing the $a$-edge with \circled{4}. Again \circled{6} can be placed in two ways, given by the respective pictures. 

\begin{figure}[htp]
\centering
\begin{tikzpicture}[>=latex]

\begin{scope}[gray!50, line width=3]

\draw[<-]
	(-2.7,0.5) -- (2.7,0.5);
\draw[<-]
	(-2.5,-0.5) -- (2.5,-0.5);
\draw[->]
	(-2.2,-1.5) -- (2.2,-1.5);

\begin{scope}[xshift=6.5cm]

\draw[<-]
	(-2.7,0.5) -- (2.7,0.5);
\draw[->]
	(-2.5,-0.5) -- (2.5,-0.5);
\draw[<-]
	(-2.2,-1.5) -- (2.2,-1.5);

\end{scope}
		
\end{scope}


\foreach \x in {-2,...,2}
{
\begin{scope}[xshift=\x cm, font=\tiny]

\node at (0,-0.7) {1};
\node at (0.3,-0.15) {3};
\node at (-0.3,-0.15) {2};

\end{scope}
}

\foreach \x in {-2,...,2}
{
\begin{scope}[xshift=6.5cm+\x cm, font=\tiny]

\node at (0,-0.7) {1};
\node at (0.3,-0.15) {2};
\node at (-0.3,-0.15) {3};

\end{scope}
}

\foreach \x in {-2,...,1}
{
\begin{scope}[xshift=\x cm, yshift=-1cm, font=\tiny]

\node at (0.5,-1) {1};
\node at (0.15,-0.15) {3};
\node at (0.85,-0.15) {2};

\node at (0.5,0.7) {1};
\node at (0.2,0.15) {3};
\node at (0.8,0.15) {2};

\end{scope}
}

\foreach \x in {-2,...,1}
{
\begin{scope}[xshift=6.5 cm + \x cm, yshift=-1cm, font=\tiny]

\node at (0.5,-1) {1};
\node at (0.15,-0.15) {2};
\node at (0.85,-0.15) {3};

\node at (0.5,0.7) {1};
\node at (0.2,0.15) {2};
\node at (0.8,0.15) {3};

\end{scope}
}

\foreach \u in {0,1}
{
\begin{scope}[xshift=6.5*\u cm, font=\tiny]

\foreach \x in {-2,...,2}
{
\begin{scope}[xshift=\x cm]

\node at (0,1) {1};
\node at (-0.35,0.15) {2};
\node at (0.35,0.15) {3};

\end{scope}
}

\node[inner sep=0.5, draw, shape=circle] at (-1,0.5) {1};
\node[inner sep=0.5, draw, shape=circle] at (0,0.5) {2};
\node[inner sep=0.5, draw, shape=circle] at (-1,-0.4) {3};
\node[inner sep=0.5, draw, shape=circle] at (-0.5,-0.6) {4};
\node[inner sep=0.5, draw, shape=circle] at (0,-0.4) {5};
\node[inner sep=0.5, draw, shape=circle] at (-0.5,-1.5) {6};

\draw
	(-2.5,0) -- (2.5,0)
	(-2,-1) -- (2,-1);

\end{scope}
}

\node at (9,1.2) {\tiny $l_1$};
\node at (4,1.2) {\tiny $l_2$};


\begin{scope}
[decoration={markings,mark=at position 0.5 with {
\draw[white, thick] (0.2,0)--(-0.2,0);
\draw (0.2,0)--(0,0.1)--(-0.2,0);}}]

\foreach \x in {-2,...,3}
\foreach \u in {0,1}
\draw[postaction={decorate}, xshift=\x cm+6.5*\u cm]
	(-0.5,0) -- (-0.5,1);

\foreach \x in {-2,...,2}
{
\begin{scope}[xshift=\x cm]

\draw[postaction={decorate}]
	(0,-1) -- (-0.5,0);

\draw[postaction={decorate}]
	(0,-1) -- (0.5,0);
		
\end{scope}
}

\foreach \x in {-2,...,2}
\draw[postaction={decorate}, xshift=\x cm]
	(0,-1) -- (0,-2);

\begin{scope}[xshift=6.5cm]

\foreach \x in {-2,...,2}
{
\begin{scope}[xshift=\x cm]

\draw[postaction={decorate}]
	(-0.5,0) -- (0,-1);
\draw[postaction={decorate}]
	(0.5,0) -- (0,-1);
		
\end{scope}
}

\foreach \x in {-2,...,2}
\draw[postaction={decorate}, xshift=\x cm]
	(0,-2) -- (0,-1);

\end{scope}

\end{scope}	

\end{tikzpicture}
\caption{Earth map tiling $E_{\triangle}2$ for $h\bar{h}a$.}
\label{tiling_hha6}
\end{figure}
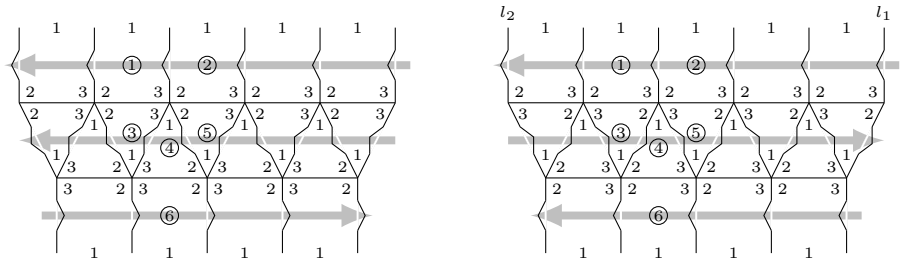

More generally, $q+1$ consecutive corners 1 at a vertex form $\langle 1\langle 1\langle \cdots\langle 1\langle 1\langle$. The sequence contains $q$ pairs $\langle 1\langle 1\langle$, and each pair determines six tiles \circled{1}, \circled{2}, \circled{3}, \circled{4}, \circled{5}, \circled{6} as in Figure \ref{tiling_hha6}. For all the sextuples to be compatible, we get the partial tiling in Figure \ref{tiling_hha6}, with $1^{q+1}$ at one end and $1^q$ at the other end ($q=4$ in the picture). The partial tiling has three layers, and each layer has a direction (indicated by gray arrows) that we can independently choose. Due to the independent choice of the directions of the three layers, there are actually three distinct versions of $E_{\triangle}2$ for $h\bar{h}a$.

The argument above also applies to $h\bar{h}r$, $rra$, $rr{\color{red} r'}$. The difference is that, compared with $h\bar{h}a$, the $r$-edge in $h\bar{h}r$ and the ${\color{red} r'}$-edge in $rr{\color{red} r'}$ impose extra constraint on the tiling. As a result, we only have one version of $E_{\triangle}2$ for $h\bar{h}r$ and $rr{\color{red} r'}$, given by the left and right of Figure \ref{tiling_hha11}. On the other hand, the situation for $rra$ is similar to $h\bar{h}a$. Each layer consists of either all $r$, or all $r^{-1}$, and the choices can be independent. There are three distinct versions of $E_{\triangle}2$ for $rra$. See the middle two of Figure \ref{tiling_hha11}.

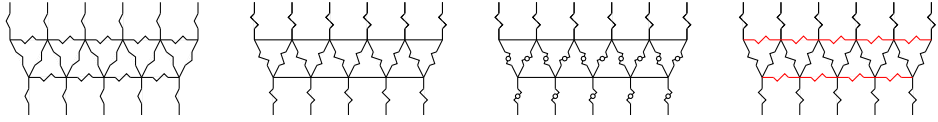
\begin{figure}[htp]
\centering
\begin{tikzpicture}[>=latex, scale=0.5]


\begin{scope}
[decoration={markings,mark=at position 0.5 with {
\draw[white, thick] (0.1,0)--(-0.1,0);
\draw (0.1,0)--(0,0.05)--(-0.1,0);}},
xshift=-6.5cm]

\foreach \x in {-2,...,3}
\draw[postaction={decorate}, xshift=\x cm]
	(-0.5,0) -- (-0.5,1);

\foreach \x in {-2,...,2}
{
\begin{scope}[xshift=\x cm]

\draw[postaction={decorate}]
	(-0.5,0) -- (0,-1);

\draw[postaction={decorate}]
	(0.5,0) -- (0,-1);

\draw[postaction={decorate}]
	(0,-2) -- (0,-1);
			
\end{scope}
}

\end{scope}	


\begin{scope}
[decoration={markings,mark=at position 0.5 with {
\draw[white,thick](0.1,0)--(-0.1,0);
\draw (0.1,0)--(0.05,0.05)--(-0.05,-0.05)--(-0.1,0);}}]

\foreach \x in {-2,...,2}
\draw[postaction={decorate}, xshift=-6.5cm+\x cm]
	(-0.5,0) -- (0.5,0);

\foreach \x in {-2,...,1}
\draw[postaction={decorate}, xshift=-6.5cm+\x cm]
	(0,-1) -- (1,-1);


\foreach \x in {-2,...,3}
\foreach \u in {0,1,2}
\draw[postaction={decorate}, xshift=\x cm+6.5*\u cm]
	(-0.5,0) -- (-0.5,1);

\foreach \x in {-2,...,2}
\foreach \u in {0,2}
{
\draw[postaction={decorate}]
	(-0.5+\x+6.5*\u,0) -- ++(0.5,-1);

\draw[postaction={decorate}]
	(0.5+\x+6.5*\u,0) -- ++(-0.5,-1);

\draw[postaction={decorate}]
	(\x+6.5*\u,-1) -- ++(0,-1);		
}

\foreach \x in {-2,...,3}
\foreach \u in {0,1,2}
\draw[postaction={decorate}, xshift=\x cm+6.5*\u cm]
	(-0.5,0) -- (-0.5,1);

\foreach \x in {-2,...,2}
\draw[postaction={decorate}, xshift=13cm+\x cm, red]
	(-0.5,0) -- (0.5,0);

\foreach \x in {-2,...,1}
\draw[postaction={decorate}, xshift=13cm+\x cm, red]
	(0,-1) -- (1,-1);
	
\end{scope}


\begin{scope}
[decoration={markings,mark=at position 0.5 with {
\draw[white,thick](0.1,0)--(-0.1,0);
\draw (0.1,0)--(0.05,-0.05)--(-0.05,0.05)--(-0.1,0);
\filldraw[fill=white] (0,0) circle (0.03); }},
xshift=6.5cm]

\foreach \x in {-2,...,2}
{
\draw[postaction={decorate}]
	(-0.5+\x,0) -- ++(0.5,-1);

\draw[postaction={decorate}]
	(0.5+\x,0) -- ++(-0.5,-1);

\draw[postaction={decorate}]
	(\x,-1) -- ++(0,-1);		
}

\end{scope}


\foreach \u in {0,1}
\draw[xshift=6.5*\u cm]
	(-2.5,0) -- (2.5,0)
	(-2,-1) -- (2,-1);

\end{tikzpicture}
\caption{Earth map tiling $E_{\triangle}2$ for $h\bar{h}r$, $rra$, $rr{\color{red} r'}$.}
\label{tiling_hha11}
\end{figure}

Suppose $1^{\frac{f}{4}}$ is not a vertex. Then $\text{AVC}=\{12^23^2,1^{q+1}23\}$, $q=\frac{f-4}{8}$, are all the vertices. Since the total numbers of the corners 1 and 2 in the tiling are the same, both must appear as vertices. We continue the argument for $h\bar{h}a$.

The vertex $1^{q+1}23$ is one fan $|3\langle 1\langle \dots \langle 1\langle 2|$. The $1^{q+1}$ part of the vertex induces a partial tiling ${\mc H}$ with $1^q$ at the other end. There are three versions of ${\mc H}$, and we first argue for the version on the right of Figure \ref{tiling_hha6}. In Figure \ref{tiling_hha12A}, this ${\mc H}$ is the left of $l_1$ and the right of $l_2$. We also highlight the $a$-edges by thick lines. We need to fill the region on the right of $l_1$ and the left of $l_2$. 

The $23$ part of the vertex $1^{q+1}23$ determines \circled{1} and \circled{2} in this region. Then $1_1\cdots=12^23^2,1^{q+1}23$. The left of Figure \ref{tiling_hha12A} is the case $1_1\cdots=12^23^2$. The vertex $12^23^2$ consists of one fan $|3\langle 2|$ and one fan $|3\langle 1\langle 2|$. By \circled{1} and the tiling ${\mc H}$ on the left of $l_1$, this information about fans determine \circled{3} and \circled{4}. Then $1_4\cdots=1^223\cdots=1^{q+1}23$. The $1^{q+1}$ part of this vertex determines another partial tiling ${\mc H}$ that consists of \circled{1}, \circled{2}, \circled{3}, \circled{4}, \circled{5}, \circled{6}, \circled{7} and the tiles between them, and two tiles on the left of $l_1$ (the picture shows the case $q=3$). The two ends $1^{q+1}$ and $1^q$ are indicated by $\bullet$. Then $1_21_7\cdots=1^{q+1}23$, and the $1^{q+1}$ part of this vertex determines \circled{8} and \circled{9}. 

\begin{figure}[htp]
\centering
\begin{tikzpicture}[>=latex]

\draw[<-, line width=3, gray!50]
	(-2,1.2) to[out=0, in=90]
	(0,-0.7) to[out=-90, in=180]
	(2,-1.2);

\draw[->, line width=3, gray!50]
	(-2,0) to[out=0, in=90]
	(-0.8,-0.6) to[out=-90, in=0]
	(-2,-1.2);

\draw[<-, line width=3, gray!50]
	(2,0) to[out=180, in=-90]
	(0.8,0.6) to[out=90, in=180]
	(2,1.2);
	
\begin{scope}[xshift=6cm]

\draw[->, line width=3, gray!50]
	(2,1.2) to[out=180, in=90]
	(0,-0.7) to[out=-90, in=0]
	(-2,-1.2);

\draw[->, line width=3, gray!50]
	(-2,0) to[out=0, in=-90]
	(-0.8,0.6) to[out=90, in=0]
	(-2,1.2);

\draw[<-, line width=3, gray!50]
	(2,0) to[out=180, in=90]
	(0.8,-0.6) to[out=-90, in=180]
	(2,-1.2);

\end{scope}

\begin{scope}[font=\tiny]

\node at (-0.75,1.7) {2};
\node at (0.75,1.7) {3};
\node at (-0.15,1.1) {3};
\node at (0.15,1.1) {2};
\node at (-0.25,0.8) {1};
\node at (0.25,0.8) {3};
\node at (0,0.75) {2};

\node at (-1,-1.7) {2};
\node at (0,-1.7) {3};
\node at (1,-1.7) {1};

\node at (-1.35,0.8) {1};
\node at (1.35,0.8) {1};
\node at (-1.2,0.55) {2};
\node at (1.2,0.55) {1};
\node at (-1.35,0.35) {3};
\node at (1.2,0.3) {1};
\node at (1.4,0.25) {1};
\node at (-1.35,-0.35) {1};
\node at (-1.2,-0.57) {1};
\node at (1.35,-0.5) {2};
\node at (-1.35,-0.8) {1};
\node at (1.35,-0.8) {3};

\node at (2,1.7) {1};
\node at (-2,1.7) {1};
\node at (-1.65,0.75) {3};
\node at (1.65,0.75) {2};
\node at (-1.65,0.45) {2};
\node at (1.65,0.45) {3};
\node at (-1.6,-0.3) {1};
\node at (1.6,-0.3) {1};
\node at (-1.8,-0.45) {2};
\node at (1.8,-0.45) {3};
\node at (-1.65,-0.75) {3};
\node at (1.65,-0.75) {2};
\node at (2,-1.7) {1};
\node at (-2,-1.7) {1};

\node at (-1.5,2) {$l_1$};
\node at (1.5,2) {$l_2$};
\node at (4.5,2) {$l_1$};
\node at (7.5,2) {$l_2$};

\node[inner sep=0.5, draw, shape=circle] at (-0.75,1.3) {1};
\node[inner sep=0.5, draw, shape=circle] at (0.75,1.3) {2};
\node[inner sep=0.5, draw, shape=circle] at (-0.75,0.5) {3};
\node[inner sep=0.5, draw, shape=circle] at (-1,0) {4};
\node[inner sep=0.5, draw, shape=circle] at (-1,-1.3) {5};
\node[inner sep=0.5, draw, shape=circle] at (0,-1.3) {6};
\node[inner sep=0.5, draw, shape=circle] at (0.8,-0.1) {7};
\node[inner sep=0.5, draw, shape=circle] at (1.1,-0.5) {8};
\node[inner sep=0.5, draw, shape=circle] at (1,-1.3) {9};

\end{scope}

\fill
	(-1.5,-0.6) circle (0.05)
	(1.5,0.6) circle (0.05);

\fill[xshift=6 cm]
	(-1.5,0.6) circle (0.05)
	(1.5,-0.6) circle (0.05);

\begin{scope}[xshift=6cm, font=\tiny]

\node at (-0.75,1.7) {2};
\node at (0.75,1.7) {3};
\node at (-0.15,1.1) {3};
\node at (0.15,1.1) {2};
\node at (-0.25,0.8) {2};
\node at (0.25,0.8) {1};
\node at (0,0.75) {3};

\node at (-1,-1.7) {3};
\node at (0,-1.7) {2};
\node at (1,-1.7) {1};

\node at (-1.35,0.8) {1};
\node at (1.35,0.8) {1};
\node at (-1.2,0.55) {1};
\node at (1.2,0.55) {3};
\node at (1.35,0.35) {2};
\node at (-1.2,0.3) {1};
\node at (-1.4,0.25) {1};
\node at (1.35,-0.35) {1};
\node at (1.2,-0.57) {1};
\node at (1.35,-0.8) {1};
\node at (-1.35,-0.5) {3};
\node at (-1.35,-0.8) {2};

\node at (2,1.7) {1};
\node at (-2,1.7) {1};
\node at (-1.65,0.75) {3};
\node at (1.65,0.75) {2};
\node at (-1.65,0.45) {2};
\node at (1.65,0.45) {3};
\node at (-1.6,-0.3) {1};
\node at (1.6,-0.3) {1};
\node at (-1.8,-0.45) {2};
\node at (1.8,-0.45) {3};
\node at (-1.65,-0.75) {3};
\node at (1.65,-0.75) {2};
\node at (2,-1.7) {1};
\node at (-2,-1.7) {1};

\node at (-1.5,2) {$l_1$};
\node at (1.5,2) {$l_2$};

\node[inner sep=0.5, draw, shape=circle] at (-0.75,1.3) {1};
\node[inner sep=0.5, draw, shape=circle] at (0.75,1.3) {2};

\end{scope}


\begin{scope}
[decoration={markings,mark=at position 0.5 with {
\draw[white, thick] (0.2,0)--(-0.2,0);
\draw (0.2,0)--(0,0.1)--(-0.2,0);}}]

\foreach \u in {0,1}
{
\begin{scope}[xshift=6*\u cm]

\foreach \x in {-1,1}
{
\draw[postaction={decorate}]
	(1.5*\x,0.6) -- ++(0,-1.2);

\draw[postaction={decorate}]
	(1.5*\x,0.6) -- ++(0,1.2);

\draw[postaction={decorate}]
	(1.5*\x,-1.8) -- ++(0,1.2);
		
\draw[postaction={decorate}]
	(-2.5,0.6) -- (-1.5,-0.6);
\draw[postaction={decorate}]
	(2.5,0.6) -- (1.5,-0.6);
	
}

\end{scope}
}

\foreach \x/\u in {1/0, -1/1}
{
\begin{scope}[xshift=6*\u cm, xscale=\x]

\foreach \y in {0,-1}
\draw[postaction={decorate}]
	(-0.5,\y) -- (0.5,\y);
\draw[postaction={decorate}]
	(-0.5,-1) -- (0.5,0);
\draw[postaction={decorate}]
	(0.5,-1.8) -- (0.5,-1);
\draw[postaction={decorate}]
	(-1.5,0.6) -- (0,1);	
\draw[postaction={decorate}]
	(-0.5,0) -- (0,1);
\draw[postaction={decorate}]
	(-0.5,0) -- (-1.5,-0.6);
\draw[postaction={decorate}]
	(-0.5,-1) -- (-1.5,-0.6);
	
\foreach \a/\b in {0/1, 0.5/0, 0.5/-1}
\draw[postaction={decorate}]
	(1.5,0.6) -- (\a,\b);

\end{scope}
}
					
\end{scope}	


\foreach \x/\u in {1/0, -1/1}
{
\begin{scope}[xshift=6*\u cm, xscale=\x]

\draw[thick]
	(0,1.8) -- (0,1) -- (0.5,0) -- (0.5,-1) -- (1.5,-0.6) -- (2.5,-0.6)
	(1.5,0.6) -- (2.5,0.6)
	(-2.5,0.6) -- (-1.5,0.6) -- (-0.5,0) -- (-0.5,-1.8)
	(-2.5,-0.6) -- (-1.5,-0.6);

\end{scope}
}

\end{tikzpicture}
\caption{Modification of $E_{\triangle}2$ for $h\bar{h}a$.}
\label{tiling_hha12A}
\end{figure}
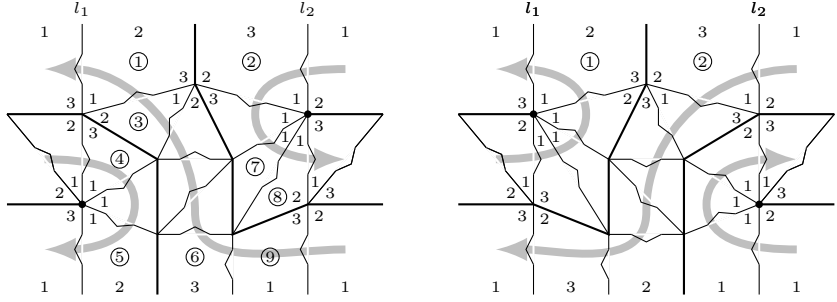

The boundary between the layers in two ${\mc H}$ in the left picture are the $a$-edges. In fact, these layers together form one strip, and all $h$ in this strip should have the same direction. The compatible directions are indicated by the thick gray arrows, and imply that the three layers of both ${\mc H}$ must be alternating. This is the version on the right of Figure \ref{tiling_hha6}, and is the reason we use this version for our argument. Any other version would lead to contradiction in directions. 

We still need to consider the case that the vertex $1_1\cdots=1^{q+1}23$. This is the right of Figure \ref{tiling_hha12A}. The $1^{q+1}$ part of the vertex induces a partial earth tiling ${\mc H}$ that fills the region on the right of $l_1$ and the left of $l_2$. Again we need to check the compatibility of the directions of the layers. We find they must be alternating, as indicated by the thick gray arrows.

We may apply the same argument to $h\bar{h}r$, $rra$, $rr{\color{red} r'}$, for the updated $\text{AVC}=\{12^23^2,1^{q+1}23\}$. We get similar tilings. Specifically, tilings for $h\bar{h}r$ are obtained by replacing all $a$-edges in Figure \ref{tiling_hha12A} to $r$-edges (not $r^{-1}$-edges). For $rra$, we get the tilings in Figure \ref{tiling_hha12B}, completely similar to Figure \ref{tiling_hha12A}. We indicate the $a$-edges by thick lines, which also serve as the boundary between the layers of two ${\mc H}$. These layers together again form one strip. Therefore there are only $r$-edges, and no $r^{-1}$-edges. For $rr{\color{red} r'}$, the tilings are obtained by replacing the $a$-edges in Figure \ref{tiling_hha12B} by ${\color{red} r'}$-edges.

\begin{figure}[htp]
\centering
\begin{tikzpicture}[>=latex]

\fill
	(-1.5,-0.6) circle (0.05)
	(1.5,0.6) circle (0.05);

\fill[xshift=6 cm]
	(-1.5,0.6) circle (0.05)
	(1.5,-0.6) circle (0.05);
		

\begin{scope}
[decoration={markings,mark=at position 0.5 with {
\draw[white, thick] (0.2,0)--(-0.2,0);
\draw (0.2,0)--(0.1,0.1)--(-0.1,-0.1)--(-0.2,0);}}]

\foreach \u in {0,1}
{
\begin{scope}[xshift=6*\u cm]

\foreach \x in {-1,1}
{
\draw[postaction={decorate}]
	(1.5*\x,0.6) -- ++(0,-1.2);

\draw[postaction={decorate}]
	(1.5*\x,0.6) -- ++(0,1.2);

\draw[postaction={decorate}]
	(1.5*\x,-1.8) -- ++(0,1.2);
		
\draw[postaction={decorate}]
	(-2.5,0.6) -- (-1.5,-0.6);
\draw[postaction={decorate}]
	(2.5,0.6) -- (1.5,-0.6);
	
}

\end{scope}
}

\foreach \x/\u in {1/0, -1/1}
{
\begin{scope}[xshift=6*\u cm, xscale=\x]

\foreach \y in {0,-1}
\draw[postaction={decorate}]
	(-0.5,\y) -- (0.5,\y);
\draw[postaction={decorate}]
	(-0.5,-1) -- (0.5,0);
\draw[postaction={decorate}]
	(0.5,-1.8) -- (0.5,-1);
\draw[postaction={decorate}]
	(-1.5,0.6) -- (0,1);	
\draw[postaction={decorate}]
	(-0.5,0) -- (0,1);
\draw[postaction={decorate}]
	(-0.5,0) -- (-1.5,-0.6);
\draw[postaction={decorate}]
	(-0.5,-1) -- (-1.5,-0.6);
	
\foreach \a/\b in {0/1, 0.5/0, 0.5/-1}
\draw[postaction={decorate}]
	(1.5,0.6) -- (\a,\b);

\end{scope}
}
					
\end{scope}	


\foreach \x/\u in {1/0, -1/1}
{
\begin{scope}[xshift=6*\u cm, xscale=\x]

\draw[thick]
	(0,1.8) -- (0,1) -- (0.5,0) -- (0.5,-1) -- (1.5,-0.6) -- (2.5,-0.6)
	(1.5,0.6) -- (2.5,0.6)
	(-2.5,0.6) -- (-1.5,0.6) -- (-0.5,0) -- (-0.5,-1.8)
	(-2.5,-0.6) -- (-1.5,-0.6);

\end{scope}
}

\end{tikzpicture}
\caption{Modification of $E_{\triangle}2$ for $rra$.}
\label{tiling_hha12B}
\end{figure}
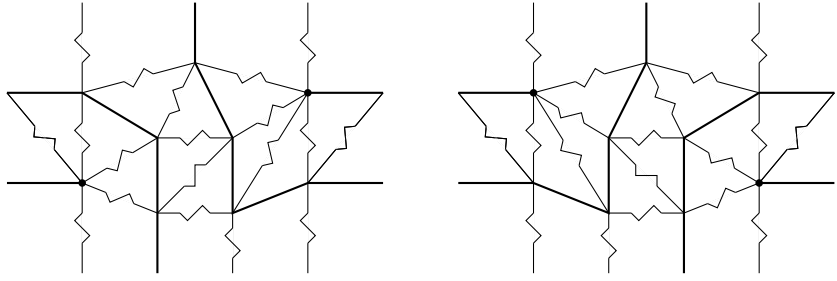

Finally, we interpret the tilings in Figures \ref{tiling_hha12A} and \ref{tiling_hha12B}. The tilings are the unions of two partial tilings ${\mc H}$, one on the left of $l_1$ and the right of $l_2$, and another on the right of $l_1$ and the left of $l_2$. Therefore ${\mc H}$ is actually a hemisphere tiling. In Figure \ref{tiling_hha13}, we draw the union of two tilings similar to the right of Figure \ref{tiling_hha10}. The left two correspond the two tilings in Figure \ref{tiling_hha12A} for $h\bar{h}a$ (and $h\bar{h}r$), and the right two correspond the two tilings in Figure \ref{tiling_hha12B} for $rra$ (and $rr{\color{red} r'}$).

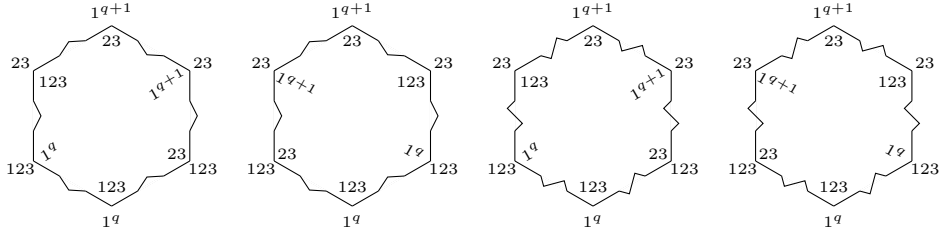
\begin{figure}[htp]
\centering
\begin{tikzpicture}[>=latex]


\foreach \a in {0,1,2}
\foreach \x in {0,1}
{
\begin{scope}
[xshift=3.2*\x cm, rotate=120*\a,
decoration={markings,mark=at position 0.5 with {
\draw[white, thick] (0.2,0)--(-0.2,0);
\draw (0.2,0)--(0,0.1)--(-0.2,0);}}]

\draw[postaction={decorate}]
	(30:1.2) -- (-30:1.2);

\draw[postaction={decorate}]
	(30:1.2) -- (90:1.2);
					
\end{scope}	
}


\foreach \x in {2,3}
\foreach \a in {0,...,5}
{
\begin{scope}[xshift=3.2*\x cm, rotate=60*\a, decoration={markings,mark=at position 0.5 with {
\draw[white,thick](0.2,0)--(-0.2,0);
\draw (0.2,0)--(0.1,0.1)--(-0.1,-0.1)--(-0.2,0);}}]

\draw[postaction={decorate}]
	(30:1.2) -- (90:1.2);
		
\end{scope}
}

\foreach \x in {0,...,3}
{
\begin{scope}[xshift=3.2*\x cm, font=\tiny]

\node at (90:1.4) {$1^{q+1}$};
\node at (30:1.4) {23};
\node at (150:1.4) {23};
\node at (210:1.4) {123};
\node at (-30:1.4) {123};
\node at (-90:1.4) {$1^q$};

\end{scope}
}

\foreach \x in {0,2}
{
\begin{scope}[xshift=3.2*\x cm, rotate=120, font=\tiny]

\node[rotate=30] at (90:0.95) {$1^q$};
\node at (30:0.9) {123};
\node at (150:0.95) {123};
\node at (210:1) {23};
\node at (-30:1) {23};
\node[rotate=30] at (-90:0.85) {$1^{q+1}$};

\end{scope}
}

\foreach \x in {1,3}
{
\begin{scope}[xshift=3.2*\x cm, rotate=-120, font=\tiny]

\node[rotate=-30] at (90:0.95) {$1^q$};
\node at (30:0.95) {123};
\node at (150:0.9) {123};
\node at (210:1) {23};
\node at (-30:1) {23};
\node[rotate=-30] at (-90:0.85) {$1^{q+1}$};

\end{scope}
}

\end{tikzpicture}
\caption{Rotation modifications $RE_{\triangle}2$.}
\label{tiling_hha13}
\end{figure}

The standard earth map tiling $E_{\triangle}2$ is described by the similar pictures where $1^q$ and $1^{q+1}$ are matched, and $23$ and $123$ are matched. The first and third of Figure \ref{tiling_hha13} are obtained by rotating the interior ${\mc H}$ of $E_{\triangle}2$ by $\frac{2}{3}\pi$. The second and fourth of Figure \ref{tiling_hha13} are obtained by rotating by $\frac{4}{3}\pi$. Therefore tilings for $\text{AVC}=\{12^23^2,1^{q+1}23\}$ are the two rotation modifications. 

By exchanging the interior and exterior hemispheres ${\mc H}$, we find the two rotations are equivalent tilings for $h\bar{h}a$ and $h\bar{h}r$. However, we cannot do the exchange for $rra$ and $rr{\color{red} r'}$ because this changes $r$ on the boundary circle to $r^{-1}$. Therefore the two rotations are distinct tilings for $rra$ and $rr{\color{red} r'}$.
\end{proof}

\begin{proposition}
Tilings of the sphere by congruent $raa$-triangles are the following:
\begin{itemize}
\item Tetrahedron $P_4$.
\item Triangular subdivisions $T_{\triangle}P_n$ of all five Platonic solids.
\item Simple triangular subdivisions $S_{\triangle}P_6$ of the cube.
\item Earth map tilings $E_{\triangle}^I1$, $E_{\triangle}^J1$, $E_{\triangle}2$, and their rotation modifications $RE_{\triangle}^I1$, $RE_{\triangle}^J1$, $RE_{\triangle}2$. 
\end{itemize}
\end{proposition}

\begin{proof}
The eighth of Figure \ref{pair} shows a companion pair gives the quadrilateral $aaaa$. This is rhombus with straight edges. Therefore we may adopt the argument in Section 4 of \cite{cly}. 

The argument has two parts. The first part is the classification of all rhombus tilings under the assumption that all vertices have degree $\ge 3$. By Proposition 23 of \cite{cly}, rhombus tilings are the quadricentric subdivisions $C_{\square}P_n$ of Platonic solids, the earth map tiling $E_{\square}^R1$ and the flip modification $FE_{\square}^R1$. Then the $raa$-tilings are simple triangular subdivisions of these rhombus tilings, which means using $r$ or $r^{-1}$ to divide each rhombus into two $raa$-triangles. Since two divisions are different, we may get many different versions of $raa$-tilings from the same rhombus tiling.

The simple triangular subdivisions of $C_{\square}P_n$ are the simple triangular subdivisions $S_{\square}P_6$ of the square and the triangular subdivisions $T_{\triangle}P_n$ of Platonic solids, $n=6,8,12,20$. The middle and right of Figure \ref{tiling_raa1} show examples of $S_{\square}P_6$ and $T_{\triangle}P_n$. We can independently exchange any $r$-edge and $r^{-1}$-edge.

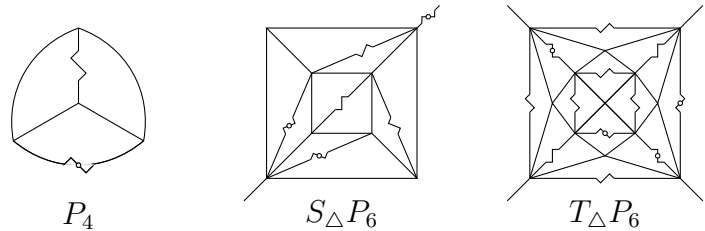
\begin{figure}[htp]
\centering
\begin{tikzpicture}[>=latex,scale=1]

\draw[xshift=-3.5cm]
	(0,0) -- (-30:1) to[out=80, in=-20] 
	(90:1) to[out=200, in=100] 
	(210:1) -- (0,0)
	(0,1) -- (0,0.7) -- (-0.1,0.6) -- (0.1,0.4) -- (0,0.3) -- (0,0);

\foreach \a in {0,1,2,3}
\draw[rotate=90*\a]
	(0.4,-0.4) -- (0.4,0.4) -- (1,1) -- (1,-1);

\foreach \a in {0,1,2,3}
\draw[xshift=3.5cm, rotate=90*\a]
	(0,0) -- (0.4,0.4) -- (0.7,0) -- (0.4,-0.4) -- (0,0)
	(1,-1) -- (0.7,0) -- (1,1) -- (1.3,1.3);


\begin{scope}
[decoration={markings,mark=at position 0.5 with {
\draw[white, thick] (0.12,0)--(-0.12,0);
\draw (0.12,0)--(0.06,0.06)--(-0.06,-0.06)--(-0.12,0);
}}]

\draw[postaction={decorate}]
	(0.4,0.4) -- (-0.4,-0.4);
\draw[postaction={decorate}]
	(0.4,0.4) -- (1,-1);
\draw[postaction={decorate}]
	(-0.4,0.4) -- (1,1);

\begin{scope}[xshift=3.5cm]

\foreach \a in {0,1,2}
{
\draw[postaction={decorate}, rotate=90*\a]
	(0.4,0.4) -- (0.4,-0.4);
\draw[postaction={decorate}, rotate=90*\a]
	(1,1) -- (-1,1);
}

\foreach \a in {1,-1}
\draw[postaction={decorate}, scale=\a]
	(0.4,0.4) -- (1,1);
	
\end{scope}
		
\end{scope}


\begin{scope}
[decoration={markings,mark=at position 0.5 with {
\draw[white, thick] (0.2,0)--(-0.2,0);
\draw (0.2,0)--(0.1,-0.1)--(-0.1,0.1)--(-0.2,0);
\filldraw[fill=white] (0,0) circle (0.03);
}}]

\foreach \a in {1,-1}
\draw[postaction={decorate}, xshift=-3.5cm]
	(210:1) to[out=-40, in=220] (-30:1);

\end{scope}

\begin{scope}
[decoration={markings,mark=at position 0.5 with {
\draw[white, thick] (0.12,0)--(-0.12,0);
\draw (0.12,0)--(0.06,-0.06)--(-0.06,0.06)--(-0.12,0);
\filldraw[fill=white] (0,0) circle (0.03);
}}]

\draw[postaction={decorate}]
	(0.4,-0.4) -- (-1,-1);
\draw[postaction={decorate}]
	(-0.4,0.4) -- (-1,-1);
\draw[postaction={decorate}]
	(1,1) -- (1.3,1.3);

\begin{scope}[xshift=3.5cm]

\draw[postaction={decorate}]
	(-0.4,-0.4) -- (0.4,-0.4);
\draw[postaction={decorate}]
	(1,1) -- (1,-1);
	
\foreach \a in {1,-1}
\draw[postaction={decorate}, scale=\a]
	(-0.4,0.4) -- (-1,1);
	
\end{scope}
		
\end{scope}


\draw
	(-1,-1) -- (-1.3,-1.3);

\node at (0,-1.5) {$S_{\triangle}P_6$};
\node at (3.5,-1.5) {$T_{\triangle}P_6$};
\node at (-3.5,-1.5) {$P_4$};
	
\end{tikzpicture}
\caption{Platonic type tilings $P_4$, $S_{\triangle}P_6$, $T_{\triangle}P_n$ for $raa$.}
\label{tiling_raa1}
\end{figure}

The simple triangular subdivisions of $E_{\square}^R1$ are the earth map tilings $E_{\triangle}^J1$ and $E_{\triangle}2$ in the middle and right of Figure \ref{tiling_raa2}. Again, any $r$ and $r^{-1}$ can be independently exchanged.

\begin{figure}[htp]
\centering
\begin{tikzpicture}[>=latex]


\foreach \a in {0,...,3}
\draw[xshift=0.8*\a cm]
	(0,1) -- (0,-1)
	(0.8,1) -- (0.8,-1);

\foreach \a in {0,2,3}
\draw[xshift=0.8*\a cm]	
	(0,0) -- (0.2,0) -- (0.3,-0.1) -- (0.5,0.1) -- (0.6,0) -- (0.8,0);

\draw[xshift=0.8 cm]	
	(0,0) -- (0.2,0) -- (0.3,0.1) -- (0.5,-0.1) -- (0.6,0) -- (0.8,0);
\filldraw[fill=white]
	(1.2,0) circle (0.05);

\node at (1.6,-1.4) {$E_{\triangle}^I1$};


\begin{scope}[xshift=4.5cm]

\draw
	(0,1) -- (0,0) -- (3.2,0) -- (3.2,1)
	(0.8,0) -- (0.8,-1)
	(1.6,0) -- (1.6,1)
	(2.4,0) -- (2.4,-1);

\draw[xshift=0.8 cm]
	(0,1) -- (0,0.7) -- (-0.1,0.6) -- (0.1,0.4) -- (0,0.3) -- (0,0);

\draw[xshift=2.4 cm]
	(0,1) -- (0,0.7) -- (0.1,0.6) -- (-0.1,0.4) -- (0,0.3) -- (0,0);
	
\foreach \a in {2,4}
\draw[xshift=0.8*\a cm]
	(0,-1) -- (0,-0.7) -- (0.1,-0.6) -- (-0.1,-0.4) -- (0,-0.3) -- (0,0);

\draw
	(0,-1) -- (0,-0.7) -- (-0.1,-0.6) -- (0.1,-0.4) -- (0,-0.3) -- (0,0);
	
\filldraw[fill=white]
	(2.4,0.5) circle (0.05)
	(0,-0.5) circle (0.05);

\node at (1.6,-1.4) {$E_{\triangle}^J1$};

\end{scope}


\begin{scope}[xshift=9cm]

\foreach \a in {0,...,4}
\draw[xshift=0.8*\a cm]
	(0.4,1) -- (0.4,0.4) -- (0,-0.4) -- (0,-1);
	
\foreach \a in {0,1,2,3}
\draw[xshift=0.8*\a cm]
	(0.4,0.4) -- (0.8,-0.4);

\foreach \a in {0,1,3}
\draw[shift={(0.4cm+0.8*\a cm,0.4cm)}]	
	(0,0) -- (0.2,0) -- (0.3,-0.1) -- (0.5,0.1) -- (0.6,0) -- (0.8,0);

\foreach \a in {1,3}
\draw[shift={(0.8*\a cm,-0.4cm)}]	
	(0,0) -- (0.2,0) -- (0.3,-0.1) -- (0.5,0.1) -- (0.6,0) -- (0.8,0);
	
\draw[shift={(2 cm,0.4cm)}]	
	(0,0) -- (0.2,0) -- (0.3,0.1) -- (0.5,-0.1) -- (0.6,0) -- (0.8,0);

\foreach \a in {0,2}
\draw[shift={(0.8*\a cm,-0.4cm)}]	
	(0,0) -- (0.2,0) -- (0.3,0.1) -- (0.5,-0.1) -- (0.6,0) -- (0.8,0);

\filldraw[fill=white]
	(2.4,0.4) circle (0.05)
	(0.4,-0.4) circle (0.05)
	(2,-0.4) circle (0.05);

\node at (1.6,-1.4) {$E_{\triangle}2$};
		
\end{scope}

\end{tikzpicture}
\caption{Earth map tilings $E_{\triangle}^I1$, $E_{\triangle}^J1$, $E_{\triangle}2$ for $raa$.}
\label{tiling_raa2}
\end{figure}

The second part of the argument in Section 4 of \cite{cly} is the case that some vertices in a triangular tiling are changed to degree 2 vertices in the corresponding rhombus tiling. By the same argument in \cite{cly}, we know the tiling is either the tetrahedron $P_4$ on the left of Figure \ref{tiling_raa1}, or the earth map tiling $E_{\triangle}^I1$ on the left of Figure \ref{tiling_raa2}, or the flip modification $FE_{\triangle}^I1$. The flip modification is actually the same as the flip modification $FE_{\triangle}^I1$ in Figure \ref{tiling_hha14A}, and the rotation modification $RE_{\triangle}^I1$ in Figure \ref{tiling_hha14B}. We need to distinguish between the flip $F$ and the rotation $R$ in the earlier tilings due to different symmetry properties of $h$ and $r$. For $raa$, the boundary circle consists of straight lines, and the flip is actually the same as the rotation (up to the free exchange of $r$ and $r^{-1}$ in the interiors of hemisphere tilings). See the left of Figure \ref{tiling_rra3}, where we choose to use $R$ instead of $F$. 

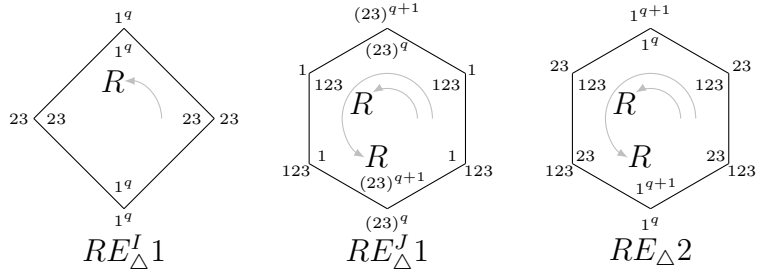
\begin{figure}[htp]
\centering
\begin{tikzpicture}[>=latex]
	
\foreach \a in {0,...,3}
\draw
	(90*\a:1.2) -- (90+90*\a:1.2);

\foreach \a in {0,...,5}
\foreach \x in {1,2}
\draw[xshift=3.5*\x cm]
	(-30+60*\a:1.2) -- (30+60*\a:1.2);


\node at (0,0.9) {\tiny $1^q$};
\node at (0.9,0) {\tiny $23$};
\node at (0,-0.9) {\tiny $1^q$};
\node at (-0.9,0) {\tiny $23$};

\node at (0,1.35) {\tiny $1^q$};
\node at (1.4,0) {\tiny $23$};
\node at (0,-1.35) {\tiny $1^q$};
\node at (-1.4,0) {\tiny $23$};

\draw[gray!50, ->]
	(0:0.5) arc (0:90:0.5);

\node at (-0.15,0.5) {$R$};

\node at (0,-1.8) {$RE_{\triangle}^I1$};


\begin{scope}[xshift=3.5 cm]

\node at (90:0.9) {\tiny $(23)^q$};
\node at (30:0.9) {\tiny 123};
\node at (150:0.9) {\tiny 123};
\node at (210:1) {\tiny 1};
\node at (-30:1) {\tiny 1};
\node at (-85:0.85) {\tiny $(23)^{q+1}$};

\node at (90:1.4) {\tiny $(23)^{q+1}$};
\node at (30:1.3) {\tiny 1};
\node at (150:1.3) {\tiny 1};
\node at (210:1.4) {\tiny 123};
\node at (-30:1.4) {\tiny 123};
\node at (-90:1.4) {\tiny $(23)^q$};

\draw[->, gray!50]
	(0:0.4) arc (0:120:0.4);
\draw[->, gray!50]
	(0:0.6) arc (0:240:0.6);	

\node at (-0.35,0.2) {$R$};
\node at (-0.15,-0.5) {$R$};

\node at (0,-1.8) {$RE_{\triangle}^J1$};

\end{scope}


\begin{scope}[xshift=7 cm]

\node at (90:1) {\tiny $1^q$};
\node at (30:0.9) {\tiny 123};
\node at (150:0.9) {\tiny 123};
\node at (210:1) {\tiny 23};
\node at (-30:1) {\tiny 23};
\node at (-85:0.9) {\tiny $1^{q+1}$};

\node at (90:1.4) {\tiny $1^{q+1}$};
\node at (30:1.4) {\tiny 23};
\node at (150:1.4) {\tiny 23};
\node at (210:1.4) {\tiny 123};
\node at (-30:1.4) {\tiny 123};
\node at (-90:1.4) {\tiny $1^q$};

\draw[->, gray!50]
	(0:0.4) arc (0:120:0.4);
\draw[->, gray!50]
	(0:0.6) arc (0:240:0.6);	

\node at (-0.35,0.2) {$R$};
\node at (-0.15,-0.5) {$R$};

\node at (0,-1.8) {$RE_{\triangle}2$};

\end{scope}

\end{tikzpicture}
\caption{Rotation modifications $RE_{\triangle}^I1$, $RE_{\triangle}^J1$, $RE_{\triangle}2$ for $raa$.}
\label{tiling_rra3}
\end{figure}

The simple triangular subdivisions of the flip modification $FE_{\square}^R1$ are the flip modifications $FE_{\triangle}^J1$ and $FE_{\triangle}2$. They are illustrated by the middle and right of Figure \ref{tiling_rra3}, and are the same as the rotations in Figures \ref{tiling_hha9}, \ref{tiling_hha10} and \ref{tiling_hha13}. Again the flip is actually the same as the rotation (up to the free exchange of $r$ and $r^{-1}$), and we choose to use $R$ instead of $F$ in Figure \ref{tiling_rra3}. 
\end{proof}

   \bigskip
\subsection*{Acknowledgments}
This work was started in 2023 as an Undergraduate Research Project in the Tiling Studio of Zhejiang Normal University, guided by the third and the last author. After the other authors graduated in 2024, they continued the project until a full classification was obtained.  The bird prototile in Figure \ref{bird} was designed by Luoning Zhu, another junior student joining the team later specially for interesting drawings. The authors would like to thank her for several beautiful drawings during the project.

\end{document}